\documentclass[12pt, twoside]{article}

\usepackage{amsmath,amsthm,mathtools,mathrsfs}
\usepackage{amssymb}

\usepackage{enumitem}

\usepackage{graphicx}

\usepackage[T1]{fontenc}

\newtheorem{thm}{Theorem}[section]
\newtheorem{cor}[thm]{Corollary}
\newtheorem{lem}[thm]{Lemma}
\newtheorem{prop}[thm]{Proposition}

\theoremstyle{definition}
\newtheorem{dfn}[thm]{Definition}
\newtheorem{rem}[thm]{Remark}

\numberwithin{equation}{section}

\newcommand{\Rd}{\mathbb{R}^d}

\newcommand{\dhh}{\partial_h}
\newcommand{\dt}{\partial_t}

\allowdisplaybreaks
\begin{document}

\baselineskip=17pt


\title{Minimizing movements for quasilinear\\ Keller--Segel systems with nonlinear mobility\\ in weighted Wasserstein metrics}

\author{Kyogo Murai}

\date{2026}

\maketitle

\renewcommand{\thefootnote}{}



\renewcommand{\thefootnote}{\arabic{footnote}}
\setcounter{footnote}{0}


\begin{abstract}
We prove the global existence of weak solutions to quasilinear Keller--Segel systems 
with nonlinear mobility by minimizing movements (JKO scheme)
in the product space of the weighted Wasserstein space and $L^2$ space.
In particular, we newly show the global existence of weak solutions 
to the Keller--Segel system with the degenerate diffusion and the sub-linear sensitivity in the critical case.
The advantage of our approach is that we can connect the global existence of weak solutions 
to the Keller--Segel systems with the boundedness from below of a suitable functional.
While minimizing movements for Keller--Segel systems with linear mobility are adapted 
in the product space of the Wasserstein space and $L^2$ space,
due to the nonlinearity of mobility,
we need to use the weighted Wasserstein space instead of the Wasserstein space.
Moreover, since the mobility function is not Lipschitz,
we first find solutions to the Keller--Segel systems 
whose mobility is approximated by a Lipschitz function,
and then we establish additional uniform estimates and convergences
to derive solutions to the Keller--Segel systems.
\end{abstract}


\section{Introduction}

\quad We consider the following parabolic system:
\begin{align}\label{peks}
    \begin{cases}
      \partial_t u = \Delta u^p - \nabla\cdot(\chi u^\alpha \nabla v)
       &\mathrm{in}\ \Omega\times(0,\infty),\\
      \dt v = \Delta v - v + u 
       &\mathrm{in}\ \Omega\times(0,\infty),\\
      \nabla u\cdot \boldsymbol{n} = \nabla v\cdot \boldsymbol{n} = 0
       &\mathrm{on}\ \partial\Omega\times(0,\infty),\\
      u(\cdot,0) = u_0(\cdot),\ v(\cdot,0) = v_0(\cdot)
       &\mathrm{in}\ \Omega,
    \end{cases}
\end{align}
where 
$p\geq1,\ 0<\alpha<1,\ \chi>0,\ d\geq2$, $\boldsymbol{n}$ is the outer unit normal vector to $\partial\Omega$ 
and $\Omega$ is a bounded convex domain in $\Rd$ with smooth boundary.
In addition, 
$u_0 \in L^{p+1-\alpha}\cap\mathcal{P}(\Omega)$ is a nonnegative function
and $v_0 \in H^1(\Omega)$ is also a nonnegative function, where $\mathcal{P}(\Omega)$ is the set of Borel probability measures on $\Omega$.

\quad The Keller--Segel system is a model to describe an aggregation phenomenon of cellular slime molds with chemotaxis,
where we denote by $u$ the cell density and by $v$ the concentration of the chemical attractant.
We thus consider nonnegative solutions to \eqref{peks}.
In order to take into account a volume-filling effect, 
the exponent $\alpha \in (0,1)$ is introduced,
where the volume-filling effect is the phenomenon 
that the movement of cells is restricted by the presence of other cells.
There are various mathematical analyses about \eqref{peks} (see \cite{PH,W,TW,ISY,SW,Wi}).

\quad The purpose in this paper is to find global weak solutions to \eqref{peks} 
by regarding the system \eqref{peks} as a gradient flow of a suitable functional 
in a suitable metric space.
In more detail, we use the time discrete variational method called minimizing movements (JKO scheme \cite{JKO}).
When $\alpha=1$, the system \eqref{peks} can be seen as a gradient flow of the energy functional
\begin{align}\label{eq120}
   \tilde{E}(u,v) \coloneq 
   \frac{1}{\chi(p-1)}\int_{\Omega}u^p\, dx 
    - \int_{\Omega} uv\, dx
    + \frac{1}{2}\int_{\Omega} \left(|\nabla v|^2 + v^2\right)\, dx
\end{align}
in the product space of the Wasserstein space and $L^2$ space (see \cite{BL,BCKKLL,Mi,M}).
Here, Wassrestein space is the metric space of Borel probability measures with finite second moment $\mathcal{P}_2(\Omega)$
endowed with the Wasserstein distance
\begin{align}\label{eq81}
   \mathcal{W}_2(\mu_0,\mu_1)^2
   \coloneq \inf_{\gamma\in\Gamma(\mu_0,\mu_1)}\int_{\Omega\times\Omega} |x-y|^2\, d\gamma(x,y)\quad
   \mathrm{for}\ \mu_0, \mu_1 \in \mathcal{P}_2(\Omega),
\end{align}
where $\Gamma(\mu_0,\mu_1)$ is the set of $\gamma \in \mathcal{P}(\Omega\times\Omega)$ satisfying
\begin{align*}
   \gamma(A\times\Omega) = \mu_0(A), \
   \gamma(\Omega\times A) = \mu_1(A) \ \mathrm{for\ all\ Borel\ set}\ A \subset \Omega.
\end{align*}
In \cite{BL}, Blanchet and Lauren\c{c}ot showed the global exsitence of weak solutions to the
Keller--Segel system with $\alpha = 1, p=2-2/d$ and small initial data in $\Omega = \Rd\ (d\geq3)$ by minimizing movements.
In \cite{Mi} and \cite{M}, Mimura proved the global existence of weak solutions to the Keller--Segel system 
with $\alpha = 1$ and $p \geq 2-2/d$, adding the assumption of small initial data if $p=2-2/d$,
in bounded smooth domain $\Omega\subset\Rd\ (d\geq3)$ by minimizing movements. 
 
However, since the mobility function $u^\alpha\ (0<\alpha<1)$ is nonlinear, 
the system \eqref{peks} cannot be seen as a gradient flow of a corresponding energy functional
in the product space of the Wasserstein space and $L^2$ space.
We thus change the Wasserstein space to the weighted Wasserstein space (see \cite{DNS} and Section 2.2), 
which is the extension of the Wasserstein space in some sense.
Then we will see the system \eqref{peks} as a gradient flow of the energy functional
\begin{equation*}
   E(u,v) \coloneq 
   \frac{p}{\chi(p-\alpha)(p+1-\alpha)}\int_{\Omega}u^{p+1-\alpha}\, dx 
    - \int_{\Omega} uv\, dx
    + \frac{1}{2}\int_{\Omega} \left(|\nabla v|^2 + v^2\right)\, dx
\end{equation*}
in the product space of the weighted Wasserstein space and $L^2$ space.

\quad Minimizing movements with the Wasserstein space have been studied a lot,
thus the methods of a improved regularity of minimizers and convergences to weak formulation
are established (for instance \cite{BL,BCKKLL,MMS,CDFLS,CMV,JKO,Mi,M}).
On the other hand, minimizing movements with the weighted Wasserstein space 
are used in only a few papers (\cite{LMS,Z,Zi}).
In \cite{LMS} and \cite{Z}, they deal with some fourth-order partial differential equations 
like the Cahn--Hilliard type equation and mainly with the Lipschitz mobility.
While the mobility function $u^\alpha$ in \eqref{peks} is not Lipschitz,
we show that minimizing movements with the weighted Wasserstein space can be also used for
the type of the equations \eqref{peks}.
Set $X \coloneq (L^{p+1-\alpha}\cap\mathcal{P}(\Omega))\times H^1(\Omega)$.

\begin{thm}\label{thm1.1}
   Let $d\geq2,\ \alpha \in (0,1),\ 1+\alpha-2/d < p \leq 1+\alpha,\ \chi>0$ 
   and $(u_0,v_0) \in X$ be a pair of nonnegative functions.
   Then for all $T>0$, there exists a nonnegative weak solution $(u,v)$ to \eqref{peks}
   on the time interval $[0,T]$ satisfying
   \begin{align*}
      &\bullet u \in L^\infty((0,T);L^{p+1-\alpha}(\Omega)),\ u^{\frac{p+1-\alpha}{2}} \in L^2((0,T);H^1(\Omega)),\\
      &\bullet \|u(t)\|_{L^1(\Omega)} = 1\quad \mathrm{for}\ t \in [0,T],\\
      &\bullet v \in L^\infty((0,T);H^1(\Omega))\cap L^2((0,T);H^2(\Omega))\cap C^{\frac{1}{2}}([0,T];L^2(\Omega)),\\
      &\bullet \lim_{t\to0}W_m(u(t),u_0) = 0\ \mathrm{and}\ \lim_{t\to0}\|v(t) - v_0\|_{L^2(\Omega)} = 0,
   \end{align*}
   where $W_m$ is the weighted Wasserstein distance (see Definition \ref{dfn2.4})
   and $m(u) = u^\alpha$, moreover
   \begin{align*}
      &\int_0^T\int_{\Omega} (\nabla u^p - \chi u^\alpha\nabla v)\cdot\nabla\varphi\, dx\,dt 
       = \int_{\Omega} (u_0 - u(\cdot,T))\varphi\, dx,\\
      &\int_0^T\int_{\Omega} \left[\nabla v\cdot\nabla\zeta + v\zeta - u\zeta\right]\, dx\, dt
       = \int_{\Omega} (v_0 - v(\cdot,T))\zeta\, dx,
   \end{align*}
   for all $\varphi \in C^\infty(\overline{\Omega})$ 
   with $\nabla \varphi\cdot\boldsymbol{n} = 0$ on $\partial\Omega$
   and $\zeta \in H^1(\Omega)$.
\end{thm}

\begin{rem}
   In Theorem \ref{thm1.1}, we have the condition
   $$p > 1 + \alpha - \frac{2}{d}.$$
   We can see this exponent in the point of view of scalling.
   Let $(u,v)$ satisfies the simplified system:
   \begin{align}\label{eq90}
      \begin{cases}
         \dt u = \Delta u^p - \nabla\cdot(u^\alpha\nabla v) &\mathrm{in}\ \Rd\times(0,\infty),\\
         \dt v = \Delta v + u &\mathrm{in}\ \Rd\times(0,\infty),
      \end{cases}
   \end{align}
   then $(u_\lambda,v_\lambda)$, 
   where $u_\lambda(x,t) = \lambda^{\frac{2}{1+\alpha-p}}u(\lambda x,\lambda^2t)$ and
   $v_\lambda = \lambda^{\frac{2(p-\alpha)}{1+\alpha-p}}v(\lambda x,\lambda^2t)$ 
   for $(x,t)\in\Rd\times(0,\infty),\ \lambda>0$,
   also satisfies \eqref{eq90}. 
   By the first equation of \eqref{eq90}, 
   the $L^1$ norm of $u_\lambda$ is preserved for $t\in(0,\infty)$.
   Thus we focus on the $L^1$ norm, then by the change of variables, we have
   \begin{align*}
      \|u_\lambda(\cdot,t)\|_{L^1(\Rd)}
      &= \int_{\Rd} \lambda^{\frac{2}{1+\alpha-p}}u(\lambda x,\lambda^2t)\, dx\\
      &= \int_{\Rd} \lambda^{\frac{2}{1+\alpha-p}-d}u(x,\lambda^2t)\, dx\\
      &= \lambda^{\frac{2}{1+\alpha-p}-d}\|u(\cdot,\lambda^2t)\|_{L^1(\Rd)}
      \quad \mathrm{for\ all}\ \lambda>0,\ t>0.
   \end{align*}
   Hence the exponent that the $L^1$ norm is invariant by scalling is 
   \begin{align*}
      \frac{2}{1+\alpha-p} - d = 0
      \Leftrightarrow
      p=1+\alpha-\frac{2}{d}.
   \end{align*}
   On the other hand, 
   when $\alpha=1$, the case $p>2-2/d$ is called sub-critical case
   and it is known that the Keller--Segel system has global weak solutions in that case,
   which is proved by various ways including minimizing movements (\cite{IY,BL,Mi,M}).
   In particular, the proof by minimizing movements (\cite{BL,Mi,M}) implies that
   the global existence of weak solutions to the Keller--Segel system with $\alpha=1$
   is related to the boundedness from below of the energy functional $\tilde{E}$ in \eqref{eq120}.
   When $0<\alpha<1$, 
   it is shown that system \eqref{peks} has global weak solutions if $p>1+\alpha-2/d$ in \cite{TW} and \cite{ISY},
   where minimizing movements are not used.
   Moreover if $p\geq1$ and $\alpha\geq1$ satisfy the condition $p < 1 + \alpha - 2/d$,
   there exists a finite time blow-up solution of the Keller--Segel system (\cite{HIY}).
   Hence from these facts, 
   we may derive the proper condition for global existence of weak solutions to \eqref{peks} by minimizing movements.
\end{rem}

\quad In the critical case $p=1+\alpha-2/d$, 
it is known that the Keller--Segel system with the non-degenerate diffusion for $0<\alpha<1$
has a global solution by assuming small initial data (see Remark \ref{rem1.4}).
However, it is open that the Keller--Segel system with the degenerate diffusion and the sub-linear sensitivity 
such as \eqref{peks} has a global weak solution.
The following theorem gives the positive answer to the above open problem, that is, 
if we assume that $\chi>0$ is sufficiently small, 
which is equivalent to the smallness of the $L^1$ norm of the initial data,
then there exist global weak solutions to \eqref{peks}. 

\begin{thm}\label{thm1.4}
   Let $d\geq3,\ 0<\alpha<1,\ p = 1+\alpha-2/d$ and $(u_0,v_0) \in X$ be a pair of nonnegative functions.
   If $\chi>0$ is small enough then the same statement in Theorem \ref{thm1.1} holds.
\end{thm}

\begin{rem}\label{rem1.4}
   In \cite{Wi}, it is shown that if the $L^1$ norm of the initial data is small enough
   then there exists a global solution to the Keller--Segel system:
   \begin{align}\label{peks_g}
      \begin{cases}
      \dt u = \nabla\cdot(D(u)\nabla u) - \nabla\cdot(S(u)\nabla v)\quad &\mathrm{in}\ \Omega\times(0,\infty),\\
      \dt v = \Delta v - v + u &\mathrm{in}\ \Omega\times(0,\infty),
      \end{cases}
   \end{align}
   where $D \in C^2([0,\infty))$ and $S \in C^2([0,\infty))$ such that
   $D>0$ on $[0,\infty)$, $S(0) = 0 < S(s)$ for all $s>0$ and
   \begin{align*}
      \frac{S(s)}{D(s)} \leq K_{SD}s^{\frac{2}{n}}\quad \mathrm{for\ all}\ s>0
   \end{align*}
   for some constant $K_{SD}>0$.
   This result includes the critical case $p=1+\alpha-2/d$,
   but the first equation of \eqref{peks_g} needs to be non-degenerate ($D>0$ on $[0,\infty)$).
   On the other hand, 
   we treat the degenerate case $D(u) = p u^{p-1}$.
   Thus Theorem \ref{thm1.4} shows the global existence of weak solutions to 
   the Keller--Segel system with the degenerate diffusion in the critical case $p=1+\alpha-2/d$ ($0<\alpha<1$)
   by assuming small initial data.
\end{rem}

\begin{rem}
   The assumptions, $p>1+\alpha-2/d$ in Theorem \ref{thm1.1} and 
   $p=1+\alpha-2/d$ with small $\chi>0$ in Theorem \ref{thm1.4},
   are essentially used to get the boundedness from below of the energy functional $E$ (see Section 3).
   Hence our approach implies that 
   the global existence of weak solutions to \eqref{peks}
   is related to the boundedness from below of $E$,
   and has an advantage in that point.
   Indeed, the Keller--Segel systems with the degenerate diffusion in the sub-critical case 
   are considerd in \cite{TW,ISY},
   however that relationship cannot be seen.
   In addition, the similar relationship can be seen in \cite{Wi},
   in particular for the Keller--Segel system with the non-degenerate diffusion \eqref{peks_g},
   but we deal with the degenerate diffusion case \eqref{peks}.
   In other words,
   our approach gives the relationship between the global existence of weak solutions to 
   \eqref{peks}, which has the degenerate diffusion and the sub-linear sensitivity, 
   and the boundedness from below of the functional $E$
   in both the sub-critical case and the critical one.
\end{rem}

\begin{rem}
   Theorem \ref{thm1.1} and Theorem \ref{thm1.4} require 
   the initial data $u_0 \in L^{p+1-\alpha}(\Omega)$
   and $v_0 \in H^1(\Omega)$.
   On the other hand, in \cite{ISY},
   the initial data $u_0$ and $v_0$ should belong to $L^\infty(\Omega)$ and $W^{1,\infty}(\Omega)$
   respectivily.
   In addition, in \cite{Wi},
   the initial data $u_0$ and $v_0$ must be in $W^{1,\infty}(\Omega)$.
   Thus our results assume the lower regularity of the initial data to get the global weak solutions to \eqref{peks}.
\end{rem}

\begin{rem}
   In Theorems \ref{thm1.1} and \ref{thm1.4}, 
   by a little modification of the proof, 
   it also holds that for all $[s_1,s_2] \subset [0,T]$, 
   \begin{align*}
      &\int_{s_1}^{s_2}\int_{\Omega} (\nabla u^p - \chi u^\alpha\nabla v)\cdot\nabla\varphi\, dx\,dt 
       = \int_{\Omega} (u(\cdot,s_1) - u(\cdot,s_2))\varphi\, dx,\\
      &\int_{s_1}^{s_2}\int_{\Omega} \left[\nabla v\cdot\nabla\zeta + v\zeta - u\zeta\right]\, dx\, dt
       = \int_{\Omega} (v(\cdot,s_1) - v(\cdot,s_2))\zeta\, dx.
   \end{align*}
\end{rem}

\quad Due to lack of the Lipschitz property of the mobility function $u^\alpha\ (0<\alpha<1)$,
it is complicated for us to consider the Euler--Lagrange equation.
When $\alpha=1$, the mobility function $u$ is a smooth and Lipschitz function.
Then the Wasserstein distance has a good property that 
the perturbation $\mu_a$ of a measure $\mu$ can be represented by the push-forward measure of $\mu$
by a map $T_a : \Rd\ni x \mapsto x + a\xi \in \Rd$ for $a>0$ and $\xi \in C_c^\infty(\Rd;\Rd)$,
that is, $\mu_a = {T_a}_\#\mu$,
where ${T_a}_\#\mu$ is defined by 
${T_a}_\#\mu(A) = \mu(T_a^{-1}(A))$ for all Borel set $A \subset \Rd$.
On the other hand, when $0<\alpha<1$, the mobility function $u^\alpha$ is smooth but not Lipschitz
because of the singularity of its derivative at $u=0$.
Moreover, since it is not known that the weighted Wasserstein distance has a representation 
such as \eqref{eq81}, we cannot use the same way to consider the perturbation.

\quad Then we need to apply another way called the flow interchange lemma (see Section 4).
The flow interchange lemma is introduced in \cite{MMS} for the Wasserstein space, 
and Lisini, Matthes and Savar\'e apply it for the weighted Wasserstein space in \cite{LMS}.
But, since the Lipschitz property of the mobility function is needed for thier method,
the flow interchange lemma does not work directly in the case of the mobility $u^\alpha$.
To overcome this problem, 
we approximate the function $u^\alpha$ by a $C^\infty$ and Lipschitz function 
$m_{\varepsilon} : [0,\infty)\to[0,\infty)$ for $\varepsilon \in (0,1)$:
\begin{equation*}
   m_{\varepsilon}(r) \coloneq (r+\varepsilon)^\alpha,\
   m_\varepsilon^\prime(r) = \frac{\alpha}{(r+\varepsilon)^{1-\alpha}} \leq \frac{\alpha}{\varepsilon^{1-\alpha}}\quad
   \mathrm{for}\ r \in [0,\infty),
\end{equation*}
that is, we first consider the system: 
\begin{align}\label{peks_e}
   \begin{cases}
      \partial_t u = \nabla\cdot m_\varepsilon(u)\left(\frac{p}{p-\alpha}\nabla u^{p-\alpha} - \chi \nabla v\right)
       &\mathrm{in}\ \Omega\times(0,\infty),\\
      \dt v = \Delta v - v + u 
       &\mathrm{in}\ \Omega\times(0,\infty),\\
      \nabla u\cdot \boldsymbol{n} = \nabla v\cdot \boldsymbol{n} = 0
       &\mathrm{on}\ \partial\Omega\times(0,\infty),\\
      u(\cdot,0) = u_0(\cdot),\ v(\cdot,0) = v_0(\cdot)
       &\mathrm{in}\ \Omega.
    \end{cases}\tag*{(1.1)$_\varepsilon$}
\end{align}

\quad Thanks to this approximation, we can get the solutions to \ref{peks_e},
and then we need to obtain uniform estimates with respect to $\varepsilon$ 
and convergences as $\varepsilon\to0$.
The key point for uniform estimates is the boundedness of the functional
$\boldsymbol{U}_{\varepsilon} : L^{p+1-\alpha}\cap\mathcal{P}(\Omega) \to \mathbb{R}$ defined by
\begin{equation*}
   \boldsymbol{U}_{\varepsilon}(u) \coloneq \int_{\Omega} U_{\varepsilon}(u(x))\, dx,
\end{equation*}
where 
$U_{\varepsilon} : [0,\infty)\to\mathbb{R}$ satisfies $U_{\varepsilon}^{\prime\prime}(r)m_{\varepsilon}(r)=1$ and 
$U_\varepsilon^\prime(0) = U_\varepsilon(0) = 0$ (see Lemma \ref{lem2.11}).
On the other hand, the key point for the convergences, 
in particular the pointwise convergence for $t$ weakly in $L^1\cap L^{p+1-\alpha}(\Omega)$,
is the lower semicontinuity of the weighted Wasserstein distance (see Lemma \ref{lem2.7} and Lemma \ref{lem5.8}).
In order to get the convergence, we use the refined Ascoli--Arzel\`a theorem (\cite[Proposition 3.3.1]{AGS}),
and in more detail, we need the estimate like the equi-continuity
with respect to the weighted Wasserstein distance:
\begin{equation*}
   W_m(u_\tau(t),u_\tau(s)) \leq C(\sqrt{|t-s|} + \sqrt{\tau})\quad \mathrm{for}\ t,s \in [0,T],
\end{equation*}
where $m(u) = u^\alpha$, $u_\tau$ is a pointwise constant function (Definition \ref{dfn5.1})
and $C>0$ is a constant independent of $\varepsilon$ and $\tau$.
However, we only obtain such estimate replaced $m$ by $m_\varepsilon$,
that is, the distance depends on $\varepsilon$ (see Lemma \ref{lem5.3}).
Thus combining the lower semicontinuity of the weighted Wasserstein distance 
with the above estimate, we prove the pointwise convergence for $t$ weakly in $L^1\cap L^{p+1-\alpha}(\Omega)$.
This approach is original and may be applied for other cases 
if we use the lower semicontinuity of the weighted Wasserstein distance similarly.

\quad Finally, we remark that the way of the approximation of $u^\alpha$ by $(u+\varepsilon)^\alpha$
is important in the point of view of minimizing movements with the weighted Wasserstein space.
In \cite{LMS}, when they investigate the fourth-order partial differential equations,
which do not have Lipschitz mobility,
they also approximate the mobility function by another way 
such as $u^\alpha$ by $(u+\varepsilon)^\alpha - \varepsilon^\alpha$.
However, thier approximation requires that initial data $u_0$ belongs to $L^2(\Omega)$ 
in order to obtain the uniform estimate for $\boldsymbol{U}_\varepsilon$.
If $p< 1+\alpha$ then our initial data $u_0$ does not belong to $L^2(\Omega)$,
thus we cannot use thier approximation in that case.
On the other hand, in order to derive uniform estimate for $\boldsymbol{U}_\varepsilon$, 
our approximation requires that initial data $u_0$ belongs to $L^{2-\alpha}(\Omega)$,
which is naturally satisfied, hence it is effective to use our apporoximation
in our case.

\quad This paper is organized as follows.
In Section 2, we recall the definition of the weighted Wasserstein distance
and some properties of it introduced in \cite{DNS}.
Section 3 is devoted to the time discrete variational scheme.
In Section 4, we deal with the flow interchange lemma and prepare some lemmas to adapt it.
Fundamentally, we refer to the method in \cite{LMS},
but our functions (minimizers in Section 3) have a lower regularity than thier ones, 
so we derive a suitable regurality of minimizers to obtain the Euler--Lagrange equation (Lemma \ref{lem4.9}).
In Section 5, we consider uniform estimates with respect to $\tau$,
which yield that the time discrete solution $(u_\tau,v_\tau)$ (Definition \ref{dfn5.1})
converges to a weak solution to \ref{peks_e}.
Then in Section 6, we also establish uniform estimates with respect to $\varepsilon$,
which is guaranteed by the uniform estimate of $\boldsymbol{U}_\varepsilon$ (Lemma \ref{lem2.11}).
In addition, the pointwise convergence for $t$ weakly in $L^1\cap L^{p+1-\alpha}(\Omega)$ (Lemma \ref{lem5.8})
plays an important role in this section.
Using these estimates and convergences, we prove Theorem \ref{thm1.1} ($1+\alpha-2/d<p\leq 1+\alpha$) 
and Theorem \ref{thm1.4} ($p=1+\alpha-2/d$ and small $\chi>0$).


\section{Preliminary}

\subsection{Notations}

\begin{align*}
   &\mathcal{P}(\Rd) = \{\mu:\mu\ \mathrm{is\ a\ Borel\ probability\ measure\ in}\ \Rd\}\\
   &\mathcal{P}(\Omega) = \{\mu \in \mathcal{P}(\Rd):\mu(\Omega) = 1,\ \mu(\Rd\setminus\Omega) = 0\}\\
   &\mathcal{M}_{loc}^+(\Rd) = \{\mu:\mu\ \mathrm{is\ a\ nonnegative\ Radon\ measure\ in}\ \Rd\}\\
   &\mathcal{M}^+(\Omega) = \{\mu \in \mathcal{M}_{loc}^+(\Rd):\mu(\Rd\setminus\Omega) = 0\}\\
   &\mathcal{M}_{loc}(\Rd;\Rd) = \{\boldsymbol{\nu}:\boldsymbol{\nu}\ \mathrm{is\ a}\ \Rd\mbox{-}\mathrm{valued\ Radon\ measure\ in}\ \Rd\}\\
   &\mathcal{M}(\Omega;\Rd) = \{\boldsymbol{\nu} \in \mathcal{M}_{loc}(\Rd;\Rd):\boldsymbol{\nu}(\Rd\setminus\Omega) = 0\}\\
   &C_c(\Rd) = \{f \in C(\Rd):\mathrm{supp}(f)\ \mathrm{is\ compact\ in}\ \Rd\}\\
   &B_R = \{x \in \Rd:|x|<R\}
\end{align*}
Note that $\mathcal{P}(\Rd) \subset \mathcal{M}_{loc}^+(\Rd)$ and 
$\mathcal{P}(\Omega) \subset \mathcal{M}^+(\Omega)$.
By the Riesz representation theorem, 
$\mathcal{M}_{loc}^+(\Rd)$ (resp. $\mathcal{M}_{loc}(\Rd;\Rd)$) 
can be identified with the dual space of $C_c(\Rd)$ (resp. $C_c(\Rd;\Rd)$).

\subsection{Weighted Wasserstein distance}

\quad We recall the weighted Wasserstein distance 
which is a distance on the space of nonnegative Radon measures $\mathcal{M}^+(\Omega)$
and introduced in \cite{DNS}.
First, we recall the continuity equation for Radon measures.

\begin{dfn}({\cite[Definition 4.2]{DNS}}, Solutions of the continuity equation)\label{dfn2.2}
   Let $\mu_0, \mu_1 \in \mathcal{M}^+(\Omega)$.
   We denote by $CE(0,1;\mu_0\to\mu_1)$ the set of pairs of time dependent measures 
   $\{\mu_t\}_{t\in[0,1]} \subset \mathcal{M}^+(\Omega)$ and 
   $\{\boldsymbol{\nu}_t\}_{t\in[0,1]} \subset \mathcal{M}(\Omega;\Rd)$
   such that
   \begin{enumerate}
      \item $t\mapsto \mu_t$ is weakly* continuous in $\mathcal{M}_{loc}^+(\Rd)$
      :\ for all $f \in C_c(\Rd)$,\ 
      $$\int_{\Rd} f(x)\, d\mu_t(x)\ \mbox{is continuous with respect to }t \in [0,1],$$
      \item $\{\boldsymbol{\nu}_t\}_{t\in[0,1]}$ is a Borel measureable family with
      $$\int_0^1\int_{\Omega}\, d\boldsymbol{\nu}_t(x)\,dt <\infty,$$
      \item $(\mu_t,\boldsymbol{\nu}_t)$ satisfies the continuity equation in the weak sense
      :\ for all $\psi \in C_c^1(\mathbb{R}^d\times(0,1))$,\ 
      $$\int_0^1\int_{\mathbb{R}^d}\dt\psi(x,t)\, d\mu_t(x)\,dt + \int_0^1\int_{\mathbb{R}^d}\nabla\psi(x,t)\cdot d\boldsymbol{\nu}_t(x)\,dt = 0,$$
      \item $\mu_t|_{t=0} = \mu_0,\ \mu_t|_{t=1} = \mu_1$:\ for all $f \in C_c(\Rd)$,
      $$\lim_{t\to i}\int_{\Rd} f(x)\, d\mu_t(x) = \int_{\Rd} f(x)\, d\mu_i(x) \quad i=0,1.$$
   \end{enumerate}
\end{dfn}

\quad Next, we define the functional in the space of Radon measures, 
which is important for the 
definition of the weighted Wasserstein distance.

\begin{dfn}[{\cite[Section 3]{DNS}}, The action functional]\label{dfn2.3}
   Let $\mu \in \mathcal{M}^+(\Omega)$ and $\boldsymbol{\nu} \in \mathcal{M}(\Omega;\Rd)$.
   Let $m : [0,\infty) \to [0,\infty)$ be a concave function.
   Then we define the action functional 
   $\Psi_{\Omega} : \mathcal{M}^+(\Omega)\times\mathcal{M}(\Omega;\Rd) \to [0,\infty]$ by
   \begin{align*}
      \Psi_{\Omega}(\mu,\boldsymbol{\nu}) \coloneq 
      \begin{dcases}
         \int_{\Omega} \frac{|\boldsymbol{w}|^2}{m(\rho)}\, dx &\mathrm{if}\ \boldsymbol{\nu}^\perp = 0,\\
         \infty &\mathrm{if}\ \boldsymbol{\nu}^\perp \ne 0,
      \end{dcases} 
   \end{align*}
   where $\mu = \rho\mathscr{L}^d + \mu^\perp$ and $\boldsymbol{\nu} = \boldsymbol{w}\mathscr{L}^d + \boldsymbol{\nu}^\perp$ 
   are Lebesgue decomposition with respect to Lebesgue's measure $\mathscr{L}^d$,
   that is, $\rho \in L^1(\Omega),\ \boldsymbol{w} \in L^1(\Omega;\Rd)$ and 
   $\mu^\perp$ (resp. $\boldsymbol{\nu}^\perp$) is a singular part of $\mu$ (resp. $\boldsymbol{\nu}$).
\end{dfn}

\begin{dfn}[{\cite[Definition 5.1]{DNS}}, Weighted Wasserstein distance]\label{dfn2.4}
   Let $\mu_0, \mu_1 \in \mathcal{M}^+(\Omega)$ and 
   $m : [0,\infty) \to [0,\infty)$ be a concave function. 
   Then weighted Wasserstein distance between $\mu_0$ and $\mu_1$ is defined as
   \begin{align}\label{eq42}
      W_{m,\Omega}(\mu_0,\mu_1)^2 
      &\coloneq \inf\left[\int_0^1 \Psi_{\Omega}(\mu_t,\boldsymbol{\nu}_t)\, dt : (\mu_t,\boldsymbol{\nu}_t) \in CE(0,1;\mu_0\to\mu_1)\right]\\
      &=
      \begin{dcases}
         \inf\left[\int_0^1\int_{\Omega} \frac{|\boldsymbol{w}_t(x)|^2}{m(\rho_t(x))}\, dx\, dt : (\mu_t,\boldsymbol{\nu}_t) \in CE(0,1;\mu_0\to\mu_1)\right]\ &\mathrm{if}\ \boldsymbol{\nu}^\perp=0,\\
         \infty\ &\mathrm{if}\ \boldsymbol{\nu}^\perp\ne0.\notag
      \end{dcases}
   \end{align}
   We usually omit to write $\Omega$, then $W_m(\mu_0,\mu_1)$,
   but if we emphasize the domain $\Omega$, we write $W_{m,\Omega}(\mu_0,\mu_1)$.
\end{dfn}

\quad Next, we collect some properties of the weighted Wasserstein distance 
(\cite[Theorems 5.5, 5.4, 5.6, 2.3]{DNS}, \cite[Lemma 8.1.10]{AGS}).

\begin{prop}[Distance and topology]\label{prop2.5}
   The functional $W_m$ is a (pseudo) distance on $\mathcal{M}_{loc}^+(\Rd)$ 
   which induces a stronger topology than weak* one in $\mathcal{M}_{loc}^+(\Rd)$.
   Moreover bounded sets with respect to $W_m$ are weakly* relatively compact in $\mathcal{M}_{loc}^+(\Rd)$.
\end{prop}

\begin{lem}[Existence of minimizers]\label{lem2.6}
   Let $\mu_0,\mu_1 \in \mathcal{M}^+(\Omega)$.
   Whenever the infimun in \eqref{eq42} is a finite value, 
   it is attained by a curve $(\mu_t,\boldsymbol{\nu}_t) \in CE(0,1;\mu_0\to\mu_1)$.
\end{lem}

\begin{lem}[Lower semicontinuity]\label{lem2.7}
   Let $\{\Omega_n\}_n$ be a sequence of bounded domain converging to 
   a bounded domain $\Omega$, that is, 
   $\mathscr{L}^d|_{\Omega_n} \rightharpoonup \mathscr{L}^d|_{\Omega}$ weakly* in $\mathcal{M}_{loc}^+(\Rd)$ as $n\to\infty$:
   \begin{align*}
      \int_{\Rd} f(x)\, d\mathscr{L}^d|_{\Omega_n}(x) \to \int_{\Rd} f(x)\, d\mathscr{L}^d|_{\Omega}(x)
      \quad \mathrm{as}\ n\to\infty,
   \end{align*}
   for all $f\in C_c(\Rd)$.
   Moreover, let 
   series of functions $\{m_n\}_n$ be monotonically decreasing with respect to $n$ 
   and pointwise converging to $m$ as $n\to\infty$.
   Then, 
   the map $(\mu_0,\mu_1) \mapsto W_m(\mu_0,\mu_1)$ is lower semicontinuous with respect to
   weak* convergence in $\mathcal{M}_{loc}^+(\Rd)$: 
   if sequences of measures $\{\mu_0^n\}_n$ and $\{\mu_1^n\}_n$ satisfy
   $\mu_0^n \rightharpoonup \mu_0,\ \mu_1^n \rightharpoonup \mu_1$ 
   weakly* in $\mathcal{M}_{loc}^+(\Rd)$ as $n\to\infty$
   then
   \begin{equation}
      W_{m,\Omega}(\mu_0,\mu_1) \leq \liminf_{n\to\infty}W_{m_n,\Omega_n}(\mu_0^n,\mu_1^n).
   \end{equation}
\end{lem}

\begin{lem}[Convolution]\label{lem2.8}
   Let $\psi : (0,\infty)\times\Rd \to [0,\infty)$ be a convex 
   and lower semicontinuous function satisfying $\psi(\cdot,0) = 0$.
   Let $\mu \in \mathcal{M}^+(\Omega)$ and 
   $\boldsymbol{\nu} \in \mathcal{M}(\Omega;\Rd)$ 
   be such that $\mu = \rho\mathscr{L}^d + \mu^\perp$ and 
   $\boldsymbol{\nu} = \boldsymbol{w}\mathscr{L}^d + \boldsymbol{\nu}^{\perp}$ 
   where $\rho \in L^1(\Omega)$,
   $\boldsymbol{w} \in L^1(\Omega;\Rd)$ and 
   $\mu^\perp$,
   $\boldsymbol{\nu}^{\perp}$ are singular parts of Lebesgue decomposition.
   We define $\Psi(\mu,\boldsymbol{\nu}) \coloneq 
    \int_{\Rd} \psi(\rho(x),\boldsymbol{w}(x))\, dx = \int_{\Omega}\psi(\rho(x),\boldsymbol{w}(x))\, dx$ 
   and let $k \in C_c^\infty(\Rd)$ be a nonnegative convolution kernel 
   with $\int_{\Rd}k(x)\, dx=1$.
   Then 
   \begin{equation}
      \Psi(\mu*k,\boldsymbol{\nu}*k) \leq \Psi(\mu,\boldsymbol{\nu}),
   \end{equation}
   where $\mu*k$ (resp. $\boldsymbol{\nu}*k$) is the measure defined by the (density) function
   \begin{equation*}
      x \mapsto \mu*k(x) \coloneq \int_{\Rd} k(x-y)\, d\mu(y)\
      \left(\mbox{resp.}\ \boldsymbol{\nu}*k \coloneq \int_{\Rd} k(x-y)\, d\boldsymbol{\nu}(y)\right).
   \end{equation*}
\end{lem}

\quad Next lemma implies that 
the minimizer of the weighted Wasserstein distance can be approximated by smooth denisity functions.
In \cite[Lemma 3.6]{CLSS}, 
they showed the existence of smooth functions ($\rho_n$ and $\boldsymbol{w}_n$ in next lemma)
approximating the minimizer of the weighted Wasserstein distance,
however we will prove that 
the weighted Wasserstein distance can be approximated by another smooth functions ($\rho_n$ and $\phi_n$ in next lemma).
Note that, in \cite[Proposition 2.2]{LMS}, they stated similar approximation lemma,
but they did not give the proof, so we give rigorous one.
Thanks to this approximation, we can do calculations simply and adapt the flow interchange lemma (see Section 4).

\begin{lem}\label{lem2.10}
   Let $m \in C^\infty(0,\infty)$ be a positive concave function
   such that $\inf_{r\in(0,\infty)}m(r)>0$.
   Let $\mu_0, \mu_1 \in L^q\cap\mathcal{P}(\Omega)$ for $q\in[1,\infty)$
   with $W_m(\mu_0,\mu_1)<\infty$.
   Then for every decreasing sequence of smooth bounded sets $\Omega_n$ converging to $\Omega$ as $n\to\infty$,
   that is, $\Omega_{n+1} \subset \Omega_n$ for $n\in \mathbb{N}$ 
   and $\mathscr{L}^d|_{\Omega_n} \rightharpoonup \mathscr{L}^d|_{\Omega}$ weakly* in $\mathcal{M}_{loc}^+(\Rd)$ as $n\to\infty$,
   there exists a vanishing sequence $\{b_n\}_n$ such that $\Omega_{[b_n]} \subset \Omega_n$, 
   where $\Omega_{[b_n]} \coloneq \Omega + b_nB_1 = \{x + b_ny:x\in \Omega, y \in B_1\}$, and
   there exist a nonnegative function $\rho_n \in C^\infty(\overline{\Omega}_n\times[0,1])$ 
   and a function $\phi_n \in C^\infty(\overline{\Omega}_n\times[0,1])$
   with $\nabla\phi_n\cdot\boldsymbol{n} = 0$ on $\partial\Omega_n\times[0,1]$,
   satisfying the following:
   \begin{enumerate}
      \item $\|\rho_n(t)\|_{L^1(\Omega_n)} = 1$ {\rm for}\ $n \in \mathbb{N},\ t \in [0,1]$,
      \item $\|\rho_n(0)-\mu_0\|_{L^q(\Rd)} + \|\rho_n(1)-\mu_1\|_{L^q(\Rd)} \to 0$ {\rm as}\ $n\to\infty$,
      \item $(\rho_n,m(\rho_n)\nabla\phi_n)$ satisfies the continuity equation:
      $$\dt\rho_n(x,t)=-\nabla\cdot(m(\rho_n(x,t))\nabla\phi_n(x,t))\quad \mathrm{for\ all}\ (x,t)\in\Omega_n\times(0,1)$$
      and moreover 
      $$W_m(\mu_0,\mu_1)^2 = \lim_{n\to\infty}\int_0^1\int_{\Omega_n} m(\rho_n(x,t))|\nabla\phi_n(x,t)|^2\, dx\,dt.$$
   \end{enumerate}
\end{lem}

\begin{proof}
   Let $\{\Omega_n\}_n$ be a decreasing sequence of smooth convex bounded sets converging to $\Omega$ as $n\to\infty$.
   Since it is decreasing and bounded, we can find a sequence $\{b_n\}_n$ such that
   $b_n \to 0$ as $n\to\infty$ and $\Omega_{[b_n]} \subset \Omega_n$.
   Let $(\mu_t,\boldsymbol{\nu}_t) \in CE(0,1;\mu_0\to\mu_1)$ 
   be a minimizer of $W_m(\mu_0,\mu_1)$ (Lemma \ref{lem2.6}).
   Let us extend $(\mu_t,\boldsymbol{\nu}_t)$ outside the interval $[0,1]$ by setting 
   $\boldsymbol{\nu}_t=0$ if $t<0$ or $t>1$, 
   and $\mu_t=\mu_0$ if $t<0$,\ $\mu_t=\mu_1$ if $t>1$.
   Then $(\mu_t,\boldsymbol{\nu}_t)$ still satisfies the continuity equation.
   For $\{b_n\}_n$ as the above, 
   let $k_n \in C_c^\infty(\Rd)$ be a nonnegative mollifier such that
   $\mathrm{supp}(k_n) \subset B_{b_n}$.
   Then define measures
   $\mu_t*k_n$ and  
   $\boldsymbol{\nu}_t*k_n$
   which have spatial smooth densities 
   \begin{align*}
      &\tilde{\mu}_{n,t}(x) 
       \coloneq \int_{\Rd} k_n(x-y)\, d\mu_t(y),\\
      &\tilde{\boldsymbol{\nu}}_{n,t}(x) 
       \coloneq \int_{\Rd} k_n(x-y)\, d\boldsymbol{\nu}_t(y).
   \end{align*}
   Note that $\mathrm{supp}(\tilde{\mu}_{n,t}),\ \mathrm{supp}(\tilde{\boldsymbol{\nu}}_{n,t}) \subset \Omega_n$.
   Moreover let $h_{\frac{1}{n}} \in C_c^\infty(\mathbb{R})$ be a nonnegative mollifier in $\mathbb{R}$
   such that $\mathrm{supp}(h_{\frac{1}{n}}) \subset [-\frac{1}{n},\frac{1}{n}]$ 
   and define functions
   \begin{align*}
      &\tilde{\rho}_n(x,t) \coloneq \int_{\mathbb{R}} \tilde{\mu}_{n,z}(x) h_{\frac{1}{n}}(t-z)\, dz
      \quad \mathrm{for}\ (x,t) \in \Rd\times\mathbb{R},\\
      &\tilde{\boldsymbol{w}}_n(x,t) \coloneq \int_{\mathbb{R}} \tilde{\boldsymbol{\nu}}_{n,z}(x) h_{\frac{1}{n}}(t-z)\, dz
      \quad \mathrm{for}\ (x,t) \in \Rd\times\mathbb{R}.
   \end{align*}
   Then $\tilde{\rho}_n \in C^\infty(\Rd\times\mathbb{R})$ and 
   $\tilde{\boldsymbol{w}}_n \in C^\infty(\Rd\times\mathbb{R};\Rd)$,
   in addition, their spatial supports are included in $\Omega_n$.
   Notice that $\tilde{\rho}_n(\cdot,-\frac{1}{n}) = \tilde{\mu}_{n,0}$ 
   and $\tilde{\rho}_n(\cdot,1+\frac{1}{n}) = \tilde{\mu}_{n,1}$.
   Indeed, for all $x \in \Rd$, we have
   \begin{align*}
      \tilde{\rho}_n\left(x,-\frac{1}{n}\right) 
      &= \int_{\mathbb{R}} \tilde{\mu}_{n,z}(x) h_{\frac{1}{n}}\left(-\frac{1}{n}-z\right)\, dz 
       = \int_{-\frac{2}{n}}^0 \tilde{\mu}_{n,z}(x) h_{\frac{1}{n}}\left(-\frac{1}{n}-z\right)\, dz\\
      &= \int_{-\frac{2}{n}}^0 \tilde{\mu}_{n,0}(x) h_{\frac{1}{n}}\left(-\frac{1}{n}-z\right)\, dz
       = \tilde{\mu}_{n,0}(x). 
   \end{align*}
   The other equality can be proved similarly.
   By the convexity of $|a|^2/m(b)$ for each $(a,b) \in \Rd\times(0,\infty)$ 
   and Jensen's inequality, we have 
   $$\frac{|\tilde{\boldsymbol{w}}_n(x,t)|^2}{m(\tilde{\rho}_n(x,t))} 
   \leq \int_{\mathbb{R}} \frac{|\tilde{\boldsymbol{\nu}}_{n,z}(x)|^2}{m(\tilde{\mu}_{n,z}(x))}h_{\frac{1}{n}}(t-z)\, dz
   \quad \mathrm{for}\ (x,t) \in \Rd\times\mathbb{R},$$
   then
   \begin{equation}\label{eq63}
      \int_{\Rd} \frac{|\tilde{\boldsymbol{w}}_n(x,t)|^2}{m(\tilde{\rho}_n(x,t))}\, dx 
      \leq \int_{\Rd}\int_{\mathbb{R}} \frac{|\tilde{\boldsymbol{\nu}}_{n,z}(x)|^2}{m(\tilde{\mu}_{n,z}(x))}h_{\frac{1}{n}}(t-z)\, dz\,dx
      \quad \mathrm{for}\ t \in \mathbb{R}.
   \end{equation}
   Since $\tilde{\boldsymbol{w}}_n(\cdot,t)=0$ if $t<-1/n$ or $t>1+1/n$ 
   and $\tilde{\boldsymbol{w}}_n(x,\cdot) = 0$ if $x \in \Rd\setminus\Omega_n$,
   by \eqref{eq63} and Fubini's theorem, it follows
   \begin{align}\label{eq64}
      \int_{-\frac{1}{n}}^{1+\frac{1}{n}}\int_{\Omega_n} \frac{|\tilde{\boldsymbol{w}}_n(x,t)|^2}{m(\tilde{\rho}_n(x,t))}\, dx\,dt 
      &= \int_{\mathbb{R}}\int_{\Rd} \frac{|\tilde{\boldsymbol{w}}_n(x,t)|^2}{m(\tilde{\rho}_n(x,t))}\, dx\,dt\notag\\
      &\leq \int_{\mathbb{R}}\int_{\Rd}\int_{\mathbb{R}} \frac{|\tilde{\boldsymbol{\nu}}_{n,z}(x)|^2}{m(\tilde{\mu}_{n,z}(x))}h_{\frac{1}{n}}(t-z)\, dz\,dx\,dt\notag\\
      &= \int_{\mathbb{R}}\int_{\Rd} \frac{|\tilde{\boldsymbol{\nu}}_{n,z}(x)|^2}{m(\tilde{\mu}_{n,z}(x))}\, dx\,dz\notag\\
      &\leq W_m(\mu_0,\mu_1)^2,
   \end{align}
   where last inequality is followed by Lemma \ref{lem2.8}.
   Thus we set 
   $$\rho_n(x,t) \coloneq \tilde{\rho}_n\left(x,c_nt-\frac{1}{n}\right),\ 
   \boldsymbol{w}_n(x,t) \coloneq c_n\tilde{\boldsymbol{w}}_n\left(x,c_nt-\frac{1}{n}\right)\quad 
   \mbox{where}\ c_n \coloneq 1+\frac{2}{n}.$$
   Then $\rho_n \in C^\infty(\overline{\Omega}_n\times[0,1]),\ 
   \boldsymbol{w}_n \in C^\infty(\overline{\Omega}_n\times[0,1];\Rd)$ and it holds
   $$\dt\rho_n(x,t)+\nabla\cdot \boldsymbol{w}_n(x,t) = 0\quad (x,t) \in \Omega_n\times(0,1).$$
   By \eqref{eq64} and the change of variables, we have
   \begin{align}\label{eq45}
      \frac{1}{c_n}\int_0^1\int_{\Omega_n} \frac{|\boldsymbol{w}_n(x,t)|^2}{m(\rho_n(x,t))}\, dx\,dt 
      &= \frac{1}{c_n}\int_0^1\int_{\Omega_n} \frac{c_n^2|\tilde{\boldsymbol{w}}_n\left(x,c_nt-\frac{1}{n}\right)|^2}{m\left(\tilde{\rho}_n\left(x,c_nt-\frac{1}{n}\right)\right)}\, dx\,dt\notag\\
      &= \int_{-\frac{1}{n}}^{1+\frac{1}{n}}\int_{\Omega_n} \frac{|\tilde{\boldsymbol{w}}_n(x,t)|^2}{m(\tilde{\rho}_n(x,t))}\, dx\,dt\notag\\
      &\leq W_m(\mu_0,\mu_1)^2.
   \end{align}
   Since $\mu_0, \mu_1 \in L^q\cap\mathcal{P}(\Omega)$, 
   using the property of the mollifier, we have
   $$\rho_n(i) \to \mu_i\ \mathrm{in}\ L^q(\Rd)\quad \mathrm{as}\ n \to \infty\quad i=0,1,$$
   in particular
   $$\rho_n(i) \rightharpoonup \mu_i\ \mathrm{weakly^*\ in}\ \mathcal{M}_{loc}^+(\Rd)\quad \mathrm{as}\ n \to \infty\quad i=0,1.$$
   We infer from Lemma \ref{lem2.7} that 
   $$W_m(\mu_0,\mu_1)^2 
   \leq \liminf_{n\to\infty}W_m(\rho_n(0),\rho_n(1))^2 
   \leq \liminf_{n\to\infty}\int_0^1\int_{\Omega_n} \frac{|\boldsymbol{w}_n|^2}{m(\rho_n)}\, dx\,dt.$$
   Combining the above with \eqref{eq45} and letting $n \to \infty$, we have
   \begin{equation}\label{eq65}
      W_m(\mu_0,\mu_1)^2 
      = \lim_{n\to\infty}W_m(\rho_n(0),\rho_n(1))^2
      = \lim_{n\to\infty}\int_0^1\int_{\Omega_n} \frac{|\boldsymbol{w}_n|^2}{m(\rho_n)}\, dx\,dt.
   \end{equation}

   Next, for fixed $t\in[0,1]$ we consider the following equation:
   \begin{align}\label{ee}
      \begin{cases}
         \nabla\cdot \boldsymbol{w}_n(x,t) 
          = \nabla\cdot(m(\rho_n(x,t))\nabla\phi_n(x)) &x \in \Omega_n,\\
         \nabla\phi_n(x)\cdot\boldsymbol{n} = 0 &x \in \partial\Omega_n.
      \end{cases}
   \end{align}
   Since $\mathrm{supp}(\boldsymbol{w}_n(\cdot,t)) \subset \Omega_n$, then
   $\boldsymbol{w}_n(x,t) = 0$ for $x \in \partial\Omega_n$.
   By the Gauss--Green theorem, it follows
   \begin{equation*}
      \int_{\Omega_n} \nabla\cdot\boldsymbol{w}_n(x,t)\, dx = 0.
   \end{equation*}
   Hence \eqref{ee} has a unique weak solution $\phi_n \in H^1(\Omega_n)$ such that
   \begin{equation}\label{wf}
      \int_{\Omega_n} m(\rho_n)\nabla\phi_n\cdot\nabla \psi\, dx 
      = \int_{\Omega_n} \boldsymbol{w}_n\cdot\nabla \psi\, dx\quad \forall \psi \in H^1(\Omega_n).
   \end{equation}
   Due to the elliptic regurality theorem, $\phi_n$ can be a smooth function 
   and satisfies $\nabla\cdot \boldsymbol{w}_n = \nabla\cdot(m(\rho_n)\nabla\phi_n)$ in $\Omega_n$.
   Taking $\psi=\phi_n$ in \eqref{wf} and using H\"older's inequality, we have
   \begin{equation*}
      \int_{\Omega_n} m(\rho_n)|\nabla\phi_n|^2\, dx 
      = \int_{\Omega_n} \boldsymbol{w}_n\cdot\nabla\phi_n\, dx 
      \leq \left(\int_{\Omega_n} \frac{|\boldsymbol{w}_n|^2}{m(\rho_n)}\, dx\right)^{\frac{1}{2}}\left(\int_{\Omega_n} m(\rho_n)|\nabla\phi_n|^2\, dx\right)^{\frac{1}{2}},
   \end{equation*}
   then
   $$\int_{\Omega_n} m(\rho_n)|\nabla\phi_n|^2\, dx 
   \leq \int_{\Omega_n} \frac{|\boldsymbol{w}_n|^2}{m(\rho_n)}\, dx.$$
   In addition, since it holds 
   $$\dt\rho_n + \nabla\cdot(m(\rho_n)\nabla\phi_n) 
    = \dt\rho_n + \nabla\cdot \boldsymbol{w}_n 
    = 0\quad \mathrm{in}\ \Omega_n\times(0,1),$$
   we have $(\rho_n,m(\rho_n)\nabla\phi_n) \in CE(0,1;\rho_n(0),\rho_n(1))$.
   Combining this with \eqref{eq65}, we can conclude
   \begin{align*}
      W_m(\mu_0,\mu_1)^2 
      &= \lim_{n\to\infty}W_m(\rho_n(0),\rho_n(1))^2 
       \leq \lim_{n\to\infty}\int_0^1\int_{\Omega_n} m(\rho_n)|\nabla\phi_n|^2\, dx\,dt\\
      &\leq \lim_{n\to\infty}\int_0^1\int_{\Omega_n}\frac{|\boldsymbol{w}_n|^2}{m(\rho_n)}\, dx\,dt 
       = W_m(\mu_0,\mu_1)^2,
   \end{align*}
   then
   $$W_m(\mu_0,\mu_1)^2 
   = \lim_{n\to\infty}\int_0^1\int_{\Omega_n} m(\rho_n)|\nabla\phi_n|^2\, dx\,dt.$$
   The proof is completed.
\end{proof}

\subsection{Properties of the functional $\boldsymbol{U}_\varepsilon$ and a boundary estimate}

\quad First, we establish the uniform estimate for the functional $\boldsymbol{U}_\varepsilon$ with respect to $\varepsilon$.
This estimate plays an important role in sections 5 and 6.

\begin{lem}\label{lem2.11}
   Let $p\geq1, 0<\alpha<1$ and 
   $U_\varepsilon : [0,\infty) \to \mathbb{R}$ be a function such that
   $U_\varepsilon^{\prime\prime}(r)m_\varepsilon(r) = 1$ for $r \in [0,\infty)$ 
   and $U_\varepsilon^\prime(0) = U_\varepsilon(0) = 0$, where $m_\varepsilon(r) = (r+\varepsilon)^\alpha$.
   Then setting $\boldsymbol{U}_\varepsilon : L^{p+1-\alpha}\cap\mathcal{P}(\Omega) \to \mathbb{R}$ by
   \begin{align*}
      \boldsymbol{U}_{\varepsilon}(u) \coloneq \int_{\Omega} U_\varepsilon(u(x))\, dx,
   \end{align*}
   it hold $\boldsymbol{U}_\varepsilon \geq 0$ and
   \begin{align}\label{eq41}
      \boldsymbol{U}_\varepsilon(u) \leq \frac{1}{1-\alpha}\|u\|_{L^{2-\alpha}(\Omega)}^{2-\alpha}\quad 
      \mathrm{for}\ u \in L^{p+1-\alpha}\cap\mathcal{P}(\Omega).
   \end{align}
   In addition, $\boldsymbol{U}_{\varepsilon}$ is lower semicontinuous with respect to weak* convergence in $\mathcal{M}_{loc}^+(\Rd)$.
\end{lem}

\begin{proof}
   First, the function $U_{\varepsilon}$ can be explicitly represented as
   $$U_{\varepsilon}(r) 
   = \frac{1}{(2-\alpha)(1-\alpha)}[(r+\varepsilon)^{2-\alpha} 
   - \varepsilon^{2-\alpha}] - \frac{\varepsilon^{1-\alpha}}{1-\alpha}r.$$
   Since $U_\varepsilon^{\prime\prime} \geq 0$ and $U_\varepsilon^\prime(0) = U_\varepsilon(0) = 0$, 
   we see $U_\varepsilon \geq 0$ and then $\boldsymbol{U}_\varepsilon \geq 0$.
   Using the mean value theorem and the inequality
   \begin{equation}\label{eq43}
      a^\beta - b^\beta \leq |a-b|^\beta\quad \mathrm{for}\ a, b \geq0,\ 0<\beta<1,
   \end{equation}
   we have for $r \in [0,\infty)$,
   \begin{align*}
      U_\varepsilon(r) 
      &= \frac{1}{(2-\alpha)(1-\alpha)}[(r+\varepsilon)^{2-\alpha} 
      - \varepsilon^{2-\alpha}] - \frac{\varepsilon^{1-\alpha}}{1-\alpha}r\\
      &\leq \frac{1}{1-\alpha}r(r+\varepsilon)^{1-\alpha} - \frac{1}{1-\alpha}r\varepsilon^{1-\alpha}\\
      &= \frac{r}{1-\alpha}[(r+\varepsilon)^{1-\alpha} - \varepsilon^{1-\alpha}]
       \leq \frac{1}{1-\alpha}r^{2-\alpha},
   \end{align*}
   then 
   \begin{equation*}
      \boldsymbol{U}_\varepsilon(u) \leq \frac{1}{1-\alpha}\|u\|_{L^{2-\alpha}(\Omega)}^{2-\alpha}\quad
      \mathrm{for}\ u \in L^{p+1-\alpha}\cap \mathcal{P}(\Omega).
   \end{equation*}
   Finally, since the function $U_\varepsilon$ is convex and continuous, 
   by \cite[Theorem 2.3.4]{AFP}, $\boldsymbol{U}_\varepsilon$ is lower semicontinuous 
   with respect to weak* convergence in $\mathcal{M}_{loc}^+(\Rd)$.
\end{proof}

\quad Next lemma is about the estimate on the smooth boundary of convex domain (\cite[Lemma 5.1]{GST}).
Due to using this lemma in Lemma \ref{lem4.5}, 
we need to assume that domain $\Omega$ is convex.

\begin{lem}\label{lemA.2}
   Let $\Omega$ be a smooth convex set in $\Rd$ and 
   $\varphi \in C^3(\overline{\Omega})$ with $\nabla\varphi\cdot\boldsymbol{n} = 0$ on $\partial\Omega$.
   Then
   \begin{equation*}
      \nabla^2\varphi\nabla\varphi\cdot\boldsymbol{n} = \sum_{i,j=1}^d \partial_{ij}^2\varphi\, \partial_i\varphi\, n_j \leq 0
      \quad \mathrm{on}\ \partial\Omega,
   \end{equation*}
   where $\nabla^2\varphi$ is the Hessian matrix and 
   $\boldsymbol{n} = (n_1,\cdots,n_d)$ is the outer unit normal vector to $\partial\Omega$.
\end{lem}


\section{Existence of minimizers}

\quad In this section, we consider the following discrete scheme:

let $X \coloneq (L^{p+1-\alpha}\cap\mathcal{P}(\Omega))\times H^1(\Omega)$ 
and $m_\varepsilon(r) = (r+\varepsilon)^\alpha$ for $r\geq0$.
For a fixed time step $\tau>0$, 
\begin{align}\label{mms}
      \mbox{find}\ (u_{\tau}^k,v_\tau^k) \in X\ \mbox{satisfying}\
      F_{\tau}(u_{\tau}^k,v_\tau^k) = \inf_{(u,v)\in X}F_{\tau}(u,v)\
      \mbox{for each}\ k \in \mathbb{N},\
\end{align}
where $(u_\tau^0,v_\tau^0) = (u_0,v_0)$,
\begin{equation}\label{Ftau}
   F_{\tau}(u,v) \coloneq 
   \frac{1}{2\tau}\left(\frac{W_{m_{\varepsilon}}(u,u_{\tau}^{k-1})^2}{\chi} 
   + \|v-v_\tau^{k-1}\|_{L^2(\Omega)}^2\right) + E(u,v)\quad (u,v) \in X,
\end{equation}
and $W_{m_{\varepsilon}}$ is the weighted Wasserstein distance 
on the space of nonnegative Radon measures 

$\mathcal{M}^+(\Omega)$ (see Definition \ref{dfn2.4}).
Notice that $\mathcal{P}(\Omega) \subset \mathcal{M}^+(\Omega)$.
In order to show the existence of minimizers, 
we apply direct method, then we check that the functional $F_\tau$ is bounded below in $X$,
the sublevel set of $F_\tau$ is relatively compact in $X$ 
and $F_\tau$ is lower semicontinuous with respect to weak topology in $X$. 

\subsection{Boundedness from below of the functional $F_{\tau}$}

\begin{lem}\label{lem3.1}
   Let $p \geq 1+\alpha-2/d$ and assume that $\chi>0$ is small enough if $p=1+\alpha-2/d$.
   Then the energy functional $E$ is bounded below in $X$.
   In particular the functional $F_\tau$ is also bounded below.
\end{lem}

\begin{proof}
   If $p\geq1+\alpha$ then $p+1-\alpha\geq2$.
   Thus we infer from H\"older's inequality and the interpolation inequality that
   \begin{align*}
      \|uv\|_{L^1(\Omega)} 
      &\leq \|u\|_{L^2(\Omega)}\|v\|_{L^2(\Omega)}\\
      &\leq \|u\|_{L^{p+1-\alpha}(\Omega)}^{\frac{p+1-\alpha}{2(p-\alpha)}}\|u\|_{L^1(\Omega)}^\frac{p-1-\alpha}{2(p-\alpha)}\|v\|_{H^1(\Omega)}\\
      &= \|u\|_{L^{p+1-\alpha}(\Omega)}^{\frac{p+1-\alpha}{2(p-\alpha)}}\|v\|_{H^1(\Omega)}.
   \end{align*}
   Note that $(p+1-\alpha)/(p-\alpha) \leq p+1-\alpha$ because of $p\geq1+\alpha$.

   On the other hand, when $1+\alpha-2/d < p < 1+\alpha$, 
   since $d\geq2$, it follows
   $$p>1+\alpha-\frac{2}{d}\geq1+\alpha-\frac{4}{d+2},$$
   then
   $$p+1-\alpha > \frac{2d}{d+2}.$$
   When $d=2$, we can choose $q \in (1,p+1-\alpha)$ satisfying $q<2/(2-p+\alpha)$ and fix it.
   By H\"older's inequality, the interpolation inequality and the Sobolev embedding,
   we have for $(u,v) \in X$,
   \begin{align*}
      \|uv\|_{L^1(\Omega)} 
      &\leq \|u\|_{L^{q}(\Omega)}\|v\|_{L^{q^*}(\Omega)}\\
      &\leq \|u\|_{L^{p+1-\alpha}(\Omega)}^{\theta_2}\|u\|_{L^1(\Omega)}^{1-\theta_2}\|v\|_{L^{q^*}(\Omega)}\\
      &\leq C\|u\|_{L^{p+1-\alpha}(\Omega)}^{\theta_2}\|v\|_{H^1(\Omega)},
   \end{align*}
   where $C$ is a constant and 
   \begin{equation*}
      q^* = \frac{q}{q-1},\quad
      \theta_2 \coloneq \frac{(p+1-\alpha)(q-1)}{(p-\alpha)q} \in (0,1).
   \end{equation*}
   Note that $2\theta_2 < p+1-\alpha$ since $q<2/(2-p+\alpha)$.
   On the other hand, when $d\geq3$, we infer from the similar estimate that 
   \begin{align*}
      \|uv\|_{L^1(\Omega)} 
      &\leq \|u\|_{L^{\frac{2d}{d+2}}(\Omega)}\|v\|_{L^{\frac{2d}{d-2}}(\Omega)}\\
      &\leq \|u\|_{L^{p+1-\alpha}(\Omega)}^{\theta_d}\|u\|_{L^1(\Omega)}^{1-\theta_d}\|v\|_{L^{\frac{2d}{d-2}}(\Omega)}\\
      &\leq C\|u\|_{L^{p+1-\alpha}(\Omega)}^{\theta_d}\|v\|_{H^1(\Omega)},
   \end{align*}
   where 
   \begin{align*}
      \theta_d \coloneq \frac{(p+1-\alpha)(d-2)}{(p-\alpha)2d} \in (0,1).
   \end{align*}
   Observe that $2\theta_d < p+1-\alpha$ because of $p>1+\alpha-2/d$.
   Hence, when $p>1+\alpha-2/d$ and $d\geq2$, by Young's inequality, it follows
   \begin{align}\label{eq30}
      \|uv\|_{L^1(\Omega)}
      &\leq \frac{1}{\chi(p+1-\alpha)}\|u\|_{L^{p+1-\alpha}(\Omega)}^{p+1-\alpha} 
       + \frac{1}{4}\|v\|_{H^1(\Omega)}^2 + C(\alpha,p,d,\chi),
   \end{align}
   where $C(\alpha,p,d,\chi)$ is a constant.
   Then for $(u,v) \in X$, we obtain
   \begin{align}\label{eq22}
      E(u,v) 
      &\geq \frac{p}{\chi(p-\alpha)(p+1-\alpha)}\|u\|_{L^{p+1-\alpha}(\Omega)}^{p+1-\alpha} 
       - \frac{1}{\chi(p+1-\alpha)}\|u\|_{L^{p+1-\alpha}(\Omega)}^{p+1-\alpha}\notag\\
      &\quad - \frac{1}{4}\|v\|_{H^1(\Omega)}^2 + \frac{1}{2}\|\nabla v\|_{L^2(\Omega)}^2 
       + \frac{1}{2}\|v\|_{L^2(\Omega)}^2 - C(\alpha,p,d,\chi)\notag\\
      &\geq \frac{\alpha}{\chi(p-\alpha)(p+1-\alpha)}\|u\|_{L^{p+1-\alpha}(\Omega)}^{p+1-\alpha} 
       + \frac{1}{4}\|v\|_{H^1(\Omega)}^2 
       - C(\alpha,p,d,\chi)\\
      &\geq -C(\alpha,p,d,\chi)> -\infty.\notag
   \end{align}
   If $p=1+\alpha-2/d$ then $p+1-\alpha = 2 - 2/d$ and $\theta_d = 1 - 1/d$.
   Hence by the same argument in the above, it follows
   \begin{align}\label{eq89}
      \|uv\|_{L^1(\Omega)} 
      &\leq C\|u\|_{L^{2-\frac{2}{d}}(\Omega)}^{1-\frac{1}{d}}\|v\|_{H^1(\Omega)}\notag\\
      &\leq C\|u\|_{L^{2-\frac{2}{d}}(\Omega)}^{2-\frac{2}{d}} + \frac{1}{4}\|v\|_{H^1(\Omega)}^2.
   \end{align}
   Then for $(u,v) \in X$, we obtain
   \begin{align}\label{eq88}
      E(u,v)
      &\geq \frac{1+\alpha-\frac{2}{d}}{\chi\left(1-\frac{2}{d}\right)\left(2-\frac{2}{d}\right)}\|u\|_{L^{2-\frac{2}{d}}(\Omega)}^{2-\frac{2}{d}}
       - C\|u\|_{L^{2-\frac{2}{d}}(\Omega)}^{2-\frac{2}{d}}\notag\\
      &\quad - \frac{1}{4}\|v\|_{H^1(\Omega)}^2 + \frac{1}{2}\|\nabla v\|_{L^2(\Omega)}^2 
       + \frac{1}{2}\|v\|_{L^2(\Omega)}^2\notag\\
      &\geq \frac{1}{\chi}\left[\frac{1+\alpha-\frac{2}{d}}{\left(1-\frac{2}{d}\right)\left(2-\frac{2}{d}\right)} - \chi C\right]\|u\|_{L^{2-\frac{2}{d}}(\Omega)}^{2-\frac{2}{d}}
       + \frac{1}{4}\|v\|_{H^1(\Omega)}^2.
   \end{align}
   If $\chi>0$ is small enough then the first term of the right hand side is positive and $E$ is bounded below in $X$.
   Since $\chi>0$, $W_{m_{\varepsilon}}\geq0$ and $\|\cdot\|_{L^2(\Omega)}\geq0$, 
   $F_{\tau}$ is also bounded below.
   We complete the proof.
\end{proof}

\subsection{Compactness}

\begin{lem}\label{lem3.2}
   Let $\{(u_n,v_n)\}_{n\in\mathbb{N}}$ be a minimizing sequence of $F_\tau$ in $X$.
   Then there exist a subsequence $\{(u_{n_l},v_{n_l})\}_{l\in\mathbb{N}}$ and a pair of functions $(u,v) \in X$ such that
   \begin{align*}
      &u_{n_l} \rightharpoonup u\quad \mathrm{weakly\ in}\ L^1\cap L^{p+1-\alpha}(\Omega)\ \mathrm{as}\ l \to \infty,\\
      &v_{n_l} \rightharpoonup v\quad \mathrm{weakly\ in}\ H^1(\Omega)\ \mathrm{as}\ l \to \infty,\\
      &v_{n_l} \to v\quad \mathrm{strongly\ in}\ L^q(\Omega)\ \mathrm{as}\ l \to \infty\ \mathrm{for\ all}\ q \in [2,2^*),
   \end{align*}
   where 
   \begin{align*}
      2^* \coloneq 
      \begin{cases}
         \infty\quad &\mathrm{if}\ d=2,\\
         \frac{2d}{d-2} &\mathrm{if}\ d\geq3.
      \end{cases}
   \end{align*}
\end{lem}

\begin{proof}
   Let $\{(u_n,v_n)\}\subset X$ be a minimizing sequence of $F_{\tau}$, 
   that is, $F_{\tau}(u_n,v_n)$ is bounded in $\mathbb{R}$.
   Combining this with the estimate \eqref{eq22} or \eqref{eq88}, 
   we get the boundedness of $\|u_n\|_{L^{p+1-\alpha}(\Omega)}$ and 
   $\|v_n\|_{H^1(\Omega)}$:
   there exists a constant $C = C(\alpha,p,d,\chi)$ such that 
   \begin{align}\label{eq23}
      &\|u_n\|_{L^{p+1-\alpha}(\Omega)}^{p+1-\alpha} \leq C,\\
      &\|v_n\|_{H^1(\Omega)}^2 \leq C.\notag
   \end{align}
   Since $p+1-\alpha > 1$, by the Banach--Alaoglu theorem, 
   there exist subsequences $\{u_{n_l}\}_{l\in\mathbb{N}}, \{v_{n_l}\}_{l\in\mathbb{N}}$ 
   and functions $u \in L^{p+1-\alpha}(\Omega), v \in H^1(\Omega)$ such that
   \begin{align}\label{eq24}
      &u_{n_l} \rightharpoonup u\quad \mathrm{weakly\ in}\ L^{p+1-\alpha}(\Omega)\ \mathrm{as}\ l \to \infty,\\
      &v_{n_l} \rightharpoonup v\quad \mathrm{weakly\ in}\ H^1(\Omega)\ \mathrm{as}\ l \to \infty.\notag
   \end{align}
   Moreover, by the Rellich--Kondrachov theorem, 
   we can take a subsequence, still denote $\{v_{n_l}\}$, satisfying
   \begin{equation*}
      v_{n_l} \to v\quad \mathrm{strongly\ in}\ L^q(\Omega)\ \mathrm{as}\ l \to \infty\ \mathrm{for\ all}\ q \in [2,2^*).
   \end{equation*}
   Since $\{u_{n_l}\} \subset \mathcal{P}(\Omega)$,
   $\|u_{n_l}\|_{L^1(\Omega)}$ is bounded.
   For $c>0$,\ we see
   $$ \int_{\{u_{n_l}\geq c\}} u_{n_l}\, dx 
   \leq \int_{\{u_{n_l}\geq c\}} \frac{u_{n_l}^{p+1-\alpha}}{c^{p-\alpha}}\, dx 
   \leq \frac{1}{c^{p-\alpha}}\int_{\Omega} u_{n_l}^{p+1-\alpha}\, dx.$$
   From \eqref{eq23}, we have
   $$\limsup_{c\to\infty}\sup_{l\in\mathbb{N}}\int_{\{u_{n_l}\geq c\}} u_{n_l}\, dx =0.$$
   Thus $\{u_{n_l}\}$ is equi-integrable.
   By the Dunford--Pettis theorem, 
   there exist a subsequence (not relabeled) 
   and a function $\tilde{u} \in L^1(\Omega)$ such that
   \begin{equation}\label{eq26}
      u_{n_l} \rightharpoonup \tilde{u}\quad \mathrm{weakly\ in}\ L^1(\Omega)\ \mathrm{as}\ l \to \infty.
   \end{equation}
   Here, for all $f \in C_c^{\infty}(\Omega)$, 
   from \eqref{eq24} and \eqref{eq26}, we have
   $$\int_{\Omega} u(x)f(x)\, dx  
   = \int_{\Omega} \tilde{u}(x)f(x)\, dx.$$
   Therefore we see $u = \tilde{u}$ a.e. in $\Omega$. 
   By \eqref{eq26}, we have 
   $$1 = \int_{\Omega} u_{n_l}(x)\, dx \to \int_{\Omega} u(x)\, dx\quad \mathrm{as}\ l\to\infty.$$
   Hence we obtain $u \in L^{p+1-\alpha}\cap\mathcal{P}(\Omega)$ then $(u,v) \in X$.
   The proof is completed.
\end{proof}

\subsection{Lower semicontinuity}

\begin{lem}\label{lem3.4}
   Let $\{(u_n,v_n)\} \subset X$ and $(u,v) \in X$ satisfying 
   $u_n \rightharpoonup u$ weakly in $L^1\cap L^{p+1-\alpha}(\Omega)$ and 
   $v_n \to v$ weakly in $H^1(\Omega)$ and strongly in $L^q(\Omega)$ as $n \to \infty$ 
   for all $q \in [2,2^*)$.
   Then 
   $$F_\tau(u,v) \leq \liminf_{n\to\infty}F_\tau(u_n,v_n).$$
\end{lem}

\begin{proof}
   Fix $(u_\tau^{k-1},v_\tau^{k-1}) \in X$.
   Let $\{(u_n,v_n)\} \subset X$ and $(u,v) \in X$ satisfying 
   $u_n \rightharpoonup u$ weakly in $L^1\cap L^{p+1-\alpha}(\Omega)$ and 
   $v_n \to v$ weakly in $H^1(\Omega)$ and strongly in $L^q(\Omega)$ for all $q\in[2,2^*)$ as $n \to \infty$.
   Then, for all $f \in C_c(\Rd)$, it follows
   \begin{align*}
      \int_{\Rd} f(x)u_n(x)\, dx 
      = \int_{\Omega} f(x)u_n(x)\, dx
      \to \int_{\Omega} f(x)u(x)\, dx
      =\int_{\Rd} f(x)u(x)\, dx
      \quad \mathrm{as}\ n\to\infty,
   \end{align*}
   thus $u_n \rightharpoonup u$ weakly* in $\mathcal{M}_{loc}^+(\Rd)$ as $n\to\infty$.
   Since $W_{m_{\varepsilon}}$ is lower semicontinuous 
   with respect to weak* convergence in $\mathcal{M}_{loc}^+(\Rd)$ (see Lemma \ref{lem2.7}),
   we have
   \begin{equation*}
      W_{m_{\varepsilon}}(u,u_\tau^{k-1})^2 
      \leq \liminf_{n\to\infty}W_{m_{\varepsilon}}(u_n,u_\tau^{k-1})^2.
   \end{equation*}
   In addition, since the norm is lower semicontinuous
   with respect to the weak topology, 
   we have
   \begin{align*}
      &\|v - v_\tau^{k-1}\|_{L^2(\Omega)}^2 
      \leq \liminf_{n\to\infty}\|v_n - v_\tau^{k-1}\|_{L^2(\Omega)}^2,\\
      &\|u\|_{L^{p+1-\alpha}(\Omega)}^{p+1-\alpha}
      \leq \liminf_{n\to\infty}\|u_n\|_{L^{p+1-\alpha}(\Omega)}^{p+1-\alpha},\\
      &\|v\|_{H^1(\Omega)}^2 
      \leq \liminf_{n\to\infty}\|v_n\|_{H^1(\Omega)}^2.
   \end{align*}
   We will show 
   \begin{equation}\label{eq29}
      \lim_{n\to\infty}\|u_nv_n\|_{L^1(\Omega)} = \|uv\|_{L^1(\Omega)}.
   \end{equation}
   By H\"older's inequality, it follows
   \begin{align*}
      \left|\int_{\Omega} (uv - u_nv_n) \, dx\right|
      &\leq \left|\int_{\Omega} (u - u_n)v\, dx\right| + \int_{\Omega} |u_n||v - v_n|\, dx\\
      &\leq \left|\int_{\Omega} (u - u_n)v\, dx\right| + \|u_n\|_{L^{p+1-\alpha}(\Omega)}\|v-v_n\|_{L^{\frac{p+1-\alpha}{p-\alpha}}(\Omega)}.
   \end{align*}
   Note that $2\leq (p+1-\alpha)/(p-\alpha) < 2^*$ and $v \in L^{\frac{p+1-\alpha}{p-\alpha}}(\Omega)$
   because of $p+1-\alpha>2/d, p\geq1$ and $0<\alpha<1$.
   Since $\|u_n\|_{L^{p+1-\alpha}(\Omega)}$ is bounded, 
   using asummptions, we obtain
   \begin{equation*}
      \lim_{n\to\infty}\left|\int_{\Omega} (uv - u_nv_n)\, dx\right| = 0,
   \end{equation*}
   then \eqref{eq29} holds.
   From these, we complete the proof.
\end{proof}

\subsection{Conclusion}

\begin{prop}\label{prop3.5}
   Let $p\geq1+\alpha-2/d$ and assume that $\chi>0$ is small enough if $p=1+\alpha-2/d$. 
   Let $(u_0,v_0) \in X$ be a pair of nonnegative functions.
   Then for each $k \in \mathbb{N}$,
   there is at least one minimizer $(u_\tau^k,v_\tau^k) \in X$ in \eqref{mms}
   and the following inequalities hold: 
   \begin{align}
      &\frac{1}{2\tau}\left(\frac{W_{m_\varepsilon}(u_\tau^k,u_\tau^{k-1})^2}{\chi} + \|v_\tau^k - v_\tau^{k-1}\|_{L^2(\Omega)}^2\right) + E(u_\tau^k,v_\tau^k)\label{eq37}\\
      &\hspace{0.5cm} \leq \frac{1}{2\tau}\left(\frac{W_{m_\varepsilon}(\tilde{u},u_\tau^{k-1})^2}{\chi} + \|\tilde{v} - v_\tau^{k-1}\|_{L^2(\Omega)}^2\right) + E(\tilde{u},\tilde{v})\quad \forall (\tilde{u},\tilde{v}) \in X,\ \forall k \in \mathbb{N},\notag\\
      &E(u_\tau^k,v_\tau^k) \leq E(u_\tau^{k-1},v_\tau^{k-1})\quad \forall k \in \mathbb{N}.\label{eq38}
   \end{align}
\end{prop}

\begin{proof}
   Let $(u_\tau^{k-1},v_\tau^{k-1}) \in X$ for $k \in \mathbb{N}$.
   By Lemma \ref{lem3.1}, 
   there exists a minimizing sequence $\{(u_n,v_n)\} \subset X$ 
   such that $F_\tau(u_n,v_n) \to \inf_{(u,v)\in X}F_\tau(u,v) > -\infty$ as $n\to\infty$.
   By Lemma \ref{lem3.2},
   there exist a subsequence $\{(u_{n_l},v_{n_l})\}_{l\in\mathbb{N}}$ and a pair of functions $(u_\tau^k,v_\tau^k) \in X$ such that
   \begin{align*}
      &u_{n_l} \rightharpoonup u_\tau^k\quad \mathrm{weakly\ in}\ L^1\cap L^{p+1-\alpha}(\Omega)\ \mathrm{as}\ l\to\infty,\\
      &v_{n_l} \rightharpoonup v_\tau^k\quad \mathrm{weakly\ in}\ H^1(\Omega)\ \mathrm{as}\ l \to \infty,\\
      &v_{n_l} \to v_\tau^k\quad \mathrm{strongly\ in}\ L^q(\Omega)\ \mathrm{as}\ l \to \infty\ \mathrm{for\ all}\ q \in [2,2^*).
   \end{align*}
   By Lemma \ref{lem3.4}, we have
   \begin{equation*}
      F_\tau(u_\tau^k,v_\tau^k) \leq \liminf_{l\to\infty} F_\tau(u_{n_l},v_{n_l}).
   \end{equation*}
   As a result, we see
   \begin{align*}
      \inf_{(u,v)\in X} F_{\tau}(u,v) 
      \leq F_\tau(u_\tau^k,v_\tau^k)
      \leq \liminf_{l\to\infty} F_\tau(u_{n_l},v_{n_l})
      = \inf_{(u,v)\in X} F_{\tau}(u,v).
   \end{align*}
   Thus $(u_{\tau}^k,v_\tau^k)$ is the minimizer of $F_{\tau}$ in $X$
   and \eqref{eq37} holds obviously.
   In particular, choosing $(\tilde{u},\tilde{v}) = (u_\tau^{k-1},v_\tau^{k-1})$ in \eqref{eq37}, 
   we obtain \eqref{eq38} and the proof is completed.
\end{proof}

\begin{rem}\label{rem3.5}
   In Proposition \ref{prop3.5}, 
   we may take $v_\tau^k \in H^1(\Omega)$ which is not nonnegative,
   however, we can choose a pair of nonnegative functions as a minimizer of $F_\tau$ in $X$.
   Indeed, let $(u_\tau^k,v_\tau^k)$ be the minimizer in Proposition \ref{prop3.5} 
   and assume that $v_\tau^{k-1}$ is nonnegative.
   Then since $u_\tau^k$ is nonnegative, $|v_\tau^k| \in H^1(\Omega)$ and 
   $||v_\tau^k| - v_\tau^{k-1}| \leq |v_\tau^k - v_\tau^{k-1}|$, 
   we have
   \begin{align*}
      F_\tau(u_\tau^k,v_\tau^k)
      &\leq F_\tau(u_\tau^k,|v_\tau^k|)\\
      &= \frac{1}{2\tau}\left[\frac{W_{m_\varepsilon}(u_\tau^k,u_\tau^{k-1})^2}{\chi} + \||v_\tau^k| - v_\tau^{k-1}\|_{L^2(\Omega)}^2\right]
       + E(u_\tau^k,|v_\tau^k|)\\
      &\leq \frac{1}{2\tau}\left[\frac{W_{m_\varepsilon}(u_\tau^k,u_\tau^{k-1})^2}{\chi} + \|v_\tau^k - v_\tau^{k-1}\|_{L^2(\Omega)}^2\right]
       + E(u_\tau^k,v_\tau^k)
       = F_\tau(u_\tau^k,v_\tau^k). 
   \end{align*}
   Thus we can choose $(u_\tau^k,|v_\tau^k|) \in X$ as a minimizer of $F_\tau$ in $X$ instead of $(u_\tau^k,v_\tau^k)$.
   In the rest of the paper, 
   we call $(u_\tau^k,|v_\tau^k|)$ a minimizer of $F_\tau$ in $X$ for $k\in\mathbb{N}$ 
   and denote by $(u_\tau^k,v_\tau^k)$.
\end{rem}


\section{Euler--Lagrange equations}

\subsection{Flow interchange lemma}

\quad First, 
we show the existence of a solution to the other equation which is used later.
The important properties of this solution 
are the mass conservation law and nonnegativity.

\begin{prop}\label{prop4.1}
   Let $1+\alpha-2/d \leq p \leq 1+\alpha$.
   Let $w_0 \in L^1\cap L^2\cap W^{1,p+1-\alpha}(\Omega)$ be a nonnegative function.
   Then, for $\delta>0$ and $\varphi \in C^\infty(\bar{\Omega})$ with $\nabla\varphi\cdot\boldsymbol{n} = 0$ on $\partial\Omega$, 
   there exist $T_0 = T_0(\alpha,\varepsilon,\varphi,\Omega,w_0)>0$ and a unique local solution $w$ satisfying
   \begin{align}\label{eq72}
      &\bullet w \in C([0,T_0];L^1\cap L^2\cap W^{1,p+1-\alpha}(\Omega))\cap C((0,T_0];W^{2,p+1-\alpha}(\Omega))\cap C^1((0,T_0];L^{p+1-\alpha}(\Omega)),\notag\\
      &\bullet 
      \begin{cases}
         \dt w = \delta\Delta w + \nabla\cdot(m_\varepsilon(w)\nabla\varphi)\quad &\mathrm{a.e.\ in}\ \Omega\times(0,T_0],\\
         \nabla w \cdot \boldsymbol{n} = 0\quad &\mathrm{on}\ \partial\Omega\times(0,T_0],\\
         w(0) = w_0\quad &\mathrm{in}\ L^1\cap L^2\cap W^{1,p+1-\alpha}(\Omega),
      \end{cases}\\
      &\bullet w(t) \geq 0,\ \|w(t)\|_{L^1(\Omega)} = \|w_0\|_{L^1(\Omega)}\quad t \in [0,T_0].\notag
   \end{align}
\end{prop}

This proposition is proved by the contraction mapping theorem.
However, we need to take care with nonnegativity of functions 
because $m_\varepsilon(r) = (r+\varepsilon)^\alpha$ can not be definded for $r<-\varepsilon$.
To overcome this, 
we inductively define a special contraction map 
depending a nonnegative function of a previous step (see Appendix).
This idea is inspired by \cite{KNST}.

\begin{prop}\label{prop4.2}
   Let $w$ be a local solution in Proposition \ref{prop4.1}.
   Then $w$ can be extended globally in time.
\end{prop}

\begin{proof}
   Let $w \in C([0,T_0];L^2\cap W^{1,p+1-\alpha}(\Omega))\cap C((0,T_0];W^{2,p+1-\alpha}(\Omega))\cap C^1((0,T_0];L^{p+1-\alpha}(\Omega))$ 
   be a nonnegative local solution to \eqref{eq72}.
   Then multiplying the first equation of \eqref{eq72} by $w(t)$ for $t \in (0,T_0]$ 
   and integrating in $\Omega$, we have 
   \begin{align*}
      \frac{d}{dt}\int_{\Omega} |w(t)|^2\, dx
      = \delta \int_{\Omega} (\Delta w(t))w(t)\, dx + \int_{\Omega} \nabla\cdot((w(t)+\varepsilon)^\alpha\nabla\varphi)w(t)\, dx.
   \end{align*}
   Note that thanks to $1+\alpha-2/d\leq p\leq 1+\alpha$ and the Sobolev embedding theorem,
   we have 
   $w(t) \in W^{2,p+1-\alpha}(\Omega) \hookrightarrow H^1(\Omega) \hookrightarrow L^{\frac{p+1-\alpha}{p-\alpha}}(\Omega)$,
   thus the right hand side is well-defined.
   Since $\nabla w\cdot\boldsymbol{n} = 0$ and $\nabla\varphi\cdot\boldsymbol{n} = 0$ on $\partial\Omega\times(0,T_0]$,
   we infer from integration by parts, H\"older's inequality and Young's inequality that
   \begin{align*}
   \frac{d}{dt}\|w(t)\|_{L^2(\Omega)}^2 
   &= -\delta\|\nabla w(t)\|_{L^2(\Omega)}^2 - \int_{\Omega} (w(t)+\varepsilon)^\alpha\nabla\varphi\cdot\nabla w(t)\, dx\\
   &\leq -\delta\|\nabla w(t)\|_{L^2(\Omega)}^2 + \|(w(t)+\varepsilon)^\alpha\nabla\varphi\|_{L^2(\Omega)}\|\nabla w(t)\|_{L^2(\Omega)}\\
   &\leq -\delta\|\nabla w(t)\|_{L^2(\Omega)}^2 + \delta\|\nabla w(t)\|_{L^2(\Omega)}^2 + \frac{1}{4\delta}\|(w(t)^\alpha+\varepsilon^\alpha)\nabla\varphi\|_{L^2(\Omega)}^2\\
   &\leq \frac{1}{2\delta}\left(\|w(t)\|_{L^2(\Omega)}^{2\alpha}\|\nabla\varphi\|_{L^{\frac{2}{1-\alpha}}(\Omega)}^2 + \varepsilon^{2\alpha}\|\nabla\varphi\|_{L^2(\Omega)}^2\right)\\
   &\leq \|w(t)\|_{L^2(\Omega)}^2 + C(\delta,\alpha,\varepsilon,\varphi),
   \end{align*}
   where $C(\delta,\alpha,\varepsilon,\varphi)$ is a constant. 
   Using Gronwall's lemma, we obtain
   \begin{align*}
      \|w(t)\|_{L^2(\Omega)}^2 \leq 
      \left[\|w(0)\|_{L^2(\Omega)}^2 + C(\delta,\alpha,\varepsilon,\varphi)\right]e^{T_0}\quad \mathrm{for}\ t \in [0,T_0].
   \end{align*}
   Since $p+1-\alpha \leq 2$, it also follows
   \begin{align*}
      \|w(t)\|_{L^{p+1-\alpha}(\Omega)}^{p+1-\alpha} 
      \leq |\Omega|^{\frac{1+\alpha-p}{2}}\left(\left[\|w(0)\|_{L^2(\Omega)}^2 + C(\delta,\alpha,\varepsilon,\varphi)\right]e^{T_0}\right)^{\frac{p+1-\alpha}{2}}
      \quad \mathrm{for}\ t \in [0,T_0].
   \end{align*}
   Next, setting $f(w) \coloneq \nabla\cdot((w+\varepsilon)^\alpha\nabla\varphi)$, 
   we have
   \begin{align*}
      &\|f(w)\|_{L^{p+1-\alpha}(\Omega)}\\
      &= \left\|\frac{\alpha\nabla w\cdot\nabla\varphi}{(w+\varepsilon)^{1-\alpha}} + (w+\varepsilon)^\alpha\Delta\varphi\right\|_{L^{p+1-\alpha}(\Omega)}\\
      &\leq \frac{\alpha\|\nabla\varphi\|_{L^\infty(\Omega)}}{\varepsilon^{1-\alpha}}\|\nabla w\|_{L^{p+1-\alpha}(\Omega)} + \|w\|_{L^{p+1-\alpha}(\Omega)}^\alpha\|\Delta\varphi\|_{L^{\frac{p+1-\alpha}{1-\alpha}}(\Omega)} + \varepsilon^\alpha\|\Delta\varphi\|_{L^{p+1-\alpha}(\Omega)}\\
      &\leq C(\alpha,\varepsilon,\varphi)(1+\|w\|_{W^{1,p+1-\alpha}(\Omega)}).
   \end{align*}
   By \cite[Proposition 7.2.2]{Lun}, it follows that
   $\|w(t)\|_{W^{1,p+1-\alpha}(\Omega)}$ is also bounded in $[0,T_0]$.
   Hence, combining this boundedness and Proposition \ref{prop4.1}, we can extend $w$ globally in time.
\end{proof}

\quad In order to consider the Euler--Lagrange equation for $u_\tau^k$, 
we need to use the flow interchange lemma (\cite{LMS}).
In the following, we prepare some lemmas to adapt it in our case.

\quad Fix $\varphi \in C^\infty(\overline{\Omega})$ with $\nabla\varphi\cdot\boldsymbol{n} = 0$ on $\partial\Omega$ and $\delta>0$.
Without loss of the generality, we assume $0\in \Omega$.
Let $\{a_n\}_n$ be a monotonically decreasing sequence converging to $1$ as $n\to\infty$
and set $\Omega_n \coloneq \{a_nx:x\in \Omega\}$.
By Lemma \ref{lem2.10} with $\mu_0=u_{\tau}^{k-1}$ and $\mu_1=u_{\tau}^k$, 
there exist a vanishing sequence $\{b_n\}_n$ such that $\Omega_{[b_n]} \subset \Omega_n$ and
a nonnegative function $\tilde{\rho}_n \in C^\infty(\overline{\Omega}_n\times[0,1])$ such that
$\|\tilde{\rho}_n(t)\|_{L^1(\Omega_n)} = 1$ for $n\in\mathbb{N}$ and $t \in [0,1]$
and a function $\tilde{\phi}_n \in C^\infty(\overline{\Omega}_n\times[0,1])$.
Notice that they satisfy
\begin{align*}
   W_{m_\varepsilon}(u_\tau^k,u_\tau^{k-1})^2 = 
   \lim_{n\to\infty}\int_0^1\int_{\Omega_n} m_\varepsilon(\tilde{\rho_n})|\nabla \tilde{\phi}_n|^2\, dx\, dt.
\end{align*}
Then we define $\rho_n : \overline{\Omega}\times[0,1] \to [0,\infty)$ by
$\rho_n(x,t) \coloneq \tilde{\rho}_n(a_nx,t)$ for $(x,t)\in \overline{\Omega}\times[0,1]$
and similarly, $\phi_n:\overline{\Omega}\times[0,1] \to \mathbb{R}$ by
$\phi_n(x,t) \coloneq \tilde{\phi}_n(a_nx,t)$ for $(x,t) \in \overline{\Omega}\times[0,1]$.
Note that $\rho_n,\phi_n \in C^\infty(\overline{\Omega}\times[0,1])$, 
$\|\rho_n(t)\|_{L^1(\Omega)} = a_n^{-d}$ for $n\in \mathbb{N}$ and $t\in[0,1]$ 
and they satisfy 
\begin{align}\label{eq82}
   W_{m_\varepsilon}(u_\tau^k,u_\tau^{k-1})^2 = 
   \lim_{n\to\infty}a_n^d\int_0^1\int_{\Omega} m_\varepsilon(\rho_n)|\nabla \phi_n|^2\, dx\, dt.
\end{align}
Moreover, since $\tilde{\rho}_n(i) \to u_\tau^{k-1+i}$ in $L^1(\Rd)$ for $i=0,1$ 
and $a_n \to 1$ as $n\to\infty$,
we have $\rho_n(i) \to u_\tau^{k-1+i}$ in $L^1(\Omega)$ as $n\to\infty$ for $i=0,1$.

Fix $n \in \mathbb{N}$ and $t \in [0,1]$.
Adapting Proposition \ref{prop4.1} and Proposition \ref{prop4.2} with $w_0 = \rho_n(t)$, 
we have a solution $S_z\rho_n(t)$ satisfying
\begin{align}
   &\bullet S_z\rho_n(t) \in C^\infty(\overline{\Omega}\times[0,\infty)),\notag\\
   &\bullet \partial_z (S_z\rho_n(t)) 
    = \delta\Delta (S_z\rho_n(t)) + \nabla\cdot(m_\varepsilon(S_z\rho_n(t))\nabla\varphi)
    \quad \mathrm{in}\ \Omega\times[0,\infty),\label{eq46}\\
   &\bullet \nabla S_z\rho_n(t)\cdot\boldsymbol{n} = 0\quad \mathrm{on}\ \partial\Omega\times(0,\infty),\notag\\
   &\bullet S_z\rho_n(t) \geq 0,\ \|S_z\rho_n(t)\|_{L^1(\Omega)} = \|\rho_n(t)\|_{L^1(\Omega)} = a_n^{-d}
    \quad \mathrm{for}\ z \in [0,\infty).\notag
\end{align}

\begin{rem}\label{rem4.4}
   Define $\rho_n^h(t) \coloneq S_{ht}\rho_n(t)$ for $h\in (0,1)$.
   Due to the smoothness of $S_z\rho_n(t)$ and $\rho_n(t)$, 
   $\rho_n^h(t)$ is $t$-differentiable in $(0,1)$.
   Since for each $n\in \mathbb{N}$, $\|\rho_n^h(t)\|_{L^1(\Omega)} = a_n^{-d}$ for all $t\in[0,1]$ and $h>0$,
   we have
   \begin{align*}
      \int_{\Omega} \dt\rho_n^h(t)\, dx = 0.
   \end{align*}
   Hence for fixed $n\in\mathbb{N}$ and $t\in[0,1]$, 
   as in the proof of Lemma \ref{lem2.10}, 
   we can find a unique solution $\phi_n^h(t) \in H^1(\Omega)$ satisfying
   \begin{equation}\label{ee2}
      \dt \rho_n^h(t) 
      = -\nabla\cdot(m_{\varepsilon}(\rho_n^h(t))\nabla\phi_n^h(t))\quad \mathrm{in}\ \Omega,\quad
      \nabla\phi_n^h(t)\cdot\boldsymbol{n} = 0\quad \mathrm{on}\ \partial\Omega.
   \end{equation}
   and the elliptic regularity theorem yields $\phi_n^h \in C^\infty(\overline{\Omega}\times[0,1])$.
   In addition, due to the smoothness of $\rho_n^h(t) = S_{ht}\rho_n(t)$ with respect to $h$,
   $\phi_n^h$ is $h$-differentiable.

   Then we define
   \begin{equation*}
      \boldsymbol{A}_n^h(t) 
      \coloneq \int_{\Omega} m_{\varepsilon}(\rho_n^h(t))|\nabla\phi_n^h(t)|^2\, dx,
   \end{equation*}
   and 
   $\boldsymbol{V}_\delta : L^{p+1-\alpha}\cap \mathcal{P}(\Omega) \to (-\infty, \infty)$ by
   \begin{align*}
      \boldsymbol{V}_\delta(u) 
      \coloneq \int_{\Omega} u\varphi\, dx + \delta \boldsymbol{U}_\varepsilon(u),
   \end{align*}
   where $\boldsymbol{U}_\varepsilon$ is defined in Lemma \ref{lem2.11}.
\end{rem}

\quad Next, we establish the Gronwall type inequality, and
it is obtained by the same argument for \cite[Lemma 4.2, Proposition 4.6]{LMS}.
Note that we use Lemma \ref{lemA.2} in the proof of the following lemma,
thus the convexity of domain $\Omega$ is required here.

\begin{lem}\label{lem4.5}
   For $n \in \mathbb{N}, t \in [0,1], h>0$, it holds 
   \begin{equation}\label{lem43}
      \frac{1}{2}\dhh\boldsymbol{A}_n^h(t) 
      + \dt\boldsymbol{V}_{\delta}(\rho_n^h(t)) 
      \leq -\lambda_\delta t \boldsymbol{A}_n^h(t),
   \end{equation}
   where 
   \begin{align*}
      \lambda_{\delta} 
      &\coloneq -\frac{1}{2\delta}\|\nabla\varphi\|_{L^\infty(\Omega)}^2 \sup_{r\geq0}|m_{\varepsilon}(r)m_{\varepsilon}^{\prime\prime}(r)|
       - \|\nabla^2\varphi\|_{L^\infty(\Omega)} \sup_{r\geq0}|m_{\varepsilon}^{\prime}(r)|\\
      &= -\frac{1}{2\delta}\|\nabla\varphi\|_{L^\infty(\Omega)}^2\frac{\alpha(1-\alpha)}{\varepsilon^{2(1-\alpha)}} 
       - \|\nabla^2\varphi\|_{L^\infty(\Omega)}\frac{\alpha}{\varepsilon^{1-\alpha}} \leq 0,
   \end{align*}
   and 
   $$\|\nabla^2\varphi\|_{L^\infty(\Omega)} 
   = \left\|\sum_{i,j=1}^d \left|\frac{\partial^2\varphi}{\partial_ix\partial_jx}\right|\right\|_{L^\infty(\Omega)}.$$
\end{lem}

\quad The following lemma is the estimate like the evolution variational inequality 
for the weighted Wasserstein distance with respect to the functional $\boldsymbol{U}_\varepsilon$.
The proof is based on \cite[Lemma 3.3, Lemma 4.2]{LMS} and similar to the proof of Lemma \ref{lem4.6}.
We can easily check the required properties of $\boldsymbol{U}_\varepsilon$ (see \cite[Definition 2]{LMS})
due to the propertty of the Neumann heat semigroup, Lemma \ref{lem2.10} and Lemma \ref{lem2.11}.

\begin{lem}\label{lem4.7}
   Let $u_\tau^k \in L^{p+1-\alpha}\cap\mathcal{P}(\Omega)$.
   Then it holds 
   \begin{equation*}
      \frac{1}{2}\limsup_{h\downarrow0}\frac{W_{m_\varepsilon}(e^{h\Delta}u_\tau^k, u_\tau^{k-1})^2 - W_{m_\varepsilon}(u_\tau^k, u_\tau^{k-1})^2}{h} 
      \leq \boldsymbol{U}_\varepsilon(u_\tau^{k-1}) - \boldsymbol{U}_\varepsilon(u_\tau^k),
   \end{equation*}
   where $e^{h\Delta}$ is the Neumann heat semigroup on $\Omega$.
\end{lem}

\quad The regularity of the minimizer $(u_\tau^k,v_\tau^k) \in X$
is not enough to get the weak formulation (the Euler--Lagrange equation).
Hence we need to improve the regularity of minimizers enough to converge to the weak formulation.

\begin{lem}\label{lem4.8}
   Let $(u_\tau^k,v_\tau^k) \in X$ be the minimizer of \eqref{mms}.
   Then $(u_\tau^k)^{\frac{p+1-\alpha}{2}} \in H^1(\Omega)$ 
   and $\Delta v_\tau^k - v_\tau^k + u_\tau^k \in L^2(\Omega)$.
   If $1+\alpha-2/d < p < 1+\alpha$, or $p=1+\alpha-2/d$ and $\chi>0$ is small enough,
   then $u_\tau^k \in L^2(\Omega)$.
   In addition, there exists a constant $C_0 = C_0(\alpha,p,d,\chi)>0$ such that 
   the following estimates hold
   \begin{align}
      \bullet\, &\frac{4p}{\chi(p+1-\alpha)^2}\|\nabla(u_\tau^k)^{\frac{p+1-\alpha}{2}}\|_{L^2(\Omega)}^2
       + \|\Delta v_\tau^k - v_\tau^k + u_\tau^k\|_{L^2(\Omega)}^2\notag\\
      &\leq \frac{2}{\tau\chi}(\boldsymbol{U}_\varepsilon(u_\tau^{k-1}) - \boldsymbol{U}_\varepsilon(u_\tau^k))
       + \frac{\|v_\tau^{k-1}\|_{H^1(\Omega)}^2 - \|v_\tau^k\|_{H^1(\Omega)}^2}{\tau}\notag\\
      &\quad + \|\nabla v_\tau^k\|_{L^2(\Omega)}^2 
       + C_0\left(\|u_\tau^k\|_{L^{p+1-\alpha}(\Omega)}^{p+1-\alpha} + \|u_\tau^k\|_{L^{p+1-\alpha}(\Omega)}^{\frac{p+1-\alpha}{p-\alpha}}\right),\label{eq74}\\
      \bullet\, 
      &\|u_\tau^k\|_{L^2(\Omega)}^2
       \leq \frac{4p}{\chi(p+1-\alpha)^2}\|\nabla (u_\tau^k)^{\frac{p+1-\alpha}{2}}\|_{L^2(\Omega)}^2 
        + C_0\|u_\tau^k\|_{L^{p+1-\alpha}(\Omega)}^{\frac{p+1-\alpha}{p-\alpha}}\label{eq75}\\
      &\hspace{9cm} \mathrm{if}\ 1+\alpha-\frac{2}{d} \leq p < 1+\alpha.\notag
   \end{align}
\end{lem}

\begin{proof} 
   The proof is almost same for \cite[Proposition 8]{BL}
   although the energy functional and the distance (the weighted Wasserstein distance) 
   are different from theirs.
   Hence, we only point out the key idea for proving $(u_\tau^k)^{\frac{p+1-\alpha}{2}} \in H^1(\Omega)$.
   Considering the Neumann heat equation with initial data $u_\tau^k$,
   we have formally
   \begin{align*}
      \frac{d}{dt}E(e^{t\Delta}u_\tau^k,v_\tau^k)
      &= \frac{p}{\chi(p-\alpha)}\int_\Omega (e^{t\Delta}u_\tau^k)^{p-\alpha}(\Delta e^{t\Delta}u_\tau^k)\, dx
       - \int_{\Omega} (\Delta e^{t\Delta}u_\tau^k) v_\tau^k\, dx\\
      &= -\frac{p}{\chi}\int_{\Omega} (e^{t\Delta}u_\tau^k)^{p-1-\alpha}|\nabla e^{t\Delta}u_\tau^k|^2\, dx
       + \int_{\Omega} \nabla e^{t\Delta}u_\tau^k\cdot \nabla v_\tau^k\, dx\\
      &= -\frac{4p}{\chi(p+1-\alpha)^2} \|\nabla(e^{t\Delta}u_\tau^k)^{\frac{p+1-\alpha}{2}}\|_{L^2(\Omega)}^2 
       + \int_{\Omega} \nabla e^{t\Delta}u_\tau^k\cdot \nabla v_\tau^k\, dx.
   \end{align*}
   By H\"older's inequality and Young's inequality, 
   the second term can be absorbed by the first term.
   Using \eqref{eq37} and Lemma \ref{lem4.7}, we obtain
   \begin{align*}
      \|\nabla(e^{t\Delta}u_\tau^k)^{\frac{p+1-\alpha}{2}}\|_{L^2(\Omega)}^2 
      \leq C\left(\frac{\boldsymbol{U}_\varepsilon(u_\tau^{k-1}) - \boldsymbol{U}_\varepsilon(u_\tau^k)}{\tau}
       + \|\nabla v_\tau^k\|_{L^2(\Omega)}^2\right) < \infty
       \quad \mathrm{for}\ t\in (0,1).
   \end{align*}
   Since $e^{t\Delta}u_\tau^k \in C([0,1];L^{p+1-\alpha}(\Omega))$, 
   that is, $(e^{t\Delta}u_\tau^k)^{\frac{p+1-\alpha}{2}} \in C([0,1];L^2(\Omega))$,
   combining this with the above boundedness for $t\in(0,1)$,
   we have $\nabla (u_\tau^k)^{\frac{p+1-\alpha}{2}} \in L^2(\Omega)$,
   then $(u_\tau^k)^{\frac{p+1-\alpha}{2}} \in H^1(\Omega)$.
\end{proof}

\begin{lem}\label{lem7.1}
   Under the same assumption in Lemma \ref{lem4.8}, 
   it holds $u_\tau^k \in W^{1,p+1-\alpha}(\Omega)$.
\end{lem}

\begin{proof}
   Set $y = (u_\tau^k)^{\frac{p+1-\alpha}{2}}$,
   then by Lemma \ref{lem4.8}, it holds $y\in H^1(\Omega)$.
   Since $2/(p+1-\alpha) \geq 1$, 
   for $f \in C_c^\infty(\Omega)$, we have
   \begin{align*}
      \int_{\Omega} u_\tau^k \nabla f\, dx
      = \int_{\Omega} y^{\frac{2}{p+1-\alpha}}\nabla f\, dx
      &= - \int_{\Omega} (\nabla y^{\frac{2}{p+1-\alpha}}) f\, dx\\
      &= - \int_{\Omega} \left(\frac{2}{p+1-\alpha}y^{\frac{2}{p+1-\alpha} - 1}\nabla w\right) f\, dx\\
      &= - \int_{\Omega} \left(\frac{2}{p+1-\alpha} (u_\tau^k)^{\frac{1+\alpha-p}{2}}\nabla (u_\tau^k)^{\frac{p+1-\alpha}{2}}\right)f\, dx.
   \end{align*}
   Here, by H\"older's inequality, it follows
   \begin{align*}
      &\int_{\Omega} |(u_\tau^k)^{\frac{1+\alpha-p}{2}}\nabla (u_\tau^k)^{\frac{p+1-\alpha}{2}}|^{p+1-\alpha}\, dx\notag\\
      &\leq \left(\int_{\Omega} (u_\tau^k)^{p+1-\alpha}\, dx\right)^{\frac{1+\alpha-p}{2}}
       \left(\int_{\Omega} |\nabla (u_\tau^k)^{\frac{p+1-\alpha}{2}}|^2\, dx\right)^{\frac{p+1-\alpha}{2}}.
   \end{align*}
   Since $u_\tau^k \in L^{p+1-\alpha}(\Omega)$ and $\nabla(u_\tau^k)^{\frac{p+1-\alpha}{2}} \in L^2(\Omega)$,
   we have $(u_\tau^k)^{\frac{1+\alpha-p}{2}}\nabla (u_\tau^k)^{\frac{p+1-\alpha}{2}} \in L^{p+1-\alpha}(\Omega)$. 
   This means that $\nabla u_\tau^k \in L^{p+1-\alpha}(\Omega)$, that is, $u_\tau^k \in W^{1,p+1-\alpha}(\Omega)$.
\end{proof}

\quad Next lemma is the estimate like the evolution variational inequality for the weighted Wasserstein distance
with respect to the functional $\boldsymbol{V}_\delta$. 
The proof is fundamentally based on \cite[Lemma 3.3, Proposition 4.6]{LMS}, 
but our minimizers have the lower regularity than theirs.
Thus we check the required properties of $\boldsymbol{V}_\delta$ (see \cite[Definition 2]{LMS}).
This lemma helps us to obtain the weak formulation (Lemma \ref{lem4.9}).

\begin{lem}\label{lem4.6}
   Let $1+\alpha-2/d \leq p \leq 1+\alpha$ and 
   $u_\tau^k \in L^{p+1-\alpha}\cap\mathcal{P}(\Omega)$ with $(u_\tau^k)^{\frac{p+1-\alpha}{2}} \in H^1(\Omega)$.
   Then it holds 
   \begin{equation*}
      \frac{1}{2}\limsup_{h\downarrow0}\frac{W_{m_\varepsilon}(S_hu_\tau^k, u_\tau^{k-1})^2 - W_{m_\varepsilon}(u_\tau^k, u_\tau^{k-1})^2}{h}
      + \frac{\lambda_\delta}{2}W_{m_\varepsilon}(u_\tau^k, u_\tau^{k-1})^2
      + \boldsymbol{V}_\delta(u_\tau^k) \leq \boldsymbol{V}_\delta(u_\tau^{k-1}),
   \end{equation*} 
      where $\lambda_\delta$ is defined in Lemma \ref{lem4.5}.
\end{lem}

\begin{proof}
   If $\varphi=0$, that is, $\lambda_\delta = 0$ 
   then we complete the proof by Lemma \ref{lem4.7}.
   Thus we can assume that $\varphi \not\equiv 0$ and then $\lambda_\delta\ne0$. 
   Since $\rho_n^0(t) = S_0\rho_n(t) = \rho_n(t)$ in $\Omega$, 
   the definition of $\boldsymbol{A}_n^h$ and \eqref{eq82} imply 
   \begin{equation}\label{wa2}
      W_{m_\varepsilon}(u_\tau^k, u_\tau^{k-1})^2 
      = \lim_{n\to\infty}a_n^d\int_0^1 \boldsymbol{A}_n^0(t)\, dt.
   \end{equation}
   Since $\|\rho_n^h(t)\|_{L^1(\Omega)} = \|\rho_n^0(t)\|_{L^1(\Omega)} = \|\rho_n(t)\|_{L^1(\Omega)} = a_n^{-d}$
   for all $n\in\mathbb{N}, h\in(0,1), t\in[0,1]$ and $\boldsymbol{U}_\varepsilon \geq0$,
   we have
   \begin{align*}
      \boldsymbol{V}_\delta(\rho_n^h(t)) 
      &= \int_{\Omega} \rho_n^h(t)\varphi\, dx + \delta\boldsymbol{U}_\varepsilon(\rho_n^h(t))\\
      &\geq - \|\varphi\|_{L^\infty(\Omega)}\|\rho_n^h(t)\|_{L^1(\Omega)}\\
      &= -a_n^{-d}\|\varphi\|_{L^\infty(\Omega)}. 
   \end{align*}
   Since $a_n \to 1$ as $n\to\infty$, that is, $\{a_n\}_n$ is bounded,  
   there exists a constant $L>0$ such that
   $\boldsymbol{V}_{\delta}(\rho_n^h(t)) \geq -L$ for all $n\in\mathbb{N}, h\in(0,1), t\in[0,1]$.

   Multiplying \eqref{lem43} by $e^{2\lambda_\delta th}$ 
   and integrating with respect to $t\in[0,1]$,
   further using integration by parts, we have
   \begin{align*}
      \frac{1}{2}\dhh\int_0^1 e^{2\lambda_\delta th}\boldsymbol{A}_n^h(t)\, dt 
      &\leq - \int_0^1 e^{2\lambda_\delta th}\dt(\boldsymbol{V}_\delta(\rho_n^h(t)) + L)\, dt\\
      &= \boldsymbol{V}_{\delta}(\rho_n^h(0)) + L - e^{2\lambda_\delta h}(\boldsymbol{V}_{\delta}(\rho_n^h(1)) + L)\\
      &\quad + 2\lambda_\delta h\int_0^1 e^{2\lambda_\delta th}(\boldsymbol{V}_{n,\delta}(\rho_n^h(t)) + L)\, dt.
   \end{align*}
   Since $\boldsymbol{V}_{\delta}(\rho_n^h(t)) + L \geq 0$ for any $t \in [0,1]$ 
   and $\lambda_\delta < 0$, it follows
   \begin{equation*}
      \frac{1}{2}\dhh\int_0^1 e^{2\lambda_\delta th}\boldsymbol{A}_n^h(t)\, dt 
      \leq \boldsymbol{V}_{n,\delta}(\rho_n^h(0)) + L - e^{2\lambda_\delta h}(\boldsymbol{V}_{n,\delta}(\rho_n^h(1)) + L).
   \end{equation*}
   Observe that $\rho_n^h(0) = \rho_n(0)$, 
   integrating over $(0,h)$, we have
   \begin{equation*}
      \frac{1}{2}\int_0^1 e^{2\lambda_\delta th}\boldsymbol{A}_n^h(t)\, dt 
      \leq \frac{1}{2}\int_0^1 \boldsymbol{A}_n^0\, dt 
      + h(\boldsymbol{V}_{\delta}(\rho_n(0)) + L) 
      - \int_0^h e^{2\lambda_\delta s}(\boldsymbol{V}_{\delta}(\rho_n^s(1)) + L)\, ds.
   \end{equation*}
   Here the function $s \mapsto \boldsymbol{V}_\delta(\rho_n^s(1)) + L$ is nonincreasing.
   Indeed, calculating the derivative and using \eqref{eq46} with $t=1$ and 
   $U_\varepsilon^{\prime\prime}(r)m_\varepsilon(r) = 1$ for $r\geq0$, 
   we obtain 
   \begin{align*}
      \frac{d}{ds}\boldsymbol{V}_{\delta}(\rho_n^s(1)) 
      &= \frac{d}{ds}\left(\int_{\Omega} \rho_n^s(1)\varphi\, dx + \delta\int_{\Omega} U_\varepsilon(\rho_n^s(1))\, dx\right)\\
      &= \int_{\Omega} (\nabla\cdot(m_\varepsilon(\rho_n^s(1))\nabla\varphi) + \delta\Delta\rho_n^s(1))\varphi\, dx\\
      &\quad + \delta\int_{\Omega} U_\varepsilon^\prime(\rho_n^s(1))(\nabla\cdot(m_\varepsilon(\rho_n^s(1))\nabla\varphi) + \delta\Delta\rho_n^s(1))\, dx\\
      &= - \int_{\Omega} (m_\varepsilon(\rho_n^s(1))\nabla\varphi + \delta\nabla\rho_n^s(1))\cdot\nabla\varphi\, dx\\
      &\quad - \delta\int_{\Omega} U_\varepsilon^{\prime\prime}(\rho_n^s(1))\nabla\rho_n^s(1)\cdot(m_\varepsilon(\rho_n^s(1))\nabla\varphi + \delta\nabla\rho_n^s(1))\, dx\\
      &= - \int_{\Omega} m_\varepsilon(\rho_n^s(1))|\nabla\varphi|^2\, dx
       - \delta\int_{\Omega} \nabla\rho_n^s(1)\cdot\nabla\varphi\, dx\\
      &\quad - \delta\int_{\Omega} \nabla\rho_n^s(1)\cdot\nabla\varphi\, dx
       - \delta^2\int_{\Omega} \frac{|\nabla\rho_n^s(1)|^2}{m(\rho_n^s(1))}\, dx\\
      &= - \int_{\Omega} \left|m_\varepsilon(\rho_n^s(1))^{\frac{1}{2}}\nabla\varphi + \frac{\delta\nabla\rho_n^s(1)}{m_\varepsilon(\rho_n^s(1))^{\frac{1}{2}}}\right|^2\, dx \leq 0.
   \end{align*}
   Thus it follows  
   \begin{equation}\label{eq7}
      \frac{1}{2}\int_0^1 e^{2\lambda_\delta th} \boldsymbol{A}_n^h(t)\, dt 
      \leq \frac{1}{2}\int_0^1 \boldsymbol{A}_n^0(t)\, dt 
      + h(\boldsymbol{V}_{\delta}(\rho_n(0)) + L)
      - \frac{1 - e^{2\lambda_\delta h}}{-2\lambda_\delta}(\boldsymbol{V}_{\delta}(\rho_n^h(1)) + L).
   \end{equation}
   On the other hand, 
   for a decreasing function $\theta \in C^1([0,1])$ with $\theta > 0$,
   we define a increasing function
   \begin{align*}
      \tilde{\theta}(t) 
       \coloneq \left[\int_0^1 \frac{dz}{\theta(z)}\right]^{-1}\int_0^t \frac{dz}{\theta(z)}
       \quad \mathrm{for}\ t\in[0,1],
   \end{align*}
   and denote $\tilde{\theta}^{-1}$ an inverse function of $\tilde{\theta}$,
   that is,
   $\tilde{\theta}\circ\tilde{\theta}^{-1}(t) = \tilde{\theta}^{-1}\circ\tilde{\theta}(t) = t$ for $t \in [0,1]$.
   Then the pair 
   $\left(\rho_n^h(\cdot,\tilde{\theta}^{-1}(\cdot)),m_\varepsilon(\rho_n^h(\cdot,\tilde{\theta}^{-1}(\cdot)))\nabla\{(\tilde{\theta}^{-1})^\prime(\cdot)\phi_n^h(\cdot,\tilde{\theta}^{-1}(\cdot))\}\right)$
   belongs to $CE(0,1;\rho_n(0),\rho_n^h(1))$.
   Indeed, by \eqref{ee2}, we have
   \begin{align*}
      \dt[\rho_n^h(x, \tilde{\theta}^{-1}(t))]
      &= (\dt\rho_n^h)(x, \tilde{\theta}^{-1}(t))(\tilde{\theta}^{-1})^\prime(t)\\
      &= \left[-\nabla\cdot\left\{m_\varepsilon\left(\rho_n^h(x,\tilde{\theta}^{-1}(t))\right)\nabla\phi_n^h(x, \tilde{\theta}^{-1}(t))\right\}\right](\tilde{\theta}^{-1})^\prime(t)\\
      &= -\nabla\cdot\left\{m_\varepsilon\left(\rho_n^h(x,\tilde{\theta}^{-1}(t))\right)\nabla\left((\tilde{\theta}^{-1})^\prime(t)\phi_n^h(x, \tilde{\theta}^{-1}(t))\right)\right\},
   \end{align*}
   hence they satisfy the continuity equation.
   In addition, since $\tilde{\theta}(0) = 0$ and $\tilde{\theta}(1) = 1$, 
   it also holds $\tilde{\theta}^{-1}(0) = 0$ and $\tilde{\theta}^{-1}(1) = 1$,
   thus we have
   \begin{align*}
      &\rho_n^h(x, \tilde{\theta}^{-1}(0)) = \rho_n^h(x, 0) = \rho_n(x,0),
      \quad \rho_n^h(x, \tilde{\theta}^{-1}(1)) = \rho_n^h(x, 1)\quad \mathrm{for}\ x\in \Omega.
   \end{align*}
   Note that $\rho_n^h(1) = S_h\rho_n(1)$.
   By the definition of the weighted Wasserstein distance (see Definition \ref{dfn2.4})
   and the change of variables, it follows
   \begin{flalign*}
      W_{m_\varepsilon}(\rho_n(0), S_h\rho_n(1))^2
      &\leq \int_0^1\int_{\Omega} \left((\tilde{\theta}^{-1})^\prime(z)\right)^2m_\varepsilon\left(\rho_n^h(x,\tilde{\theta}^{-1}(z))\right)|\nabla\phi_n^h(x, \tilde{\theta}^{-1}(z))|^2\, dx\,dz\\
      &= \int_0^1\int_{\Omega} \left((\tilde{\theta}^{-1})^\prime(\tilde{\theta}(t))\right)m_\varepsilon(\rho_n^h(x, t))|\nabla\phi_n^h(x, t)|^2\, dx\,dt\\
      &= \int_0^1\frac{dr}{\theta(r)}\int_0^1 \theta(t)\boldsymbol{A}_n^h(t)\, dt,
   \end{flalign*}
   where we used
   \begin{equation*}
      (\tilde{\theta}^{-1})^\prime(\tilde{\theta}(t)) 
      = \frac{1}{\tilde{\theta}^\prime(t)} 
      = \int_0^1\frac{dr}{\theta(r)}\theta(t).
   \end{equation*}
   Hence, choosing $\theta(t) = e^{2\lambda_\delta th}$, we obtain
   \begin{equation*}
      W_{m_\varepsilon}(\rho_n(0), S_h\rho_n(1))^2 
      \leq \frac{e^{-2\lambda_\delta h} - 1}{-2\lambda_\delta h}\int_0^1 e^{2\lambda_\delta th}\boldsymbol{A}_n^h(t)\, dt.
   \end{equation*}
   Combining the above with \eqref{eq7}, we have
   \begin{flalign}\label{eq8}
      &\frac{-\lambda_\delta h}{e^{-2\lambda_\delta h} - 1}W_{m_\varepsilon}(\rho_n(0), S_h\rho_n(1))^2\notag\\
      &\leq \frac{1}{2}\int_0^1 \boldsymbol{A}_n^0(t)\, dt
       + h(\boldsymbol{V}_{\delta}(\rho_n(0)) + L)
       - \frac{1 - e^{2\lambda_\delta h}}{-2\lambda_\delta}(\boldsymbol{V}_{\delta}(S_{\delta,h}\rho_n(1)) + L).
   \end{flalign}
   Let $S_hu_\tau^k$ be a solution in Propositions \ref{prop4.1} and \ref{prop4.2} with $w_0 = u_\tau^k$.
   We will show 
   \begin{equation*}
      S_h\rho_n(1) \to S_hu_\tau^k\ \mathrm{in}\ L^2(\Omega)\quad \mathrm{as}\ n \to \infty
      \ \mathrm{for}\ h\in(0,1). 
   \end{equation*}
   Since $y \coloneq S_h\rho_n(1) - S_hu_\tau^k$ satisfies the following equation
   \begin{align*}
      \begin{cases}
         \dhh y = \delta\Delta y + \nabla\cdot[(m_\varepsilon(S_h\rho_n(1)) - m_\varepsilon(S_hu_\tau^k))\nabla\varphi]
         &\mathrm{in}\ \Omega\times(0,1),\\
         \nabla y\cdot\boldsymbol{n} = 0
         &\mathrm{on}\ \partial\Omega\times(0,1),\\
         y(0) = \rho_n(1) - u_\tau^k
         &\mathrm{in}\ L^2(\Omega),
      \end{cases}
   \end{align*}
   multiplying the first equation by $y$ and integrating in $\Omega$, we have for $h\in (0,1)$
   \begin{align*}
      \int_{\Omega} (\dhh y)y\, dx 
      = \int_{\Omega} [\delta \Delta y + \nabla\cdot\{(m_\varepsilon(S_h\rho_n(1)) - m_\varepsilon(S_hu_\tau^k))\nabla\varphi\}]y\, dx.
   \end{align*}
   Since $\nabla y\cdot\boldsymbol{n} = 0$ and $\nabla\varphi\cdot\boldsymbol{n} = 0$ on $\partial\Omega\times(0,1)$,
   we infer from integration by parts that
   \begin{align*}
      \frac{1}{2}\dhh\|y(h)\|_{L^2(\Omega)}^2 
      &= -\delta\|\nabla y\|_{L^2(\Omega)}^2 
       - \int_{\Omega} (m_\varepsilon(S_h\rho_n(1)) - m_\varepsilon(S_hu_\tau^k))\nabla\varphi\cdot\nabla y\, dx\\
      &\leq -\delta\|\nabla y\|_{L^2(\Omega)}^2
       + \|\nabla\varphi\|_{L^\infty(\Omega)}\|\nabla y\|_{L^2(\Omega)}\|m_\varepsilon(S_h\rho_n(1)) - m_\varepsilon(S_hu_\tau^k)\|_{L^2(\Omega)}.
   \end{align*}
   Using the Lipschitz continuity of $m_\varepsilon$ and Young's inequality,
   we obtain
   \begin{align*}
      \frac{1}{2}\dhh\|y(h)\|_{L^2(\Omega)}^2 
      &\leq -\delta\|\nabla y\|_{L^2(\Omega)}^2 
       + \|\nabla\varphi\|_{L^\infty(\Omega)}\|\nabla y\|_{L^2(\Omega)}\frac{\alpha}{\varepsilon^{1-\alpha}}\|S_h\rho_n(1) - S_hu_\tau^k\|_{L^2(\Omega)}\\
      &\leq C\|S_h\rho_n(1) - S_hu_\tau^k\|_{L^2(\Omega)}^2
       = C\|y(h)\|_{L^2(\Omega)}^2,
   \end{align*}
   where $C = C(\alpha,\varepsilon,\delta,\varphi)$ is a constant.
   By Gronwall's lemma, it follows
   \begin{align*}
      \|y(h)\|_{L^2(\Omega)}^2
      \leq e^{2C}\|y(0)\|_{L^2(\Omega)}^2
      = e^{2C}\|\rho_n(1) - u_\tau^k\|_{L^2(\Omega)}^2
      \quad \mathrm{for}\ h\in(0,1).
   \end{align*}
   Since $\rho_n(1) \to u_\tau^k$ in $L^2(\Omega)$ as $n\to\infty$ (see Lemma \ref{lem2.10}),
   we conclude that $S_h\rho_n(1)$ converges to $S_hu_\tau^k$ in $L^2(\Omega)$ as $n\to\infty$ for $h\in(0,1)$.
   This convergence also implies that $S_h\rho_n(1)$ converges to $S_hu_\tau^k$ weakly* in $\mathcal{M}_{loc}^+(\Rd)$
   as $n\to\infty$.
   Hence, combining this with the lower semicontinuity of $\boldsymbol{U}_\varepsilon$ (Lemma \ref{lem2.11}),
   we see
   \begin{align*}
      \boldsymbol{V}_\delta(S_hu_\tau^k) \leq \liminf_{n\to\infty}\boldsymbol{V}_\delta(S_h\rho_n(1)).
   \end{align*}
   Moreover by Lemma \ref{lem2.10}, 
   we also have $\rho_n(0) \to u_\tau^{k-1}$ in $L^1\cap L^{2-\alpha}(\Omega)$ as $n\to\infty$,
   which thus yields
    \begin{align*}
      &\lim_{n\to\infty}\boldsymbol{V}_\delta(\rho_n(0)) 
      = \boldsymbol{V}_\delta(u_\tau^{k-1}).
    \end{align*}
   Therefore by Lemma \ref{lem2.7} and \eqref{wa2}, 
    letting $n \to \infty$ in \eqref{eq8}, we obtain
    \begin{flalign*}
      &\frac{-\lambda_\delta h}{e^{-2\lambda_\delta h} - 1}W_{m_\varepsilon}(u_\tau^{k-1}, S_hu_\tau^k)^2\\
      &\leq \frac{1}{2}W_{m_\varepsilon}(u_\tau^{k-1}, u_\tau^k)^2 
       + h(\boldsymbol{V}_\delta(u_\tau^{k-1}) + L)
       - \frac{1 - e^{2\lambda_\delta h}}{-2\lambda_\delta}(\boldsymbol{V}_\delta(S_hu_\tau^k) + L),
    \end{flalign*}
    then
    \begin{flalign*} 
      &\frac{-\lambda_\delta h}{e^{-2\lambda_\delta h} - 1}\frac{W_{m_\varepsilon}(S_hu_\tau^k, u_\tau^{k-1})^2 - W_{m_\varepsilon}(u_\tau^k, u_\tau^{k-1})^2}{h}
       + \frac{1}{h}\left(\frac{-\lambda_\delta h}{e^{-2\lambda_\delta h} - 1} - \frac{1}{2}\right)W_{m_\varepsilon}(u_\tau^k, u_\tau^{k-1})^2\\
      &\leq \boldsymbol{V}_\delta(u_\tau^{k-1}) + L
       - \frac{1 - e^{2\lambda_\delta h}}{-2\lambda_\delta h}(\boldsymbol{V}_\delta(S_hu_\tau^k) + L).
    \end{flalign*}
    Since
    \begin{align*}
      &\lim_{h\downarrow0}\frac{-\lambda_\delta h}{e^{-2\lambda_\delta h} - 1} = \frac{1}{2},
      \quad \lim_{h\downarrow0}\frac{1 - e^{2\lambda_\delta h}}{-2\lambda_\delta h} = 1,
      \quad \lim_{h\downarrow0}\frac{1}{h}\left(\frac{-\lambda_\delta h}{e^{-2\lambda_\delta h} - 1} - \frac{1}{2}\right) = \frac{\lambda_\delta}{2},
    \end{align*}
    and 
    $$\boldsymbol{V}_\delta(u_\tau^k) \leq \liminf_{h\downarrow0}\boldsymbol{V}_\delta(S_hu_\tau^k)$$
    due to $S_hu_\tau^k \to u_\tau^k$ in $L^1(\Omega)$ as $h \to 0$ (Proposition \ref{prop4.1}),
    we conclude that 
    \begin{equation*}
      \frac{1}{2}\limsup_{h\downarrow0}\frac{W_{m_\varepsilon}(S_hu_\tau^k, u_\tau^{k-1})^2 - W_{m_\varepsilon}(u_\tau^k, u_\tau^{k-1})^2}{h}
      + \frac{\lambda_\delta}{2}W_{m_\varepsilon}(u_\tau^k, u_\tau^{k-1})^2
      \leq \boldsymbol{V}_\delta(u_\tau^{k-1}) - \boldsymbol{V}_\delta(u_\tau^k).
    \end{equation*}
    The proof is completed.
\end{proof}

\subsection{A discrete type of weak formulations}

\quad First, we obtain the Euler--Lagrange equation of 
the second equation of the Keller--Segel system \eqref{peks}.
Moreover we see that $v_\tau^k$ satisfies the Neumann boundary condition.

\begin{lem}\label{lem4.10}
   Let $p\geq1+\alpha-2/d$ and assume that $\chi>0$ is small enough if $p=1+\alpha-2/d$. 
   Let $v_\tau^{k-1} \in H^1(\Omega)$ and $(u_\tau^k,v_\tau^k) \in X$ be a minimizer of \eqref{mms}.
   Then it holds 
   \begin{align*}
      \int_{\Omega} \frac{v_\tau^k - v_\tau^{k-1}}{\tau}\zeta\, dx + \int_{\Omega} (\nabla v_\tau^k\cdot\nabla\zeta + v_\tau^k\zeta - u_\tau^k\zeta)\, dx = 0
      \quad \mathrm{for\ all}\ \zeta \in H^1(\Omega).
   \end{align*}
   In additon, if $\Delta v_\tau^k \in L^2(\Omega)$ then it holds that 
   $\nabla v_\tau^k\cdot\boldsymbol{n} = 0$ on $\partial\Omega$ in the sense of distributions.
\end{lem}

\begin{proof}
   Let $\zeta \in H^1(\Omega)$ and $a>0$. Note that $v_\tau^k + a\zeta \in H^1(\Omega)$.
   By \eqref{eq37} with $(\tilde{u},\tilde{v}) = (u_\tau^k,v_\tau^k + a\zeta)$, it follows
   \begin{align*}
      \frac{1}{2\tau}\|v_\tau^k - v_\tau^{k-1}\|_{L^2(\Omega)}^2 + E(u_\tau^k,v_\tau^k)
      \leq \frac{1}{2\tau}\|v_\tau^k + a\zeta - v_\tau^{k-1}\|_{L^2(\Omega)}^2 + E(u_\tau^k,v_\tau^k + a\zeta),
   \end{align*}
   then
   \begin{align*}
      0&\leq \frac{1}{2\tau}\int_{\Omega} (|v_\tau^k + a\zeta - v_\tau^{k-1}|^2 - |v_\tau^k - v_\tau^{k-1}|^2)\, dx\\
       &\quad + \frac{1}{2}\int_{\Omega} (|\nabla v_\tau^k + a\nabla\zeta|^2 - |\nabla v_\tau^k|^2)\, dx
         + \frac{1}{2}\int_{\Omega} (|v_\tau^k + a\zeta|^2 - |v_\tau^k|^2)\, dx
         - a\int_{\Omega} u_\tau^k\zeta\, dx.
   \end{align*}
   Dividing by $a>0$ and letting $a\to0$, by simple calculations, we have
   \begin{align*}
      0\leq \int_{\Omega} \frac{v_\tau^k - v_\tau^{k-1}}{\tau}\zeta\, dx
        + \int_{\Omega} \nabla v_\tau^k\cdot\nabla\zeta\, dx
        + \int_{\Omega} v_\tau^k\zeta\, dx
        - \int_{\Omega} u_\tau^k\zeta\, dx.
   \end{align*}
   Replacing $\zeta$ by $-\zeta$, we obtain the opposite inequality.
   
   Assume that $\Delta v_\tau^k \in L^2(\Omega)$.
   Letting $\psi \in C_c^\infty(\Omega)$ be arbitary, we have 
   \begin{align*}
      \int_{\Omega} \left(\frac{v_\tau^k - v_\tau^{k-1}}{\tau} -\Delta v_\tau^k + v_\tau^k - u_\tau^k\right)\psi\, dx = 0.
   \end{align*}
   Hence it follows
   \begin{align*}
      \frac{v_\tau^k - v_\tau^{k-1}}{\tau} -\Delta v_\tau^k + v_\tau^k - u_\tau^k = 0\quad \mathrm{a.e.\ in}\ \Omega.
   \end{align*}
   Then, for all $\zeta \in H^1(\Omega)$, we conclude that
   \begin{align*}
      0&= \int_{\Omega} \left(\frac{v_\tau^k - v_\tau^{k-1}}{\tau} -\Delta v_\tau^k + v_\tau^k - u_\tau^k\right)\zeta\, dx
         + \int_{\partial\Omega} \nabla v_\tau^k\cdot\boldsymbol{n}\zeta\, dS\\
       &= \int_{\partial\Omega} \nabla v_\tau^k\cdot\boldsymbol{n}\zeta\, dS,
   \end{align*}
   then $\nabla v_\tau^k\cdot\boldsymbol{n} = 0$ on $\partial\Omega$ in the sense of distributions.
\end{proof}

\begin{cor}\label{cor4.12}
   Let $(u_\tau^k,v_\tau^k) \in X$ be a minimizer of \eqref{mms} with $\Delta v_\tau^k \in L^2(\Omega)$ 
   and $u_\tau^k \in L^2(\Omega)$.
   Then $v_\tau^k \in H^2(\Omega)$ and there exists a constant $C>0$ such that 
   \begin{align*}
      \|v_\tau^k\|_{H^2(\Omega)}^2 \leq C(\|\Delta v_\tau^k\|_{L^2(\Omega)}^2 + \|v_\tau^k\|_{H^1(\Omega)}^2).
   \end{align*}
\end{cor}

\begin{proof}
   By Lemma \ref{lem4.8}, Lemma \ref{lem4.10} and the elliptic regularity theorem, 
   we can complete the proof immediately.
\end{proof}

\quad Next, we obtain the inequality like the Euler--Lagrange equation of 
the first equation of the Keller--Segel system \eqref{peks}.

\begin{lem}\label{lem4.9}
   Let $1+\alpha -2/d \leq p \leq 1+\alpha$ and assume that $\chi>0$ is small enough if $p=1+\alpha-2/d$. 
   Let $(u_\tau^{k-1},v_\tau^{k-1}) \in X$ and $(u_\tau^k,v_\tau^k) \in X$ be a minimizer of \eqref{mms}.
   Then 
   \begin{align*}
      \frac{\boldsymbol{V}_\delta(u_\tau^k) - \boldsymbol{V}_\delta(u_\tau^{k-1})}{\chi}
      &\leq \tau\left[\frac{p}{\chi(p-\alpha)}\int_{\Omega} (u_\tau^k)^{p-\alpha} \nabla\cdot(m_\varepsilon(u_\tau^k)\nabla\varphi)\, dx
       + \int_{\Omega} m_\varepsilon(u_\tau^k)\nabla v_\tau^k\cdot\nabla\varphi\, dx\right]\\
      &\quad -\tau\lambda_\delta(E(u_\tau^{k-1},v_\tau^{k-1}) - E(u_\tau^k,v_\tau^k))
       - \tau\delta\int_{\Omega} (\Delta v_\tau^k) u_\tau^k\, dx.
   \end{align*}
\end{lem}

\begin{proof}
   In this proof, we write $\|\cdot\|_{L^q(\Omega)} = \|\cdot\|_q$ for $q\in[1,\infty]$.
   Let $S_tu_\tau^k$ be a nonnegative solution to \eqref{eq72} with $w_0=u_\tau^k$.
   Note that $S_tu_\tau^k \in C((0,T];W^{2,p+1-\alpha}(\Omega))\cap C^1((0,T];L^{p+1-\alpha}(\Omega))$
   and $S_tu_\tau^k \to u_\tau^k$ in $L^1\cap L^2\cap W^{1,p+1-\alpha}(\Omega)$ as $t \to 0$.
   Then for $t>0$, using \eqref{eq72} and integration by parts, we have 
   \begin{align*}
      \frac{d}{dt}E(S_tu_\tau^k,v_\tau^k) 
      &= \frac{p}{\chi(p-\alpha)}\int_{\Omega} (S_tu_\tau^k)^{p-\alpha}(\dt S_tu_\tau^k)\, dx
       - \int_{\Omega} v_\tau^k (\dt S_tu_\tau^k)\, dx\\
      &= \frac{p}{\chi(p-\alpha)}\int_{\Omega} (S_tu_\tau^k)^{p-\alpha}[\nabla\cdot(m_\varepsilon(S_tu_\tau^k)\nabla\varphi) + \delta\Delta S_tu_\tau^k]\, dx\\
      &\quad - \int_{\Omega} v_\tau^k [\nabla\cdot(m_\varepsilon(S_tu_\tau^k)\nabla\varphi) + \delta\Delta S_tu_\tau^k]\, dx\\
      &= \frac{p}{\chi(p-\alpha)}\int_{\Omega} (S_tu_\tau^k)^{p-\alpha}\nabla\cdot (m_\varepsilon(S_tu_\tau^k)\nabla\varphi)\, dx
       + \frac{p\delta}{\chi(p-\alpha)}\int_{\Omega} (S_tu_\tau^k)^{p-\alpha}\Delta S_tu_\tau^k\, dx\\
      &\quad + \int_{\Omega} m_\varepsilon(S_tu_\tau^k)\nabla v_\tau^k\cdot\nabla\varphi\, dx
       - \delta\int_{\Omega} (\Delta v_\tau^k) S_tu_\tau^k\, dx.
   \end{align*}
   By integrating over $(0,t)$, it follows
   \begin{align}\label{eq56}
      &E(S_tu_\tau^k,v_\tau^k) - E(u_\tau^k,v_\tau^k)\notag\\
      &\leq \frac{p}{\chi(p-\alpha)}\int_0^t\int_{\Omega} (S_tu_\tau^k)^{p-\alpha}\nabla\cdot (m_\varepsilon(S_tu_\tau^k)\nabla\varphi)\, dx\,dt
       + \frac{p\delta}{\chi(p-\alpha)}\int_0^t\int_{\Omega} (S_tu_\tau^k)^{p-\alpha}\Delta S_tu_\tau^k\, dx\, dt\notag\\
      &\quad +\int_0^t\int_{\Omega} m_\varepsilon(S_tu_\tau^k)\nabla v_\tau^k\cdot \nabla\varphi\, dx\,dt
       - \delta\int_0^t\int_{\Omega} (\Delta v_\tau^k) S_tu_\tau^k\, dx\, dt.
   \end{align}
   Since $S_tu_\tau^k \to u_\tau^k$ in $L^2(\Omega)$ as $t\to0$ 
   and $v_\tau^k \in H^2(\Omega)$,
   it immediately follows that 
   \begin{align*}
      \int_{\Omega} (\Delta v_\tau^k) S_tu_\tau^k\, dx 
       \to \int_{\Omega} (\Delta v_\tau^k) u_\tau^k\, dx
      \quad \mathrm{as}\ t\to0.
   \end{align*}
   We will show 
   \begin{align}
      &\int_{\Omega} (S_tu_\tau^k)^{p-\alpha}\nabla\cdot (m_\varepsilon(S_tu_\tau^k)\nabla\varphi)\, dx 
       \to \int_{\Omega} (u_\tau^k)^{p-\alpha}\nabla\cdot (m_\varepsilon(u_\tau^k)\nabla\varphi)\, dx\quad
       \mathrm{as}\ t \to 0,\label{eq57}\\
      &\int_{\Omega} m_\varepsilon(S_tu_\tau^k)\nabla v_\tau^k\cdot\nabla\varphi\, dx 
       \to \int_{\Omega} m_\varepsilon(u_\tau^k)\nabla v_\tau^k\cdot\nabla\varphi\, dx\quad
       \mathrm{as}\ t \to 0,\label{eq58}\\
      &\int_{\Omega} (S_tu_\tau^k)^{p-\alpha}\Delta S_tu_\tau^k\, dx \leq 0 \quad
       \mathrm{for}\ t>0.\label{eq91}
   \end{align}
   First, note that
   \begin{align*}
      &\bullet \nabla\cdot(m_\varepsilon(S_tu_\tau^k)\nabla\varphi)
      = \frac{\alpha\nabla S_tu_\tau^k\cdot\nabla\varphi}{(S_tu_\tau^k+\varepsilon)^{1-\alpha}}
      + (S_tu_\tau^k+\varepsilon)^\alpha\Delta\varphi,\\
      &\bullet \|(S_tu_\tau^k + \varepsilon)^\alpha\Delta\varphi\|_{L^{p+1-\alpha}(\Omega)}
      \leq \tilde{C}(\varphi,\alpha,\varepsilon)(\|S_tu_\tau^k\|_{L^{p+1-\alpha}}^\alpha + 1).
   \end{align*}
   Then, 
   \begin{align*}
      &\left|\int_{\Omega} (S_tu_\tau^k)^{p-\alpha}\nabla\cdot(m_\varepsilon(S_tu_\tau^k)\nabla\varphi)\, dx
       - \int_{\Omega} (u_\tau^k)^{p-\alpha}\nabla\cdot(m_\varepsilon(u_\tau^k)\nabla\varphi)\, dx\right|\\
      &\leq \int_{\Omega} |(S_tu_\tau^k)^{p-\alpha}-(u_\tau^k)^{p-\alpha}|
       \left|\frac{\alpha\nabla S_tu_\tau^k\cdot\nabla\varphi}{(S_tu_\tau^k+\varepsilon)^{1-\alpha}} 
       + (S_tu_\tau^k+\varepsilon)^\alpha\Delta\varphi\right|\, dx\\
      &\quad + \int_{\Omega} (u_\tau^k)^{p-\alpha}
       \left|\frac{\alpha\nabla S_tu_\tau^k\cdot\nabla\varphi}{(S_tu_\tau^k+\varepsilon)^{1-\alpha}} - \frac{\alpha\nabla u_\tau^k\cdot\nabla\varphi}{(u_\tau^k+\varepsilon)^{1-\alpha}}\right|\, dx\\
      &\quad + \int_{\Omega} (u_\tau^k)^{p-\alpha}
       |(S_tu_\tau^k+\varepsilon)^{\alpha}-(u_\tau^k+\varepsilon)^{\alpha}||\Delta\varphi|\, dx\\
      &\eqcolon \mathrm{I}_1 + \mathrm{I}_2 + \mathrm{I}_3.
   \end{align*}
   Since $|(S_tu_\tau^k)^{p-\alpha}-(u_\tau^k)^{p-\alpha}| \leq |S_tu_\tau^k-u_\tau^k|^{p-\alpha}$,
   by H\"older's inequality, we have
   \begin{align*}
      \mathrm{I}_1
      \leq \|S_tu_\tau^k - u_\tau^k\|_{L^{p+1-\alpha}(\Omega)}^{p-\alpha}
       C(\varphi,\alpha, \varepsilon)(\|\nabla S_tu_\tau^k\|_{L^{p+1-\alpha}(\Omega)} + \|S_tu_\tau^k\|_{L^{p+1-\alpha}(\Omega)}^\alpha + 1).
   \end{align*}
   Since $S_tu_\tau^k \to u_\tau^k$ in $L^{p+1-\alpha}(\Omega)$ as $t \to 0$
   and $\sup_{t \in [0,1]}\|S_tu_\tau^k\|_{W^{1,p+1-\alpha}(\Omega)} < \infty$,
   we obtain $\mathrm{I}_1 \to 0$ as $t \to 0$.
   Further, by H\"older's inequality and the mean value theorem, we have
   \begin{align*}
      \mathrm{I}_2
      \leq \|u_\tau^k\|_{L^{p+1-\alpha}(\Omega)}^{p-\alpha}C(\varphi,\alpha,\varepsilon)
       \|S_tu_\tau^k - u_\tau^k\|_{W^{1,p+1-\alpha}(\Omega)}.
   \end{align*}
   Since $S_tu_\tau^k \to u_\tau^k$ in $W^{1,p+1-\alpha}(\Omega)$ as $t \to 0$,
   we obtain $\mathrm{I}_2 \to 0$ as $t \to 0$.
   Similarly we have
   \begin{align*}
      \mathrm{I}_3
      \leq \|u_\tau^k\|_{L^{p+1-\alpha}(\Omega)}^{p-\alpha}C(\varphi,\alpha,\varepsilon)
       \|S_tu_\tau^k - u_\tau^k\|_{L^{p+1-\alpha}(\Omega)}^\alpha
      \to 0 \quad \mathrm{as}\ t \to 0.
   \end{align*}
   Thus \eqref{eq57} holds.
   Secondary, since $S_tu_\tau^k \to u_\tau^k$ in $L^{p+1-\alpha}(\Omega)$ as $t\to0$,
   we infer from the Lipschitz continuity of $m_\varepsilon$ and H\"older's ineqality that
   \begin{align*}
      &\left|\int_{\Omega} m_\varepsilon(S_tu_\tau^k)\nabla v_\tau^k\cdot\nabla\varphi\, dx 
       - \int_{\Omega} m_\varepsilon(u_\tau^k)\nabla v_\tau^k\cdot\nabla\varphi\, dx\right|\\
      &\leq \int_{\Omega} |m_\varepsilon(S_tu_\tau^k) - m_\varepsilon(u_\tau^k)||\nabla v_\tau^k||\nabla\varphi|\, dx\\
      &\leq \frac{\alpha\|\nabla\varphi\|_\infty}{\varepsilon^{1-\alpha}}\|S_tu_\tau^k - u_\tau^k\|_{p+1-\alpha}\|\nabla v_\tau^k\|_{\frac{p+1-\alpha}{p-\alpha}}
      \to 0\quad \mathrm{as}\ t \to 0,
   \end{align*}
   which yields \eqref{eq58}.
   Finally,
   set $y_n \coloneq S_tu_\tau^k + 1/n$, then $y_n$ still satisfies 
   $\nabla y_n\cdot\boldsymbol{n} = 0$ on $\partial\Omega$.
   By integration by parts, it follows
   \begin{align*}
      \int_{\Omega} (\Delta S_tu_\tau^k)y_n^{p-\alpha}\, dx
      &= - \int_{\Omega} \nabla S_tu_\tau^k\cdot\nabla (y_n)^{p-\alpha}\, dx\\
      &= -\int_{\Omega} (p-\alpha)\frac{|\nabla S_tu_\tau^k|^2}{(y_n)^{1+\alpha-p}}\, dx
      \leq 0
      \quad \mathrm{for}\ t>0.
   \end{align*}
   Since 
   $|\Delta S_tu_\tau^k| y_n^{p-\alpha} \leq |\Delta S_tu_\tau^k| ((S_tu_\tau^k)^{p-\alpha} + 1) \in L^1(\Omega)$
   and 
   $(y_n)^{p-\alpha} \to (S_tu_\tau^k)^{p-\alpha}$ a.e. in $\Omega$ as $n\to\infty$,
   we infer from Lebesgue's dominated convergence theorem that
   \begin{align*}
      \int_{\Omega} (\Delta S_tu_\tau^k)(S_tu_\tau^k)^{p-\alpha}\, dx 
      = \lim_{n\to\infty}\int_{\Omega} (\Delta S_tu_\tau^k)y_n^{p-\alpha}\, dx \leq 0
      \quad \mathrm{for}\ t>0,
   \end{align*}
   which gives \eqref{eq91}.
   Since $(S_tu_\tau^k,v_\tau^k) \in X$, by \eqref{eq37}, we have
   \begin{equation*}
      0 \leq \frac{1}{2\tau\chi}[W_{m_\varepsilon}(S_tu_\tau^k, u_\tau^{k-1})^2 - W_{m_\varepsilon}(u_\tau^k, u_\tau^{k-1})^2]
       + E(S_tu_\tau^k,v_\tau^k) - E(u_\tau^k,v_\tau^k).
   \end{equation*}
   Dividing \eqref{eq56} by $t>0$ and letting $t\to0$,
   we infer from Lemma \ref{lem4.6}, \eqref{eq57}, \eqref{eq58} and \eqref{eq91} that
   \begin{align*}
      0 &\leq -\frac{\lambda_\delta}{2\tau\chi}W_{m_\varepsilon}(u_\tau^k, u_\tau^{k-1})^2 + \frac{1}{\tau\chi}[\boldsymbol{V}_\delta(u_\tau^{k-1}) - \boldsymbol{V}_\delta(u_\tau^k)]\\
        &\quad \frac{p}{\chi(p-\alpha)}\int_{\Omega} (u_\tau^k)^{p-\alpha}\nabla\cdot(m_\varepsilon(u_\tau^k)\nabla\varphi)\, dx
         +\int_{\Omega} m_\varepsilon(u_\tau^k)\nabla v_\tau^k\cdot\nabla\varphi\, dx
         - \delta\int_{\Omega} (\Delta v_\tau^k) u_\tau^k\, dx.
   \end{align*}
   Further using \eqref{eq37} with $(\tilde{u},\tilde{v}) = (u_\tau^{k-1},v_\tau^{k-1})$:
   $$\frac{1}{2\tau\chi}W_{m_\varepsilon}(u_\tau^k, u_\tau^{k-1})^2 
   \leq E(u_\tau^{k-1},v_\tau^{k-1}) - E(u_\tau^k,v_\tau^k),$$
   note that $-\lambda_\delta \geq 0$, we conclude that
   \begin{align*}
      \frac{\boldsymbol{V}_\delta(u_\tau^k) - \boldsymbol{V}_\delta(u_\tau^{k-1})}{\chi}
      &\leq \tau\left[\frac{p}{\chi(p-\alpha)}\int_{\Omega} (u_\tau^k)^{p-\alpha}\nabla\cdot(m_\varepsilon(u_\tau^k)\nabla\varphi)\, dx
       + \int_{\Omega} m_\varepsilon(u_\tau^k)\nabla v_\tau^k\cdot\nabla\varphi\, dx\right]\\
      &\quad -\tau\lambda_\delta(E(u_\tau^{k-1},v_\tau^{k-1}) - E(u_\tau^k,v_\tau^k))
       - \tau\delta\int_{\Omega} (\Delta v_\tau^k) u_\tau^k\, dx.
   \end{align*}
   The proof is completed.
\end{proof}


\section{Uniform estimates and convergences}

\begin{dfn}\label{dfn5.1}
   We define the following piecewise constant functions:
   \begin{align*}
      \begin{cases}
         u_\tau(t) \coloneq u_\tau^k \quad \mathrm{if}\ t \in ((k-1)\tau, k\tau]\quad \mathrm{for}\ k\in \mathbb{N},\ u_\tau(0) \coloneq u_0,\\
         v_\tau(t) \coloneq v_\tau^k \quad \mathrm{if}\ t \in ((k-1)\tau, k\tau]\quad \mathrm{for}\ k\in \mathbb{N},\ v_\tau(0) \coloneq v_0.
      \end{cases} 
   \end{align*}
   Notice that since the minimizers $(u_\tau^k,v_\tau^k)$ are nonnegative functions (Remark \ref{rem3.5}),
   the above functions are also nonnegative.
\end{dfn}

\quad First, we establish the uniform estimates including time variables.

\begin{lem}\label{lem5.2}
   Let $T>0$, $1+\alpha-2/d \leq p \leq 1+\alpha$ and assume that $\chi>0$ is small enough if $p=1+\alpha-2/d$. 
   Then there exist positive constants $C_1, C_2, C_3, C_4$ and $C_5$ depending on $\alpha,p,d,\chi,u_0$ and $v_0$
   such that the following uniform estimates hold:
   \begin{align}
      &\sup_{0\leq t \leq T}\left(\|u_\tau(t)\|_{L^{p+1-\alpha}(\Omega)}^{p+1-\alpha} + \|v_\tau(t)\|_{H^1(\Omega)}^2\right) 
       \leq C_1,\label{eq61}\\
      &\int_0^T \left(\|\nabla (u_\tau(t))^{\frac{p+1-\alpha}{2}}\|_{L^2(\Omega)}^2 + \|\Delta v_\tau(t) - v_\tau(t) + u_\tau(t)\|_{L^2(\Omega)}^2\right)\, dt 
       \leq C_2(1+T),\label{eq62}\\
      &\int_0^T \|v_\tau(t)\|_{H^2(\Omega)}^2\, dt 
       \leq C_3(1+T),\label{eq73}\\
      &\int_0^T \|u_\tau(t)\|_{L^2(\Omega)}^2\, dt
       \leq C_4(1+T),\label{eq66}\\
      &\int_0^T \|u_\tau(t)\|_{W^{1,p+1-\alpha}(\Omega)}^2\, dt
       \leq C_5(1+T).\label{eq79}
   \end{align}
\end{lem}

\begin{proof}
   To simplify, we set $T=N\tau$ for $N\in\mathbb{N}$.
   From \eqref{eq38}, summing up $i=1$ to $i=k$ for any $k \in \mathbb{N}$, we have
   \begin{align*}
      \sum_{i=1}^k E(u_\tau^i,v_\tau^i) &\leq \sum_{i=1}^k E(u_\tau^{i-1},v_\tau^{i-1}),
   \end{align*}
   so that
   \begin{align*}
      E(u_\tau^k,v_\tau^k) &\leq E(u_0,v_0),
   \end{align*}
   then
   \begin{equation*}
      \frac{p}{\chi(p-\alpha)(p+1-\alpha)}\|u_\tau^k\|_{L^{p+1-\alpha}(\Omega)}^{p+1-\alpha} 
      - \|u_\tau^k v_\tau^k\|_{L^1(\Omega)} 
      + \frac{1}{2}\|v_\tau^k\|_{H^1(\Omega)}^2
       \leq E(u_0,v_0).
   \end{equation*}
   By using the inequality \eqref{eq30} or \eqref{eq89}, it follows
   \begin{equation*}
      \|u_\tau^k\|_{L^{p+1-\alpha}(\Omega)}^{p+1-\alpha} + \|v_\tau^k\|_{H^1(\Omega)}^2
      \leq C_1\quad \mathrm{for}\ k \in \mathbb{N}
   \end{equation*}
   for some constant $C_1 = C_1(\alpha,p,d,\chi,u_0,v_0)$, which gives
   \begin{equation*}
      \sup_{0\leq t\leq T} \left(\|u_\tau(t)\|_{L^{p+1-\alpha}(\Omega)}^{p+1-\alpha} + \|v_\tau(t)\|_{H^1(\Omega)}^2\right) \leq C_1.
   \end{equation*}
   Combining the above uniform estimate with \eqref{eq74}, we have
   \begin{align*}
      &\frac{4p}{\chi(p+1-\alpha)^2}\|\nabla (u_\tau^k)^{\frac{p+1-\alpha}{2}}\|_{L^2(\Omega)}^2 + \|\Delta v_\tau^k - v_\tau^k + u_\tau^k\|_{L^2(\Omega)}^2\\
      &\leq \frac{2}{\tau\chi}(\boldsymbol{U}_\varepsilon(u_\tau^{k-1}) - \boldsymbol{U}_\varepsilon(u_\tau^k)) + \frac{\|v_\tau^{k-1}\|_{H^1(\Omega)}^2 - \|v_\tau^k\|_{H^1(\Omega)}^2}{\tau}
       + C_1 + C_0(C_1 + C_1^{\frac{1}{p-\alpha}}).
   \end{align*} 
   Hence it follows by integrating over $(0,T)$ that
   \begin{align*}
      &\int_0^T \left(\frac{4p}{\chi(p+1-\alpha)^2}\|\nabla (u_\tau(t))^{\frac{p+1-\alpha}{2}}\|_{L^2(\Omega)}^2 + \|\Delta v_\tau(t) - v_\tau(t) + u_\tau(t)\|_{L^2(\Omega)}^2\right)\, dt\\
      &= \sum_{k=1}^N \int_{(k-1)\tau}^{k\tau} \left(\frac{4p}{\chi(p+1-\alpha)^2}\|\nabla (u_\tau^k)^{\frac{p+1-\alpha}{2}}\|_{L^2(\Omega)}^2 + \|\Delta v_\tau^k - v_\tau^k + u_\tau^k\|_{L^2(\Omega)}^2\right)\, dt\\
      &\leq \frac{2}{\chi}(\boldsymbol{U}_\varepsilon(u_0) - \boldsymbol{U}_\varepsilon(u_\tau^N)) + \|v_0\|_{H^1(\Omega)}^2 - \|v_\tau^N\|_{H^1(\Omega)}^2 + \tilde{C}T\\
      &\leq \frac{2}{\chi(1-\alpha)}\|u_0\|_{L^{2-\alpha}(\Omega)}^{2-\alpha} + \|v_0\|_{H^1(\Omega)}^2 + \tilde{C}T,
   \end{align*}
   where $\tilde{C} = C_1 + C_0(C_1+C_1^{\frac{1}{p-\alpha}})$ and 
   we used Lemma \ref{lem2.11}:
   $$\boldsymbol{U}_\varepsilon(u_0) \leq \frac{1}{1-\alpha}\|u_0\|_{L^{2-\alpha}(\Omega)}^{2-\alpha}
   \ \mathrm{and}\ \boldsymbol{U}_\varepsilon \geq 0.$$
   Thus there exists a constant $C_2 = C_2(\alpha,p,d,\chi,u_0,v_0)>0$ such that
   \begin{align*}
      \int_0^T \left(\|\nabla (u_\tau(t))^{\frac{p+1-\alpha}{2}}\|_{L^2(\Omega)}^2 + \|\Delta v_\tau(t) - v_\tau(t) + u_\tau(t)\|_{L^2(\Omega)}^2\right)\, dt
      \leq C_2(1+T).
   \end{align*}
   Observe that if $p=1+\alpha$ then $p+1-\alpha = 2$, that is, 
   $\sup_{0\leq t\leq T}\|u_\tau(t)\|_{L^2(\Omega)}^2 \leq C_1$.
   By Corollary \ref{cor4.12}, \eqref{eq61}, \eqref{eq62} and \eqref{eq75},
   we can get the estimate \eqref{eq73} for some constant $C_3$.
   By \eqref{eq75}, \eqref{eq61} and \eqref{eq62}, 
   we immediately obtain \eqref{eq66} for some constant $C_4$.
   Finally,
   it follows from Lemma \ref{lem7.1} that
   \begin{align}\label{eq85}
      \nabla u_{\tau}(t) = \frac{2}{p+1-\alpha}u_{\tau}^{\frac{1+\alpha-p}{2}}(t)\nabla (u_{\tau}(t))^{\frac{p+1-\alpha}{2}}
      \quad \mathrm{a.e.\ in}\ \Omega\ \mathrm{for}\ t\in [0,T].
   \end{align}
   Then we infer from \eqref{eq61} and \eqref{eq62} that 
   \begin{align*}
      &\int_0^T\left(\int_{\Omega} |\nabla u_\tau|^{p+1-\alpha}\, dx\right)^{\frac{2}{p+1-\alpha}}\, dt\\
      &\leq \left(\frac{2}{p+1-\alpha}\right)^2\int_0^T 
       \left(\int_{\Omega} (u_\tau)^{p+1-\alpha}\, dx\right)^{\frac{1+\alpha-p}{p+1-\alpha}}
       \left(\int_{\Omega} |\nabla(u_\tau)^{\frac{p+1-\alpha}{2}}|^2\, dx\right)\, dt\\
      &\leq \left(\frac{2}{p+1-\alpha}\right)^2
       \left(\sup_{0\leq t\leq T} \|u_\tau(t)\|_{L^{p+1-\alpha}(\Omega)}^{p+1-\alpha}\, dt\right)^{\frac{1+\alpha-p}{p+1-\alpha}}
       \left(\int_0^T \|\nabla (u_\tau(t))^{\frac{p+1-\alpha}{2}}\|_{L^2(\Omega)}^2\, dt\right)\\
      &\leq \left(\frac{2}{p+1-\alpha}\right)^2 C_1^{\frac{1+\alpha-p}{p+1-\alpha}}C_2(1+T).
   \end{align*}
   Thus, combining this estimate with \eqref{eq61},
   we obtain \eqref{eq79} and complete the proof.
\end{proof}

\quad The following lemma is about estimates like the equi-continuity to use the refined Ascoli--Arzel\`a theorem (\cite[Proposition 3.3.1]{AGS}).
Note that the weighted Wasserstein distance depends on $\varepsilon$.

\begin{lem}\label{lem5.3}
   Let $T>0$, $p\geq1+\alpha-2/d$ and assume that $\chi>0$ is small enough if $p=1+\alpha-2/d$. 
   Then there exists $C_6 = C_6(\alpha,p,d,\chi,u_0,v_0) > 0$ satisfying
   for all $(t,s) \in [0,T]^2$ and $\tau \in (0,1)$ it holds
   \begin{align*}
      W_{m_\varepsilon}(u_\tau(t), u_\tau(s)) \leq C_6(\sqrt{|t-s|} + \sqrt{\tau}),\\
      \|v_\tau(t) - v_\tau(s)\|_{L^2(\Omega)} \leq C_6(\sqrt{|t-s|} + \sqrt{\tau}).
   \end{align*}
\end{lem}

\begin{proof}
   We only prove the first inequality because the other inequality can be shown by the same argument.
   Let $0\leq s < t \leq T$ and define
   $$N \coloneq \bigg\lceil\frac{t}{\tau}\bigg\rceil,\ 
   P \coloneq \bigg\lceil\frac{s}{\tau}\bigg\rceil,$$
   where $\lceil x\rceil$ denotes the superior integer part of the real number $x$.
   From \eqref{eq37} with $\tilde{u}=u_\tau^{k-1}$ and $\tilde{v} = v_\tau^{k-1}$, we have
   $$W_{m_\varepsilon}(u_\tau^k,u_\tau^{k-1})^2 + \chi\|v_\tau^k - v_\tau^{k-1}\|_{L^2(\Omega)}^2
    \leq 2\tau\chi(E(u_\tau^{k-1},v_\tau^{k-1}) - E(u_\tau^k,v_\tau^k)),$$
   then
   $$\sum_{k=1}^N W_{m_\varepsilon}(u_\tau^k, u_\tau^{k-1})^2 
   \leq 2\tau\chi(E(u_0,v_0) - E(u_\tau^N,v_\tau^N)).$$
   Because the functional $E$ is bounded below in $X$ (see Lemma \ref{lem3.1}), we see
   $$\sum_{k=1}^N W_{m_\varepsilon}(u_\tau^k, u_\tau^{k-1})^2 
   \leq 2\tau\chi\left(E(u_0,v_0) - \inf_{(u,v)\in X}E(u,v)\right).$$
   Since $t \in ((N-1)\tau,N\tau]$ and $s \in ((P-1)\tau,P\tau]$ by the definition of $N$ and $P$, 
   it follows
   \begin{align*}
      W_{m_\varepsilon}(u_\tau(t), u_\tau(s)) 
      &= W_{m_\varepsilon}(u_\tau^N, u_\tau^P)
       \leq \sum_{k=P+1}^N W_{m_\varepsilon}(u_\tau^k, u_\tau^{k-1})\\
      &\leq \sqrt{N-P}\sqrt{\sum_{k=P+1}^N W_{m_\varepsilon}(u_\tau^k, u_\tau^{k-1})^2}\\
      &\leq \sqrt{N-P}\sqrt{2\tau\chi\left(E(u_0,v_0) - \inf_{(u,v)\in X}E(u,v)\right)}\\
      &\leq \sqrt{2\chi}\sqrt{t-s+\tau}\left(E(u_0,v_0) - \inf_{(u,v)\in X}E(u,v)\right)^{\frac{1}{2}}\\
      &\leq C_6(\sqrt{|t-s|} + \sqrt{\tau}),
   \end{align*}
   where $C_6 = C_6(\alpha,p,d,\chi,u_0,v_0)$ is a constant, and in the second ineqality, we used 
   $$(x_1 + \cdots + x_n)^2 \leq n(x_1^2 + \cdots + x_n^2)
   \quad \mathrm{for}\ x_i\geq0,\ i=1,\cdots, n.$$
   The proof is completed.
\end{proof}

\quad From the above lemmas, we obtain
the convergences with respect to $\tau$.

\begin{lem}\label{lem5.4}
   Let $T>0$,
   $1+\alpha-2/d\leq p \leq 1+\alpha$ and assume that $\chi>0$ is small enough if $p=1+\alpha-2/d$. 
   There exist a subsequence $\{(u_{\tau_n},v_{\tau_n})\}_n$ with $\tau_n \to 0$ as $n\to\infty$
   and a pair of functions $(u_\varepsilon, v_\varepsilon) \in X$ such that
   \begin{align*}
      &u_{\tau_n} \rightharpoonup u_\varepsilon\quad 
       \mathrm{weakly\ in}\ L^2((0,T);W^{1,p+1-\alpha}(\Omega))\ \mathrm{as}\ n\to\infty,\\
      &u_{\tau_n}(t) \rightharpoonup u_\varepsilon(t) \quad 
       \mathrm{weakly\ in}\ L^1\cap L^{p+1-\alpha}(\Omega)\ \mathrm{as}\ n\to\infty\ \mathrm{for}\ t\in[0,T],\\
      &v_{\tau_n} \rightharpoonup v_\varepsilon\quad
       \mathrm{weakly\ in}\ L^2((0,T);H^2(\Omega))\ \mathrm{as}\ n\to\infty,\\
      &v_{\tau_n}(t) \rightharpoonup v_\varepsilon\quad
       \mathrm{weakly\ in}\ H^1(\Omega)\ \mathrm{as}\ n\to\infty\ \mathrm{for}\ t\in[0,T].
   \end{align*}
   In particular, $v_\varepsilon \in C^{\frac{1}{2}}([0,T];L^2(\Omega))$.
\end{lem}

\begin{proof}
   By Lemma \ref{lem5.2}, $\{u_\tau\}_{\tau>0}$ is bounded in $L^2((0,T);W^{1,p+1-\alpha}(\Omega))$,
   then there exist a subsequence $\{u_{\tau_n}\}$ and a function $u_\varepsilon \in L^2((0,T);W^{1,p+1-\alpha}(\Omega))$
   such that $u_{\tau_n}$ weakly converges to $u_\varepsilon$ in $L^2((0,T);W^{1,p+1-\alpha}(\Omega))$.
   In addition, by Lemma \ref{lem5.3} and the refined Ascoli--Arzel\`a theorem (\cite[Proposition 3.3.1]{AGS}), 
   there exist a subsequence (not relabeled) and $\tilde{u}_\varepsilon : [0,T] \to \mathcal{P}(\Omega)$ such that
   $$u_{\tau_n}(t) \rightharpoonup \tilde{u}_\varepsilon(t) \quad 
   \mathrm{weakly\ in}\ L^1\cap L^{p+1-\alpha}(\Omega)\ \mathrm{as}\ n\to\infty\ \mathrm{for}\ t \in [0,T].$$
   Due to the uniqueness of limit, 
   we have $u_\varepsilon = \tilde{u}_\varepsilon$ a.e. in $\Omega\times[0,T]$.
   Similarly, by Lemma \ref{lem5.2}, $\{v_\tau\}_{\tau>0}$ is bounded in $L^2((0,T);H^2(\Omega))$ and
   by Lemma \ref{lem5.3} and the refined Ascoli--Arzel\`a theorem, we have
   \begin{align*}
      &v_{\tau_n} \rightharpoonup v_\varepsilon\quad
       \mathrm{weakly\ in}\ L^2((0,T);H^2(\Omega))\ \mathrm{as}\ n\to\infty,\\
      &v_{\tau_n}(t) \rightharpoonup v_\varepsilon\quad
       \mathrm{weakly\ in}\ H^1(\Omega)\ \mathrm{as}\ n\to\infty\ \mathrm{for}\ t\in[0,T],\\
      &v_\varepsilon \in C^{\frac{1}{2}}([0,T];L^2(\Omega)).
   \end{align*}
   The proof is completed.
\end{proof}

\quad In previous lemma, we derived the weak convergences with respect to $\tau$,
hence we next obtain the strong convergence for $\tau$.

\begin{lem}\label{lem5.5}
   Let $T>0$, $1+\alpha-2/d \leq p \leq 1+\alpha$ 
   and assume that $\chi>0$ is small enough if $p=1+\alpha-2/d$. 
   Then for the sequence $\{u_{\tau_n}\}_n$ in Lemma \ref{lem5.4}, it holds
   \begin{align*}
      &u_{\tau_n} \to u_\varepsilon\quad \mathrm{strongly\ in}\ L^2((0,T);L^{p+1-\alpha}(\Omega))\ 
      \mathrm{as}\ n\to\infty,\\
      &u_{\tau_n}(x,t) \to u_\varepsilon(x,t)\quad
      \mathrm{a.e.\ in}\ \Omega\times(0,T)\ \mathrm{as}\ n\to\infty.
   \end{align*}
\end{lem}

\begin{proof}
   Note that by the Rellich--Kondrachov theorem,
   $H^{d+1}(\Omega) = W^{d+1,2}(\Omega)$ is compactly embedded in $H^d(\Omega)$
   and by the Sobolev embedding theorem, $H^d(\Omega)$ 
   is continuously embedded in $L^{\frac{p+1-\alpha}{p-\alpha}}(\Omega)$.
   Hence it holds that
   $H^{-d}(\Omega)$ is compactly embedded in $H^{-(d+1)}(\Omega)$,
   where $H^{-d}(\Omega)$ is the dual space of $H^d(\Omega)$,
   and $L^{p+1-\alpha}(\Omega)$ is continuously embedded in $H^{-d}(\Omega)$.
   By Lemma \ref{lem5.2}, $\|u_{\tau_n}(t)\|_{L^{p+1-\alpha}(\Omega)}$ is bounded with respect to $\tau_n$ for all $t \in[0,T]$,
   thus there exist a subsequence (not relabeled) and $w_t \in H^{-(d+1)}(\Omega)$ such that
   $u_{\tau_n}(t)$ converges to $w_t$ strongly in $H^{-(d+1)}(\Omega)$.
   Now thanks to Lemma \ref{lem5.4}, 
   we know that $u_{\tau_n}(t)$ weakly converges to $u_\varepsilon(t)$ in $L^{p+1-\alpha}(\Omega)$.
   Due to the uniqueness of limit, we have $w_t = u_\varepsilon(t)$ a.e. in $\Omega$.
   Moreover, by Lemma \ref{lem5.2} and Lemma \ref{lem5.4}, we have
   \begin{align*}
      \sup_{n\in\mathbb{N}}\sup_{0\leq t\leq T}\|u_{\tau_n}(t) - u_\varepsilon(t)\|_{H^{-(d+1)}(\Omega)}^2 
      \leq \sup_{n\in\mathbb{N}}\sup_{0\leq t\leq T}\|u_{\tau_n}(t) - u_\varepsilon(t)\|_{L^{p+1-\alpha}(\Omega)}^2 < \infty.
   \end{align*}
   Hence we infer from Lebesgue's dominated convergence theorem that
   \begin{align*}
      \int_0^T \|u_{\tau_n}(t) - u_\varepsilon(t)\|_{H^{-(d+1)}(\Omega)}^2\, dt \to 0\quad
      \mathrm{as}\ n\to\infty,
   \end{align*}
   which implies that
   $\{u_{\tau_n}\}_n$ is relatively compact in $L^2((0,T);H^{-(d+1)}(\Omega))$.
   Since $\{u_{\tau_n}\}_n$ is bounded in $L^2((0,T);W^{1,p+1-\alpha}(\Omega))$ due to Lemma \ref{lem5.2},
   by \cite[Lemma 9]{S}, $\{u_{\tau_n}\}_n$ is relatively compact in $L^2((0,T);L^{p+1-\alpha}(\Omega))$.
   Therefore, taking a subsequence (not relabeled),
   $u_{\tau_n}$ converges to $u_\varepsilon$ strongly in $L^2((0,T);L^{p+1-\alpha}(\Omega))$ as $n\to\infty$.
   In addition, taking a subsequence if necessary, 
   $u_{\tau_n}(x,t) \to u_\varepsilon(x,t)$ a.e. in $\Omega\times(0,T)$.
\end{proof}

\begin{lem}\label{lem5.6}
   Let $T>0$, $1+\alpha-2/d \leq p \leq 1+\alpha$ 
   and assume that $\chi>0$ is small enough if $p=1+\alpha-2/d$. 
   Then for the sequence $\{u_{\tau_n}\}_n$ in Lemma \ref{lem5.4}, it holds
   \begin{align*}
      \nabla (u_{\tau_n})^p \rightharpoonup \nabla (u_\varepsilon)^p \quad
      \mathrm{weakly\ in}\ L^{\frac{p+1-\alpha}{p}}(\Omega\times(0,T))\ 
      \mathrm{as}\ n\to\infty.
   \end{align*}
   Moreover there exists a constant $C_7 = C_7(\alpha,p,d,\chi,u_0,v_0)>0$
   such that
   \begin{align}\label{eq92}
      \int_0^T\int_{\Omega} |\nabla (u_{\tau_n})^p|^{\frac{p+1-\alpha}{p}}\, dx\, dt
      \leq C_7(1+T).
   \end{align}
\end{lem}

\begin{proof}
   Since $\nabla(u_{\tau_n})^p = 2p/(p+1-\alpha)u_{\tau_n}^{\frac{p+\alpha-1}{2}}\nabla (u_{\tau_n})^{\frac{p+1-\alpha}{2}}$,
   we infer from H\"older's inequality and Lemma \ref{lem5.2} that
   \begin{align*}
      &\int_0^T\int_{\Omega} |\nabla(u_{\tau_n})^p|^{\frac{p+1-\alpha}{p}}\, dx\, dt\\
      &= \int_0^T\int_{\Omega} \left(\frac{2p}{p+1-\alpha}\right)^{\frac{p+1-\alpha}{p}}
       (u_{\tau_n})^{\frac{p+\alpha-1}{2}\frac{p+1-\alpha}{p}}|\nabla(u_{\tau_n})^{\frac{p+1-\alpha}{2}}|^{\frac{p+1-\alpha}{p}}\, dx\, dt\\
      &\leq \left(\frac{2p}{p+1-\alpha}\right)^{\frac{p+1-\alpha}{p}}
       \int_0^T\left(\int_{\Omega} |\nabla (u_{\tau_n})^{\frac{p+1-\alpha}{2}}|^2\, dx\right)^{\frac{p+1-\alpha}{2p}}
       \left(\int_{\Omega} (u_{\tau_n})^{p+1-\alpha}\, dx\right)^{\frac{p+\alpha-1}{2p}}\, dt\\
      &\leq \left(\frac{2p}{p+1-\alpha}\right)^{\frac{p+1-\alpha}{p}}
       C_1^{\frac{p+\alpha-1}{2p}}C_2^{\frac{p+1-\alpha}{2p}}(1+T).
   \end{align*}
   Hence there exist a subsequence (not relabeled) and 
   $y_\varepsilon \in L^{\frac{p+1-\alpha}{p}}(\Omega\times(0,T))$
   such that $\nabla(u_{\tau_n})^p \rightharpoonup y_\varepsilon$
   weakly in $L^{\frac{p+1-\alpha}{p}}(\Omega\times(0,T))$ as $n\to\infty$.
   Combining this with Lemma \ref{lem5.5}, we see that
   $\nabla(u_{\tau_n})^p$ converges to $\nabla(u_\varepsilon)^p$ 
   weakly in $L^{\frac{p+1-\alpha}{p}}(\Omega\times(0,T))$ as $n\to\infty$.
\end{proof}


\section{Proof of Theorem \ref{thm1.1} and Theorem \ref{thm1.4}}

\quad First, we establish weak formulations of the system \ref{peks_e}.

\begin{lem}\label{lem5.7}
   Let $1+\alpha-2/d\leq p\leq 1+\alpha$ and assume that $\chi>0$ is small enough if $p=1+\alpha-2/d$. 
   Then $(u_\varepsilon,v_\varepsilon)$ in Lemma \ref{lem5.4} satisfies the following weak formulation: 
   for all $T>0$ and $\varphi \in C^\infty(\overline{\Omega})$ with $\nabla \varphi\cdot\boldsymbol{n} = 0$ on $\partial\Omega$, 
   it holds 
   \begin{align}\label{eq12}
      \int_{\Omega} (u_0(x) - u_\varepsilon(x,T))\varphi(x)\, dx
      &= - \int_0^T\int_{\Omega} \frac{\alpha}{p-\alpha}\left(\frac{u_\varepsilon(x,t)}{u_\varepsilon(x,t) + \varepsilon}\right)^{1-\alpha} \nabla u_\varepsilon(x,t)^p\cdot\nabla\varphi(x)\, dx\, dt\notag\\
      &\quad - \int_0^T\int_{\Omega} \frac{p}{p-\alpha}u_\varepsilon(x,t)^{p-\alpha} m_\varepsilon(u_\varepsilon(x,t))\Delta\varphi\, dx\, dt\notag\\
      &\quad - \int_0^T\int_{\Omega} \chi m_\varepsilon(u_\varepsilon(x,t)) \nabla v_\varepsilon(x,t)\cdot\nabla\varphi(x)\, dx\, dt.
   \end{align}
\end{lem}

\begin{proof}
   Let $T>0$ and $\varphi \in C^\infty(\overline{\Omega})$ 
   with $\nabla\varphi\cdot\boldsymbol{n} = 0$ on $\partial\Omega$.
   Let $\{\tau_n\} \subset (0,1)$ be a subsequence of $\{\tau\}$ 
   which is obtained in Lemma \ref{lem5.4} and Lemma \ref{lem5.5},
   and set $\delta_n \coloneq \tau_n^{\frac{1}{2}}$.
   To simplify, we assume that $T=N\tau_n$ for some $N\in \mathbb{N}$.
   By Lemma \ref{lem4.9}, we have
   \begin{align*}
      &\frac{\boldsymbol{V}_{\delta_n}(u_{\tau_n}(T)) - \boldsymbol{V}_{\delta_n}(u_{\tau_n}(0))}{\chi}\\
      &= \frac{\boldsymbol{V}_{\delta_n}(u_{\tau_n}^N) - \boldsymbol{V}_{\delta_n}(u_0)}{\chi}
       = \sum_{k=1}^N \frac{\boldsymbol{V}_{\delta_n}(u_{\tau_n}^k) - \boldsymbol{V}_{\delta_n}(u_{\tau_n}^{k-1})}{\chi}\\
      &\leq \sum_{k=1}^N \tau_n\left[\frac{p}{\chi(p-\alpha)}\int_{\Omega} (u_{\tau_n}^k)^{p-\alpha} \nabla\cdot(m_\varepsilon(u_{\tau_n}^k)\nabla\varphi)\, dx
       + \int_{\Omega} m_\varepsilon(u_{\tau_n}^k)\nabla v_{\tau_n}^k\cdot\nabla\varphi\, dx\right]\\
      &\quad -\tau_n\lambda_{\delta_n}\sum_{k=1}^N (E(u_{\tau_n}^{k-1},v_{\tau_n}^{k-1}) - E(u_{\tau_n}^k,v_{\tau_n}^k))
       - \delta_n\sum_{k=1}^N \tau_n \int_{\Omega} (\Delta v_{\tau_n}^k) u_{\tau_n}^k\, dx\\
      &= \int_0^T \int_{\Omega} \frac{p}{\chi(p-\alpha)}(u_{\tau_n}(t))^{p-\alpha} \nabla\cdot(m_\varepsilon(u_{\tau_n}(t))\nabla\varphi)\, dx\, dt
       + \int_0^T\int_{\Omega} m_\varepsilon(u_{\tau_n}(t))\nabla v_{\tau_n}(t)\cdot\nabla\varphi\, dx\, dt\\
      &\quad -\tau_n\lambda_{\delta_n} \left(E(u_0,v_0) - E(u_{\tau_n}^N,v_{\tau_n}^N)\right)
       - \delta_n \int_0^T\int_{\Omega} (\Delta v_{\tau_n}(t)) u_{\tau_n}(t)\, dx\, dt.
   \end{align*}
   Here, by the definition of $\lambda_{\delta_n}$, we see 
   \begin{align*}
      \tau_n|\lambda_{\delta_n}| 
      &\leq \tau_n(\delta_n^{-1}+1)C(\varphi,\alpha,\varepsilon)
      \leq 2\tau_n^{\frac{1}{2}}C(\varphi, \alpha, \varepsilon),
   \end{align*}
   and by H\"older's inequality, the Sobolev embedding and Lemma \ref{lem5.2}, it follows
   \begin{align*}
      \int_0^T\int_{\Omega} |(\Delta v_{\tau_n}(t))| |u_{\tau_n}(t)|\, dx\, dt
      &\leq \left(\int_0^T \|\Delta v_{\tau_n}(t)\|_{L^2(\Omega)}^2\, dt\right)^{\frac{1}{2}}
       \left(\int_0^T \|u_{\tau_n}(t)\|_{L^2(\Omega)}^2\, dt\right)^{\frac{1}{2}}\\
      &\leq (C_3C_4)^{\frac{1}{2}}(1+T).
   \end{align*}
   Hence taking account of the definition of $\boldsymbol{V}_{\delta_n}$
   and the boundedness from below of $E$ in $X$ (Lemma \ref{lem3.1}), we obtain 
   \begin{align}\label{eq13}
      &\int_{\Omega} (u_{\tau_n}(x,T) - u_0(x))\varphi(x)\, dx\notag\\
      &\leq \int_0^T \int_{\Omega} \frac{p}{\chi(p-\alpha)}(u_{\tau_n}(t))^{p-\alpha} \nabla m_\varepsilon(u_{\tau_n}(t))\cdot\nabla\varphi\, dx\, dt\notag\\
      &\quad + \int_0^T \int_{\Omega} \frac{p}{\chi(p-\alpha)}(u_{\tau_n}(t))^{p-\alpha}m_\varepsilon(u_{\tau_n}(t))\Delta\varphi\, dx\, dt\notag\\
      &\quad + \int_0^T\int_{\Omega} m_\varepsilon(u_{\tau_n}(x,t))\nabla v_{\tau_n}(x,t)\cdot\nabla\varphi(x)\, dx\,dt\notag\\
      &\quad + \tau_n^{\frac{1}{2}}\left[2C(\varphi,\alpha,\varepsilon)\left(E(u_0,v_0) - \inf_{(\tilde{u},\tilde{v})\in X}E(\tilde{u},\tilde{v})\right) + (C_3C_4)^{\frac{1}{2}}(1+T)\right]\notag\\
      &\quad + \tau_n^{\frac{1}{2}} \left[\boldsymbol{U}_\varepsilon(u_0) - \boldsymbol{U}_\varepsilon(u_{\tau_n}(T))\right].
   \end{align}
   By Lemma \ref{lem2.11}, we have
   \begin{equation*}
      \boldsymbol{U}_\varepsilon(u_0) \leq \frac{1}{1-\alpha}\|u_0\|_{L^{2-\alpha}(\Omega)}^{2-\alpha}\ 
      \mbox{and}\ \boldsymbol{U}_{\varepsilon}(u_{\tau_n}(T)) \geq 0.
   \end{equation*}
   Thanks to Lemma \ref{lem5.4}, it is easy to check that 
   \begin{align*}
      \int_{\Omega} u_{\tau_n}(x,T) \varphi(x)\, dx
      \to \int_{\Omega} u_\varepsilon(x,T) \varphi(x)\, dx
      \quad \mathrm{as}\ n \to \infty.
   \end{align*}
   We will show 
   \begin{align*}
      &\int_0^T \int_{\Omega} (u_{\tau_n})^{p-\alpha} \nabla m_\varepsilon(u_{\tau_n})\cdot\nabla\varphi\, dx\, dt
      = \int_0^T \int_{\Omega} \frac{\alpha}{p}\left(\frac{u_{\tau_n}}{u_{\tau_n}+\varepsilon}\right)^{1-\alpha} \nabla (u_{\tau_n})^p\cdot\nabla\varphi\, dx\, dt\\
      &\hspace{6.3cm} \to 
      \int_0^T \int_{\Omega} \frac{\alpha}{p}\left(\frac{u_\varepsilon}{u_\varepsilon+\varepsilon}\right)^{1-\alpha} \nabla (u_\varepsilon)^p\cdot\nabla\varphi\, dx\, dt
      \quad \mathrm{as}\ n \to \infty,\\
      &\int_0^T \int_{\Omega} (u_{\tau_n})^{p-\alpha}m_\varepsilon(u_{\tau_n})\Delta\varphi\, dx\, dt
      \to 
      \int_0^T \int_{\Omega} (u_\varepsilon)^{p-\alpha}m_\varepsilon(u_\varepsilon)\Delta\varphi\, dx\, dt
      \quad \mathrm{as}\ n \to \infty
   \end{align*} 
   and
   \begin{align*}
      \int_0^T\int_{\Omega} m_\varepsilon(u_{\tau_n})\nabla v_{\tau_n}\cdot\nabla\varphi\, dx\,dt
      \to 
       \int_0^T\int_{\Omega} m_\varepsilon(u_\varepsilon)\nabla v_\varepsilon\cdot\nabla\varphi\, dx\,dt
       \quad \mathrm{as}\ n \to \infty.
   \end{align*}
   First, we have
   \begin{align*}
      &\left|\int_0^T\int_{\Omega} \left(\frac{u_{\tau_n}}{u_{\tau_n}+\varepsilon}\right)^{1-\alpha} \nabla (u_{\tau_n})^p\cdot\nabla\varphi\, dx\, dt
      - \int_0^T\int_{\Omega} \left(\frac{u_\varepsilon}{u_\varepsilon+\varepsilon}\right)^{1-\alpha} \nabla (u_\varepsilon)^p\cdot\nabla\varphi\, dx\, dt\right|\\
      &\leq \left|\int_0^T\int_{\Omega} \left[\left(\frac{u_{\tau_n}}{u_{\tau_n}+\varepsilon}\right)^{1-\alpha} - \left(\frac{u_\varepsilon}{u_\varepsilon+\varepsilon}\right)^{1-\alpha}\right]
       \nabla (u_{\tau_n})^p\cdot\nabla\varphi\, dx\, dt\right|\\
      &\quad + \left|\int_0^T\int_{\Omega} \left(\frac{u_\varepsilon}{u_\varepsilon+\varepsilon}\right)^{1-\alpha}(\nabla (u_{\tau_n})^p - \nabla (u_\varepsilon)^p)\cdot\nabla\varphi\, dx\, dt\right|.
   \end{align*}
   Since $a/(a+\varepsilon) \leq 1$ for $a\geq0$,
   we infer from Lemma \ref{lem5.5}, \eqref{eq92} and Lebesgue's dominated converge theorem that
   the first term converges to $0$ as $n\to\infty$.
   Further, by Lemma \ref{lem5.6}, we obtain that 
   the second term converges to $0$ as $n\to\infty$.
   Next, we have
   \begin{align*}
      &\left|\int_0^T\int_{\Omega} (u_{\tau_n})^{p-\alpha}m_\varepsilon(u_{\tau_n})\Delta\varphi\, dx\,dt
       - \int_0^T\int_{\Omega} (u_\varepsilon)^{p-\alpha}m_\varepsilon(u_\varepsilon)\Delta\varphi\, dx\,dt\right|\\
      &\leq \int_0^T\int_{\Omega} |(u_{\tau_n})^{p-\alpha} - (u_\varepsilon)^{p-\alpha}|m_\varepsilon(u_{\tau_n})|\Delta\varphi|\, dx\, dt
       + \int_0^T\int_{\Omega} |m_\varepsilon(u_{\tau_n}) - m_\varepsilon(u_\varepsilon)|(u_\varepsilon)^{p-\alpha}|\Delta\varphi|\, dx\,dt\\
      &\eqcolon \mathrm{I}_1 + \mathrm{I}_2.
   \end{align*}
   Observe that $m_\varepsilon(r) = (r+\varepsilon)^\alpha$,
   we infer from \eqref{eq43} and H\"older's inequality that 
   \begin{align*}
      \mathrm{I}_1
      &\leq \int_0^T\int_{\Omega}|u_{\tau_n} - u_\varepsilon|^{p-\alpha} m_\varepsilon(u_{\tau_n})|\Delta\varphi|\, dx\,dt\\
      &\leq \|\Delta\varphi\|_{L^\infty(\Omega)}\int_0^T \left(\int_{\Omega}|u_{\tau_n} - u_\varepsilon|^{p+1-\alpha}\, dx\right)^{\frac{p-\alpha}{p+1-\alpha}}
       \left(\int_{\Omega} m_\varepsilon(u_{\tau_n})^{p+1-\alpha}\, dx\right)^{\frac{1}{p+1-\alpha}}\, dt\\
      &\leq \|\Delta\varphi\|_{L^\infty(\Omega)} \|u_{\tau_n} - u_\varepsilon\|_{L^2((0,T);L^{p+1-\alpha}(\Omega))}^{p-\alpha}
       \left(\int_0^T \|(u_{\tau_n} + \varepsilon)^\alpha\|_{L^{p+1-\alpha}}^{\frac{2}{2-p+\alpha}}\, dt\right)^{\frac{2-p+\alpha}{2}}.
   \end{align*}
   By Lemma \ref{lem5.5} and \eqref{eq61}, we have $\mathrm{I}_1 \to 0$ as $n \to \infty$.
   By the similar argument, we obtain 
   $\mathrm{I}_2 \to 0$ as $n \to \infty$.
   Finally, we have
   \begin{align*}
      &\left|\int_0^T\int_{\Omega} m_\varepsilon(u_{\tau_n})\nabla v_{\tau_n}\cdot\nabla\varphi\, dx\,dt
       - \int_0^T\int_{\Omega} m_\varepsilon(u_\varepsilon)\nabla v_\varepsilon\cdot\nabla\varphi\, dx\,dt\right|\\
      &\leq \int_0^T\int_{\Omega} |m_\varepsilon(u_{\tau_n}) - m_\varepsilon(u_\varepsilon)||\nabla v_{\tau_n}||\nabla\varphi|\, dx\,dt
       + \left|\int_0^T\int_{\Omega} m_\varepsilon(u_\varepsilon)(\nabla v_{\tau_n} - \nabla v_\varepsilon)\nabla\varphi\, dx\,dt\right|.
   \end{align*}
   Since $\sup_{0\leq t\leq T}\|\nabla v_{\tau_n}(t)\|_{L^2(\Omega)}$ is bounded,
   as in the above argument,
   the first term converges to $0$ as $n\to\infty$.
   In additon,
   since $\nabla v_{\tau_n}$ converges to $\nabla v_\varepsilon$ weakly in $L^2(\Omega\times(0,T))$
   and $m_\varepsilon(u_\varepsilon)\nabla\varphi \in L^2(\Omega\times(0,T))$,
   the second term also converges to $0$ as $n\to\infty$.

   Hence by letting $n \to \infty$ in \eqref{eq13}, it follows
   \begin{align*}
      \frac{1}{\chi}\int_{\Omega} (u_\varepsilon(x,T) - u_0(x))\varphi(x)\, dx
      &\leq \int_0^T\int_{\Omega} \frac{\alpha}{\chi(p-\alpha)} \left(\frac{u_\varepsilon(x,t)}{u_\varepsilon(x,t)+\varepsilon}\right)^{1-\alpha}\nabla u_\varepsilon(x,t)^p\cdot\nabla\varphi(x)\, dx\,dt\\
      &\quad + \int_0^T \int_{\Omega} \frac{\alpha}{\chi(p-\alpha)}(u_\varepsilon(x,t))^{p-\alpha}m_\varepsilon(u_\varepsilon(x,t))\Delta\varphi(x)\, dx\, dt\\
      &\quad + \int_0^T\int_{\Omega} m_\varepsilon(u_\varepsilon(x,t))\nabla v_\varepsilon(x,t)\cdot\nabla\varphi(x)\, dx\,dt.
   \end{align*}
   Replacing $\varphi$ with $-\varphi$, we have 
   \begin{align*}
      \int_{\Omega} (u_0(x) - u_\varepsilon(x,T))\varphi(x)\, dx
      &= - \int_0^T\int_{\Omega} \frac{\alpha}{p-\alpha} \left(\frac{u_\varepsilon(x,t)}{u_\varepsilon(x,t)+\varepsilon}\right)^{1-\alpha}\nabla u_\varepsilon(x,t)^p\cdot\nabla\varphi(x)\, dx\,dt\\
      &\quad - \int_0^T \int_{\Omega} \frac{\alpha}{p-\alpha}(u_\varepsilon(x,t))^{p-\alpha}m_\varepsilon(u_\varepsilon(x,t))\Delta\varphi(x)\, dx\, dt\\
      &\quad - \int_0^T\int_{\Omega} \chi m_\varepsilon(u_\varepsilon(x,t))\nabla v_\varepsilon(x,t)\cdot\nabla\varphi(x)\, dx\,dt.
   \end{align*}
   The proof is completed.
\end{proof}

\begin{lem}\label{lem6.2}
   Let $T>0$. Then 
   $(u_\varepsilon,v_\varepsilon)$ satisfies the following weak formulation:
   for all $\zeta \in H^1(\Omega)$, it holds 
   \begin{align*}
      \int_0^T\int_{\Omega} [\nabla v_\varepsilon\cdot\nabla \zeta + v_\varepsilon\zeta - u_\varepsilon\zeta]\, dx\, dt
      = \int_{\Omega} (v_0 - v_\varepsilon(T))\zeta\, dx.
   \end{align*}
\end{lem}

\begin{proof}
   Let $T>0$, $\zeta \in H^1(\Omega)$ and $\{\tau_n\}$ be a subsequence in Lemma \ref{lem5.4} and Lemma \ref{lem5.5}.
   To simplify, we assume $T=N\tau_n$.
   By Lemma \ref{lem4.10}, we have
   \begin{align*}
      \sum_{k=1}^N \tau_n \int_{\Omega} (\nabla v_{\tau_n}^k\cdot\nabla\zeta + v_{\tau_n}^k\zeta - u_{\tau_n}^k\zeta)\, dx
      = \sum_{k=1}^N \int_{\Omega} (v_{\tau_n}^{k-1} - v_{\tau_n}^k)\zeta\, dx
   \end{align*}
   then
   \begin{align*}
      \int_0^T\int_{\Omega} (\nabla v_{\tau_n}(t)\cdot\nabla\zeta + v_{\tau_n}(t)\zeta - u_{\tau_n}(t)\zeta)\, dx\, dt
      = \int_{\Omega} (v_0 - v_{\tau_n}^N)\zeta\, dx.
   \end{align*}
   Hence letting $n\to\infty$, we infer from Lemma \ref{lem5.4} that
   \begin{align*}
      \int_0^T\int_{\Omega} (\nabla v_\varepsilon(t)\cdot\nabla\zeta + v_\varepsilon(t)\zeta - u_\varepsilon(t)\zeta)\, dx\, dt
      = \int_{\Omega} (v_0 - v_\varepsilon(T))\zeta\, dx.
   \end{align*}
   The proof is completed.
\end{proof}

\quad Next lemma is about the weak compactness of $\{u_\varepsilon(t)\}_\varepsilon$ for each $t\in[0,T]$.
As in the Lemma \ref{lem5.4}, if we have the equi-continuity with respect to $W_m$, where $m(r)=r^\alpha$,
then we can easily get the conclusion adapting the refined Ascoli--Arzel\`a theorem (\cite[Proposition 3.3.1]{AGS}).
However, we only have the equi-continuity with respect to $W_{m_\varepsilon}$ depending on $\varepsilon$ (Lemma \ref{lem5.3}).
To avoid this problem, we use not only the equi-continuity 
but also the lower semicontinuity of the weighted Wasserstein distance (see Lemma \ref{lem2.7}).

\begin{lem}\label{lem5.8}
   Let $T>0$,
   $1+\alpha-2/d\leq p\leq 1+\alpha$ and assume that $\chi>0$ is small enough if $p=1+\alpha-2/d$. 
   There exist a subsequence $\{u_{\varepsilon_n}\}_n$ with $\varepsilon_n \to 0$ as $n\to\infty$
   and $u : [0,T] \to \mathcal{P}(\Omega)$ such that
   $$u_{\varepsilon_n}(t) \rightharpoonup u(t) 
   \quad \mathrm{weakly\ in}\ L^1\cap L^{p+1-\alpha}(\Omega)\ \mathrm{as}\ n\to\infty\ \mathrm{for}\ t \in [0,T].$$
   In particular, $W_m(u(t),u(s)) \leq C_6\sqrt{|t-s|}\ \mathrm{for}\ t,s\in [0,T]$,
   where $m(r) = r^\alpha$.
\end{lem}

\begin{proof}
   First, $(\mathcal{M}_{loc}^+(\Rd),W_m)$ is complete (\cite[Theorem 5.7]{DNS})
   and $m_\varepsilon(r)$ is decreasing with respect to $\varepsilon$
   and pointwise converging to $m(r)$ as $\varepsilon \to 0$.
   Set 
   $$S \coloneq \{f \in L^{p+1-\alpha}\cap\mathcal{P}(\Omega); \|f\|_{L^{p+1-\alpha}(\Omega)}^{p+1-\alpha} \leq C_1\},$$
   where $C_1$ is the constant in Lemma \ref{lem5.2}.
   Then $S$ is sequentially compact with respect to the weak topology of $L^1\cap L^{p+1-\alpha}(\Omega)$.
   Indeed, 
   let $\{f_n\} \subset S$, we can easily see that
   $\{f_n\}$ is bounded in $L^1\cap L^{p+1-\alpha}(\Omega)$ and equi-integrable.
   Hence, taking a subsequence (not relabeled), there exists a function $f \in L^1\cap L^{p+1-\alpha}(\Omega)$ such that
   $$f_n \rightharpoonup f \quad \mathrm{weakly\ in}\ L^1\cap L^{p+1-\alpha}(\Omega)\ \mathrm{as}\ n \to \infty.$$
   Note that if $f_n$ converges to $f$ weakly in $L^1\cap L^{p+1-\alpha}(\Omega)$
   then $f_n$ converges to $f$ weakly* in $\mathcal{M}_{loc}^+(\Rd)$.
   Since $f_n$ weakly converges to $f$ in $L^{p+1-\alpha}(\Omega)$ as $n\to\infty$,
   we have $\|f\|_{L^{p+1-\alpha}(\Omega)}^{p+1-\alpha} \leq C_1$, then $f\in S$.
   
   Since $u_\varepsilon(t) \in L^{p+1-\alpha}\cap \mathcal{P}(\Omega)$ and $\|u_\varepsilon(t)\|_{L^{p+1-\alpha}(\Omega)}^{p+1-\alpha} \leq C_1$
   for all $t\in[0,T]$ by Lemma \ref{lem5.2} and Lemma \ref{lem5.4}, that is,
   $\{u_\varepsilon(t)\}_\varepsilon \subset S$ for all $t \in [0,T]$,
   using the diagonal argument,
   there exist a subsequence $\{u_{\varepsilon_n}\}_n$ 
   and $u : \mathbb{Q}\cap[0,T] \to S$ such that
   $$u_{\varepsilon_n}(t) \rightharpoonup u(t)\quad 
   \mathrm{weakly\ in}\ L^1\cap L^{p+1-\alpha}(\Omega)\ \mathrm{as}\ n\to\infty\  
   \mathrm{for}\ t \in \mathbb{Q}\cap[0,T].$$
   By Lemma \ref{lem2.7} and Lemma \ref{lem5.3}, we have
   \begin{align*}
      W_m(u(t), u(s)) 
      &\leq \liminf_{\varepsilon_n\to0}W_{\varepsilon_n}(u_{\varepsilon_n}(t), u_{\varepsilon_n}(s))
       \leq \limsup_{\varepsilon_n\to0}W_{\varepsilon_n}(u_{\varepsilon_n}(t), u_{\varepsilon_n}(s))\\
      &\leq \limsup_{\varepsilon_n\to0}\liminf_{\tau\to0}W_{\varepsilon_n}(u_{\tau}(t), u_{\tau}(s))\\
      &\leq \limsup_{\varepsilon_n\to0}\limsup_{\tau\to0}W_{\varepsilon_n}(u_{\tau}(t), u_{\tau}(s))\\
      &\leq C_6\sqrt{|t-s|}.
   \end{align*}
   We will show that $(S,W_m)$ is complete.
   Let $\{f_n\} \subset S$ be a Cauchy sequence, 
   since $\{f_n\} \subset L_{loc}^1(\Rd) \subset \mathcal{M}_{loc}^+(\Rd)$ 
   and $(\mathcal{M}_{loc}^+(\Rd),W_m)$ is complete, 
   there exists a Radon measure $f \in \mathcal{M}_{loc}^+(\Rd)$ such that $f_n \to f$ in $W_m$.
   In particular, $f_n \rightharpoonup f$ weakly* in $\mathcal{M}_{loc}^+(\Rd)$ (Proposition \ref{prop2.5}).
   Note that since $\{f_n\} \subset L_{loc}^1(\Rd)$ and $\{f_n\}$ is bounded in $L^{p+1-\alpha}(\Omega)$,
   we can identify the measure $f \in \mathcal{M}_{loc}^+(\Rd)$ with the density function $f \in L^1_{loc}(\Rd)$.

   On the other hand, since $\{f_n\} \subset S$,
   there exist a subsequence $\{f_{n_k}\}_k$ and $g \in S$ such that
   $f_{n_k} \rightharpoonup g$ weakly in $L^1\cap L^{p+1-\alpha}(\Omega)$.
   For all $\zeta \in C_c^\infty(\Omega)$, we have
   $$\int_{\Omega} f_{n_k}\zeta\, dx \to \int_{\Omega} f\zeta\, dx\quad 
   \mathrm{as}\ k \to \infty$$
   and
   $$\int_{\Omega} f_{n_k}\zeta\, dx \to \int_{\Omega} g\zeta\, dx\quad 
   \mathrm{as}\ k \to \infty.$$
   Hence we obtain $f = g$ a.e. in $\Omega$.
   Thus it holds that $f \in S$ and 
   $$W_m(f_n, f) \leq W_m(f_n, f_{n_k}) + W_m(f_{n_k}, f) \to 0 \quad 
   \mathrm{as}\ n,k \to \infty.$$
   Let $t \in [0,T]$, 
   then there exists $\{t_k\} \subset \mathbb{Q}\cap[0,T]$ such that
   $t_k \to t$ as $k \to \infty$.
   Since $W_m(u(t_k), u(t_l)) \leq C_6\sqrt{|t_k - t_l|} \to 0$
   as $k,l \to \infty$ and $\{u(t_k)\} \subset S$,
   we can uniquely define
   $$u(t) \coloneq \lim_{k\to\infty}u(t_k)\quad \mathrm{in}\ (S,W_m).$$
   Hence we obtain $u : [0,T] \to S \subset \mathcal{P}(\Omega)$.

   Finally, we show
   $u_{\varepsilon_n}(t) \rightharpoonup u(t)$ weakly in $L^1\cap L^{p+1-\alpha}(\Omega)$ for $t \in [0,T]$.
   It is sufficient to prove that
   all subsequences of $\{u_{\varepsilon_n}(t)\}$ have 
   a subsequence converging to $u(t)$ weakly in $L^1\cap L^{p+1-\alpha}(\Omega)$.
   Fix $t \in [0,T]$ 
   and let $\{u_{\varepsilon_n^\prime}(t)\} \subset \{u_{\varepsilon_n}(t)\}$.
   Since $S$ is sequentially compact, taking a subsequence (not relabeled),
   we have 
   $$u_{\varepsilon_n^\prime}(t) \rightharpoonup \tilde{u} \quad 
   \mathrm{weakly\ in}\ L^1\cap L^{p+1-\alpha}(\Omega)\ \mathrm{as}\ \varepsilon_n^\prime\to0$$
   for some $\tilde{u} \in S$.
   For all $s \in \mathbb{Q}\cap[0,T]$, we obtain
   \begin{align*}
      W_m(\tilde{u}, u(t)) 
      &\leq W_m(\tilde{u}, u(s)) + W_m(u(s), u(t))\\
      &\leq \liminf_{\varepsilon_n^\prime\to0}W_{m_{\varepsilon_n^\prime}}(u_{\varepsilon_n^\prime}(t), u_{\varepsilon_n^\prime}(s)) + W_m(u(s), u(t))\\
      &\leq C_6\sqrt{|t-s|} + W_m(u(s), u(t)).
   \end{align*}
   Letting $s \to t$, we have 
   $W_m(\tilde{u}, u(t)) \leq 0$
   and since $\tilde{u}, u(t) \in S$, we see $\tilde{u} = u(t)$ in $S$.
\end{proof}

\begin{rem}\label{rem5.9}
   Let $T>0$.
   Since the estimates in Lemma \ref{lem5.2} are independent of $\varepsilon$, 
   by the same arguments for $\tau$, we can easily see 
   \begin{align*}
      &u_{\varepsilon_n} \rightharpoonup u\quad \mathrm{weakly\ in}\ L^2((0,T);W^{1,p+1-\alpha}(\Omega))\ \mathrm{as}\ n \to \infty,\\
      &u_{\varepsilon_n} \to u\quad \mathrm{strongly\ in}\ L^2((0,T);L^{p+1-\alpha}(\Omega))\ \mathrm{as}\ n \to \infty,\\
      &u_{\varepsilon_n}(x,t) \to u(x,t)\quad \mathrm{a.e.\ in}\ (x,t) \in \Omega\times(0,T)\ \mathrm{as}\ n\to\infty,\\
      &\nabla (u_{\varepsilon_n})^p \rightharpoonup \nabla u^p\quad \mathrm{weakly\ in}\ L^{\frac{p+1-\alpha}{p}}(\Omega\times(0,T))\ \mathrm{as}\ n\to\infty,\\
      &v_{\varepsilon_n} \rightharpoonup v\quad \mathrm{weakly\ in}\ L^2((0,T);H^2(\Omega))\ \mathrm{as}\ n\to\infty,\\
      &v_{\varepsilon_n}(t) \rightharpoonup v(t)\quad \mathrm{weakly\ in}\ H^1(\Omega)\ \mathrm{as}\ n\to\infty\ \mathrm{for}\ t\in [0,T],\\
      &v \in C^{\frac{1}{2}}([0,T];L^2(\Omega)).
   \end{align*}
\end{rem}

\quad Set $Q \coloneq \{(x,t) \in \Omega\times(0,T);u(x,t) = 0\}$.
Then the following lemma implies that $\nabla (u_\varepsilon)^p \to 0$ 
in $L^{\frac{p+1-\alpha}{p}}(Q)$
as $\varepsilon\to0$.
This idea is inspired by \cite[Lemma 5.6]{LMS}.

\begin{lem}\label{lem6.5}
   Let $T>0$,
   $1+\alpha-2/d\leq p\leq 1+\alpha$ and 
   assume that $\chi>0$ is small enough if $p=1+\alpha-2/d$.
   Then 
   $\|\nabla (u_{\varepsilon_n})^p\|_{L^{\frac{p+1-\alpha}{p}}(Q)} \to 0$ as $n\to\infty$.
\end{lem}

\begin{proof}
   By the same argument of the proof of Lemma \ref{lem5.6}, we have
   \begin{align*}
      \int_{Q} |\nabla (u_{\varepsilon_n})^p|^{\frac{p+1-\alpha}{p}}\, dx\, dt
      \leq \left(\frac{2p}{p+1-\alpha}\right)^{\frac{p+1-\alpha}{p}}(C_2(1+T))^{\frac{p+1-\alpha}{2p}}
      \left(\int_{Q} (u_{\varepsilon_n})^{p+1-\alpha}\, dx\, dt\right)^{\frac{p+\alpha-1}{2p}}.
   \end{align*}
   Since $u_{\varepsilon_n} \to u$ strongly in $L^2((0,T);L^{p+1-\alpha}(\Omega))$ as $n\to\infty$,
   in particular $u_{\varepsilon_n} \to u$ strongly in $L^{p+1-\alpha}(Q)$ as $n\to\infty$,
   it follows
   \begin{align*}
      &\limsup_{n\to\infty}\int_{Q} |\nabla (u_{\varepsilon_n})^p|^{\frac{p+1-\alpha}{p}}\, dx\, dt\\
      &\leq \left(\frac{2p}{p+1-\alpha}\right)^{\frac{p+1-\alpha}{p}}(C_2(1+T))^{\frac{p+1-\alpha}{2p}}
      \lim_{n\to\infty}\left(\int_{Q} (u_{\varepsilon_n})^{p+1-\alpha}\, dx\, dt\right)^{\frac{p+\alpha-1}{2p}}\\
      &= \left(\frac{2p}{p+1-\alpha}\right)^{\frac{p+1-\alpha}{p}}(C_2(1+T))^{\frac{p+1-\alpha}{2p}}
      \left(\int_{Q} u^{p+1-\alpha}\, dx\, dt\right)^{\frac{p+\alpha-1}{2p}} = 0.
   \end{align*}
   The proof is completed.
\end{proof}

Finally we prove Theorem \ref{thm1.1} and Theorem \ref{thm1.4}.

\begin{proof}[Proof of Theorem \ref{thm1.1} and Theorem \ref{thm1.4}]
   Let $T>0, \varphi \in C^\infty(\overline{\Omega})$ with $\nabla\varphi\cdot\boldsymbol{n} = 0$ on $\partial\Omega$
   and $\zeta \in H^1(\Omega)$.
   Note that \eqref{eq61} and \eqref{eq62},
   then by Lemma \ref{lem5.8} and Remark \ref{rem5.9}, we have
   \begin{align*}
      &\bullet u \in L^\infty((0,T);L^{p+1-\alpha}(\Omega)),\ u^{\frac{p+1-\alpha}{2}} \in L^2((0,T);H^1(\Omega)),\\
      &\bullet \|u(t)\|_{L^1(\Omega)} = 1\quad \mathrm{for}\ t \in [0,T],\\
      &\bullet v \in L^\infty((0,T);H^1(\Omega))\cap L^2((0,T);H^2(\Omega))\cap C^{\frac{1}{2}}([0,T];L^2(\Omega)),\\
      &\bullet \lim_{t\to0}W_m(u(t),u_0) = 0\ \mathrm{and}\ \lim_{t\to0}\|v(t) - v_0\|_{L^2(\Omega)} = 0.
   \end{align*}
   Then, we infer from Lemma \ref{lem6.2} and Remark \ref{rem5.9} that
   \begin{align*}
      \int_0^T\int_{\Omega} (\nabla v\cdot \nabla\zeta + v\zeta - u\zeta)\, dx\, dt 
      = \int_{\Omega} (v_0 - v(\cdot,T))\zeta\, dx.
   \end{align*}
   By \eqref{eq12}, we have
   \begin{align*}
      \int_{\Omega} (u_0(x) - u_{\varepsilon_n}(x,T))\varphi(x)\, dx
      &= - \int_0^T\int_{\Omega} \frac{\alpha}{p-\alpha}\left(\frac{u_{\varepsilon_n}(x,t)}{u_{\varepsilon_n}(x,t) + \varepsilon_n}\right)^{1-\alpha} \nabla u_{\varepsilon_n}(x,t)^p\cdot\nabla\varphi(x)\, dx\, dt\notag\\
      &\quad - \int_0^T\int_{\Omega} \frac{p}{p-\alpha}u_{\varepsilon_n}(x,t)^{p-\alpha} m_{\varepsilon_n}(u_{\varepsilon_n}(x,t))\Delta\varphi(x)\, dx\, dt\notag\\
      &\quad - \int_0^T\int_{\Omega} \chi m_\varepsilon(u_{\varepsilon_n}(x,t)) \nabla v_{\varepsilon_n}(x,t)\cdot\nabla\varphi(x)\, dx\, dt.
   \end{align*}
   By the convergences in Remark \ref{rem5.9} and the same argument in Lemma \ref{lem5.7}
   we immediately obtain
   \begin{align*}
      &\int_{\Omega} u_{\varepsilon_n}(x,T)\varphi(x)\, dx \to \int_{\Omega} u(x,T)\varphi(x)\, dx
      \quad \mathrm{as}\ n\to\infty,\\
      &\int_0^T\int_{\Omega} (u_{\varepsilon_n})^{p-\alpha}m_{\varepsilon_n}(u_{\varepsilon_n})\Delta\varphi\, dx\, dt
      \to \int_0^T\int_{\Omega} u^{p-\alpha}u^\alpha\Delta\varphi\, dx\, dt
      \quad \mathrm{as}\ n\to\infty,\\
      &\int_0^T\int_{\Omega} m_{\varepsilon_n}(u_{\varepsilon_n})\nabla v_{\varepsilon_n}\cdot\nabla\varphi\, dx\, dt
      \to \int_0^T\int_{\Omega} u^\alpha\nabla v\cdot\nabla\varphi\, dx\, dt
      \quad \mathrm{as}\ n\to\infty.
   \end{align*}
   We will show
   \begin{align*}
      \int_0^T\int_{\Omega} \left(\frac{u_{\varepsilon_n}}{u_{\varepsilon_n}+\varepsilon_n}\right)^{1-\alpha}\nabla (u_{\varepsilon_n})^p\cdot\nabla\varphi\, dx\, dt
      \to \int_0^T\int_{\Omega} \nabla u^p\cdot\nabla\varphi\, dx\, dt
      \quad \mathrm{as}\ n\to\infty.
   \end{align*}
   Indeed, we have
   \begin{align*}
      &\left|\int_0^T\int_{\Omega} \left(\frac{u_{\varepsilon_n}}{u_{\varepsilon_n}+\varepsilon_n}\right)^{1-\alpha}\nabla (u_{\varepsilon_n})^p\cdot\nabla\varphi\, dx\, dt
      - \int_0^T\int_{\Omega} \nabla u^p\cdot\nabla\varphi\, dx\, dt\right|\\
      &\leq \int_0^T\int_{\Omega} \left|\left(\frac{u_{\varepsilon_n}}{u_{\varepsilon_n}+\varepsilon_n}\right)^{1-\alpha} - 1\right||\nabla (u_{\varepsilon_n})^p||\nabla\varphi|\, dx\, dt
      + \left|\int_0^T\int_{\Omega} (\nabla (u_{\varepsilon_n})^p - \nabla u^p)\cdot\nabla\varphi\, dx\, dt\right|\\
      &\eqcolon \mathrm{I}_1 + \mathrm{I}_2.
   \end{align*}
   By Remark \ref{rem5.9}, it follows $\mathrm{I}_2 \to 0$ as $n\to\infty$.
   On the other hand, by H\"older's inequality and Lemma \ref{lem6.5}, it follows
   \begin{align*}
      \int_{Q} \left|\left(\frac{u_{\varepsilon_n}}{u_{\varepsilon_n}+\varepsilon_n}\right)^{1-\alpha} - 1\right||\nabla (u_{\varepsilon_n})^p||\nabla\varphi|\, dx\, dt
      &\leq 2T^{\frac{1-\alpha}{p+1-\alpha}}\|\nabla\varphi\|_{L^{\frac{p+1-\alpha}{1-\alpha}}(\Omega)}\|\nabla (u_{\varepsilon_n})^p\|_{L^{\frac{p+1-\alpha}{p}}(Q)}\\
      &\to 0\quad \mathrm{as}\ n \to\infty,
   \end{align*}
   where we used
   \begin{align}\label{eq110}
      \left|\left(\frac{u_{\varepsilon_n}}{u_{\varepsilon_n}+\varepsilon_n}\right)^{1-\alpha} - 1\right|
      \leq 2.
   \end{align}
   Moreover by H\"older's inequality and \eqref{eq92}, we obtain
   \begin{align*}
      &\int_{(\Omega\times(0,T))\setminus Q} 
      \left|\left(\frac{u_{\varepsilon_n}}{u_{\varepsilon_n}+\varepsilon_n}\right)^{1-\alpha} - 1\right||\nabla (u_{\varepsilon_n})^p||\nabla\varphi|\, dx\, dt\\
      &\leq \left(\int_{(\Omega\times(0,T))\setminus Q} 
      \left|\left(\frac{u_{\varepsilon_n}}{u_{\varepsilon_n}+\varepsilon_n}\right)^{1-\alpha} - 1\right|^{\frac{p+1-\alpha}{1-\alpha}}\, dx\, dt\right)^{\frac{1-\alpha}{p+1-\alpha}}
      \|\nabla \varphi\|_{L^\infty(\Omega)}\|\nabla(u_{\varepsilon_n})^p\|_{L^{\frac{p+1-\alpha}{p}}(\Omega\times(0,T))}\\
      &\leq \left(\int_{(\Omega\times(0,T))\setminus Q} 
      \left|\left(\frac{u_{\varepsilon_n}}{u_{\varepsilon_n}+\varepsilon_n}\right)^{1-\alpha} - 1\right|^{\frac{p+1-\alpha}{1-\alpha}}\, dx\, dt\right)^{\frac{1-\alpha}{p+1-\alpha}}
      \|\nabla \varphi\|_{L^\infty(\Omega)}C_7(1+T).
   \end{align*}
   Since $u_{\varepsilon_n}(x,t) \to u(x,t) > 0$ a.e. $(x,t) \in (\Omega\times(0,T))\setminus Q$ as $n\to\infty$,
   we have
   \begin{align*}
      &\left|\left(\frac{u_{\varepsilon_n}(x,t)}{u_{\varepsilon_n}(x,t)+\varepsilon_n}\right)^{1-\alpha} - 1\right|
      \leq \left(\frac{\varepsilon_n}{u_{\varepsilon_n}(x,t) + \varepsilon_n}\right)^{1-\alpha}
      \to 0\quad \mathrm{as}\ n\to\infty\\
      &\hspace{9cm} (x,t) \in (\Omega\times(0,T))\setminus Q.
   \end{align*}
   Combining this with \eqref{eq110},
   we infer from Lebesgue's dominated convergence theorem that  
   \begin{align*}
      \int_{(\Omega\times(0,T))\setminus Q} 
      \left|\left(\frac{u_{\varepsilon_n}}{u_{\varepsilon_n}+\varepsilon_n}\right)^{1-\alpha} - 1\right|^{\frac{p+1-\alpha}{1-\alpha}}\, dx\, dt
      \to 0\quad \mathrm{as}\ n\to\infty,
   \end{align*}
   which yields that $\mathrm{I}_1$ converges to $0$ as $n\to\infty$.
   Therefore we conclude that
   \begin{align*}
      &\int_{\Omega} (u_0(x) - u(x,T))\varphi(x)\, dx\\
      &= - \int_0^T\int_{\Omega} \frac{\alpha}{p-\alpha}\nabla u(x,t)^p\cdot\nabla\varphi(x)\, dx\, dt
       - \int_0^T\int_{\Omega} \frac{p}{p-\alpha}u(x,t)^p\Delta\varphi(x)\, dx\, dt\\
      &\quad - \int_0^T\int_{\Omega} \chi u(x,t)^\alpha \nabla v(x,t)\cdot\nabla\varphi(x)\, dx\, dt\\
      &= \int_0^T\int_{\Omega} \nabla u(x,t)^p\cdot\nabla\varphi\, dx\, dt
       - \int_0^T\int_{\Omega} \chi u(x,t)^\alpha \nabla v(x,t)\cdot\nabla\varphi(x)\, dx\, dt.
   \end{align*}
   The proof is completed.
\end{proof}

\appendix
\section{Appendix}

\begin{proof}[Proof of Proposition \ref{prop4.1}]
   We devide the proof into four steps.
   To simplify, we write $\|\cdot\|_{L^q(\Omega)} = \|\cdot\|_q$ for $q\in[1,\infty]$
   and $\|\cdot\|_{W^{l,p+1-\alpha}(\Omega)} = \|\cdot\|_{W^{l,p+1-\alpha}}$ for $l\in \mathbb{N}$.

   {\bf Step 1: Existence of a local soluiton.}
   
   Set $M_0\coloneq \|w_0\|_{W^{1,p+1-\alpha}}$ 
   and 
   $$Y \coloneq \{y \in C([0,T_0];W^{1,p+1-\alpha}(\Omega));
   \|y\|_Y \leq 4M_0\},$$
   where $\|y\|_Y \coloneq \sup_{0\leq t\leq T_0}\|y(t)\|_{W^{1,p+1-\alpha}}$ and
   $T_0 \in (0,\infty)$ will be fixed later. 
   We define a function $w_1$ by $w_1 = e^{\delta t\Delta}w_0$, 
   where $e^{\delta t\Delta}$ is the Neumann heat semigroup, that is, $w_1$ is a solution to 
   \begin{align*}
      \begin{cases}
         \dt w_1 = \delta\Delta w_1\quad &\mathrm{in}\ \Omega\times(0,\infty),\\
         \nabla w_1\cdot \boldsymbol{n} = 0\quad &\mathrm{on}\ \partial\Omega\times(0,\infty),\\
         w_1(0) = w_0\quad &\mathrm{in}\ L^1\cap L^2\cap W^{1,p+1-\alpha}(\Omega).
      \end{cases}
   \end{align*}
   Then $w_1$ is nonnegative and 
   $w_1$ belongs to $C([0,\infty);L^1\cap L^2\cap W^{1,p+1-\alpha}(\Omega))\cap C^\infty(\Omega\times(0,\infty))$.
   We also define a function 
   \begin{align*}
      &\Phi[w](t) 
      \coloneq e^{\delta t\Delta}w_0 
      + \int_0^t e^{\delta(t-s)\Delta}\left(\frac{\alpha\nabla w(s)}{(w_1(s)+\varepsilon)^{1-\alpha}}\cdot\nabla\varphi 
      + (w_1(s)+\varepsilon)^\alpha\Delta\varphi\right)\, ds\\
      &\hspace{10cm} \mathrm{for}\ w \in Y,\ t \in [0,T_0].
   \end{align*}
   Then $\Phi$ belongs to $C([0,T_0];W^{1,p+1-\alpha}(\Omega))$ 
   due to the property of the heat semigroup.

   In this proof, we often use the following estimate:
   \begin{align}\label{eq44}
      \|(y+\varepsilon)^\alpha\Delta\varphi\|_{p+1-\alpha}
      &\leq \|(y^\alpha + \varepsilon^\alpha)\Delta\varphi\|_{p+1-\alpha}
      \leq \|y\|_{p+1-\alpha}^\alpha\|\Delta\varphi\|_{\frac{p+1-\alpha}{1-\alpha}} + \varepsilon^\alpha\|\Delta\varphi\|_{p+1-\alpha}\notag\\
      &\leq C(\|y\|_{p+1-\alpha}^\alpha + 1)\quad \mathrm{for}\ y \in Y\ \mathrm{and}\ y \geq 0\ \mathrm{a.e.},
   \end{align}
   where 
   $C \coloneq \max\{\|\Delta\varphi\|_{\frac{p+1-\alpha}{1-\alpha}}, \varepsilon^\alpha\|\Delta\varphi\|_{p+1-\alpha}\}$ 
   is a constant.

   First, we show that $\Phi$ is a contraction map on $Y$ 
   if $T_0$ is small enough.
   Let $w \in Y$ and $t \in [0,T_0]$ 
   then using $L^p$-$L^q$ estimates and \eqref{eq44}, we have
   \begin{align*}
      \|\Phi[w](t)\|_{p+1-\alpha}
      &\leq \|w_0\|_{p+1-\alpha} 
      + \int_0^t \left\|\frac{\alpha\nabla w(s)}{(w_1(s)+\varepsilon)^{1-\alpha}}\cdot\nabla\varphi + (w_1(s)+\varepsilon)^\alpha\Delta\varphi\right\|_{p+1-\alpha}\, ds\\
      &\leq M_0 
      + \int_0^t \frac{C^\prime\alpha\|\nabla\varphi\|_\infty}{\varepsilon^{1-\alpha}}\|\nabla w(s)\|_{p+1-\alpha} + C(\|w_1(s)\|_{p+1-\alpha}^\alpha + 1)\, ds\\
      &\leq M_0 
      + C_1(M_0 + 1)T_0,
   \end{align*}
   where $C^\prime$ is a constant by $L^p$-$L^q$ estimates and 
   $C_1(\alpha, \varepsilon, \varphi,\Omega)$ is also a constant.
   If $T_0 \leq M_0/C_1(M_0+1)$ 
   then $\sup_{0\leq t \leq T_0}\|\Phi[w](t)\|_{p+1-\alpha} \leq 2M_0$.
   Similarly, it follows
   \begin{align*}
      &\|\nabla\Phi[w](t)\|_{p+1-\alpha}\\
      &\leq \|\nabla w_0\|_{p+1-\alpha} 
      + \int_0^t \frac{\tilde{C}}{(t-s)^{\frac{1}{2}}}\left\|\frac{\alpha\nabla w(s)}{(w_1(s)+\varepsilon)^{1-\alpha}}\cdot\nabla\varphi + (w_1(s)+\varepsilon)^\alpha\Delta\varphi\right\|_{p+1-\alpha}\, ds\\
      &\leq M_0 + C_2(M_0 + 1)T_0^{\frac{1}{2}},
   \end{align*}
   where $\tilde{C}$ is a constant by $L^p$-$L^q$ estimates and $C_2 = C_2(\alpha,\varepsilon,\varphi,\Omega)$ is also a constant.
   If $T_0 \leq M_0^2/C_2^2(M_0+1)^2$ 
   then $\sup_{0\leq t \leq T_0}\|\nabla\Phi[w](t)\|_{p+1-\alpha} \leq 2M_0$.
   Thus we have $\|\Phi[w]\|_Y \leq 4M_0$ and $\Phi[w] \in Y$.

   Let $w,y \in Y$ then we infer from $L^p$-$L^q$ estimates that
   \begin{align*}
      \|\Phi[w](t) - \Phi[y](t)\|_{W^{1,p+1-\alpha}}
      &\leq \int_0^t (C^\prime + \tilde{C}(t-s)^{-\frac{1}{2}})\frac{\alpha\|\nabla\varphi\|_\infty}{\varepsilon^{1-\alpha}}\|\nabla w(s) - \nabla y(s)\|_{p+1-\alpha}\, ds\\
      &\leq C_3(T_0 + T_0^\frac{1}{2})\|w - y\|_Y,
   \end{align*}
   where $C_3 = C_3(\alpha,\varepsilon,\varphi,\Omega)$ is a constant.
   If $T_0 \leq \min\{1/4C_3,1/16C_3^2\}$ 
   then $\|\Phi[w]-\Phi[y]\|_Y \leq 1/2\|w-y\|_Y$. 
   Therefore choosing 
   $$T_0 \leq \min\left\{\frac{M_0}{C_1(M_0+1)}, \frac{M_0^2}{C_2^2(M_0+1)^2}, \frac{1}{4C_3},\frac{1}{16C_3^2}\right\},$$
   we see that $\Phi$ is a contraction map on $Y$.
   By Banach's fixed point theorem, there is a function $w_2 \in Y$ 
   such that $\Phi[w_2]=w_2$.
   In other words, we obtain 
   \begin{equation}
      w_2(t) 
      = e^{\delta t\Delta}v_0 
      + \int_0^t e^{\delta(t-s)\Delta}\left(\frac{\alpha\nabla w_2(s)}{(w_1(s)+\varepsilon)^{1-\alpha}}\cdot\nabla\varphi + (w_1(s)+\varepsilon)^\alpha\Delta\varphi\right)\, ds.
   \end{equation} 

   {\bf Step 2: Regularity and nonnegativity.}

   Set 
   $$f(t) 
   \coloneq \frac{\alpha\nabla w_2(t)}{(w_1(t)+\varepsilon)^{1-\alpha}}\cdot\nabla\varphi 
   + (w_1(t)+\varepsilon)^{\alpha}\Delta\varphi\quad \mathrm{for}\ t \in [0,T_0],$$
   then we have $f \in L^\infty((0,T_0);L^{p+1-\alpha}(\Omega))$  
   because $w_1, w_2 \in C([0,T_0];W^{1,p+1-\alpha}(\Omega))$ and
   $\varphi$ belongs to $C^\infty(\bar{\Omega})$.
   Moreover, setting 
   $$F(t) \coloneq \int_0^t e^{\delta(t-s)\Delta}f(s)\, ds\quad \mathrm{for}\ t \in [0,T_0],$$
   by \cite[Lemma 7.1.1]{Lun}, 
   we have $F \in C^{\frac{1}{2}}([0,T_0];W^{1,p+1-\alpha}(\Omega))$.
   Since $e^{\delta t\Delta}w_0 \in C^\infty(\Omega\times(0,\infty))$, 
   we also obtain $w_2 \in C^{\frac{1}{2}}((0,T_0];W^{1,p+1-\alpha}(\Omega))$.
   Then it holds $f \in C^{\frac{1}{2}}((0,T_0];L^{p+1-\alpha}(\Omega))$.
   Indeed, for $t,s \in (0,T_0]$, we have
   \begin{align*}
      &\|f(t)-f(s)\|_{p+1-\alpha}\\
      &\leq \left\|\frac{\alpha\nabla w_2(t)\cdot\nabla\varphi}{(w_1(t)+\varepsilon)^{1-\alpha}} - \frac{\alpha\nabla w_2(s)\cdot\nabla\varphi}{(w_1(s)+\varepsilon)^{1-\alpha}}\right\|_{p+1-\alpha}
      + \left\|\left[(w_1(t)+\varepsilon)^\alpha - (w_1(s)+\varepsilon)^\alpha\right]\Delta\varphi\right\|_{p+1-\alpha}\\
      &\leq \frac{\alpha(1-\alpha)}{\varepsilon^{2-\alpha}}\|\nabla\varphi\|_\infty\|w_2(t)- w_2(s)\|_{W^{1,p+1-\alpha}}
      +\frac{\alpha}{\varepsilon^{1-\alpha}}\|\Delta\varphi\|_\infty\|w_1(t)-w_1(s)\|_{p+1-\alpha},
   \end{align*}
   where we used the mean value theorem as follows
   \begin{align*}
      &\left|\frac{\nabla w_2(t)}{(w_1(t)+\varepsilon)^{1-\alpha}} - \frac{\nabla w_2(s)}{(w_1(s)+\varepsilon)^{1-\alpha}}\right|
      \leq \frac{1-\alpha}{\varepsilon^{2-\alpha}}\left(|\nabla w_2(t) - \nabla w_2(s)| + |w_1(t) - w_1(s)|\right),\\
      &|(w_1(t)+\varepsilon)^\alpha - (w_1(s)+\varepsilon)^\alpha|
      \leq \frac{\alpha}{\varepsilon^{1-\alpha}}|w_1(t) - w_1(s)|.
   \end{align*}
   Here, since $w_1, w_2 \in C^{\frac{1}{2}}((0,T_0];W^{1,p+1-\alpha}(\Omega))$
   we have $f \in C^{\frac{1}{2}}((0,T_0];L^{p+1-\alpha}(\Omega))$.

   By \cite[Theorem 4.3.4]{Lun}, 
   we obtain $w_2 \in C((0,T_0];W^{2,p+1-\alpha}(\Omega))\cap C^1((0,T_0];L^{p+1-\alpha}(\Omega))$ and $w_2$ satisfies
   \begin{align}\label{pe}
      \begin{dcases}
         \dt w_2 = \delta\Delta w_2 + \frac{\alpha\nabla w_2}{(w_1 + \varepsilon)^{1-\alpha}}\cdot\nabla\varphi + (w_1 + \varepsilon)^\alpha\Delta\varphi 
         &\mathrm{a.e.\ in}\ \Omega\times(0,T_0),\\
         (\nabla w_2\cdot\boldsymbol{n})|_{\partial\Omega} = 0 &\mathrm{for}\ t \in (0,T_0],\\
         w_2(0) = w_0 &\mathrm{in}\ W^{1,p+1-\alpha}(\Omega),
      \end{dcases}
   \end{align}
   where $(\nabla w_2\cdot\boldsymbol{n})|_{\partial\Omega}$ is defined by the trace operator on $W^{2,p+1-\alpha}(\Omega)$
   and the Neumann boundary condition is satisfies because of the definition of $w_2$ and 
   the property of the Neumann heat semigroup.

   We will show $w_2 \geq 0$ a.e. in $\Omega\times[0,T_0]$.
   Let $w_2^- \coloneq \min\{w_2, 0\}$ and $t \in (0,T_0]$.
   Then we choose arbitary $a\in(0,t)$ and fix it.
   Multiplying the first equation of \eqref{pe} by $w_2^-$ 
   and integrating in $\Omega$, we have
   \begin{align*}
      \frac{1}{2}\frac{d}{dt}\|w_2^-(t)\|_2^2
      = -\delta\|\nabla w_2^-(t)\|_2^2 
      + \int_{\Omega} \frac{\alpha\nabla w_2^-(t)w_2^-(t)}{(w_1(t)+\varepsilon)^{1-\alpha}}\cdot\nabla\varphi\, dx 
      + \int_{\Omega} (w_1(t)+\varepsilon)^\alpha w_2^-(t)\Delta\varphi\, dx,
   \end{align*}
   where we used integration by parts and 
   the condition $(\nabla w_2\cdot\boldsymbol{n})|_{\partial\Omega} = 0$ for $t \in (0,T_0]$. 
   Note that thanks to $1+\alpha-2/d \leq p \leq 1+\alpha$ and the Sobolev embedding theorem,
   it holds that $W^{2,p+1-\alpha}(\Omega) \hookrightarrow H^1(\Omega) \hookrightarrow L^{\frac{p+1-\alpha}{p-\alpha}}(\Omega)$,
   thus the right hand side is well-defined.
   Since $w_1 \geq0$, 
   it follows from H\"older's inequality and \eqref{eq44} that
   \begin{align*}
      &\frac{1}{2}\frac{d}{dt}\|w_2^-(t)\|_2^2\\
      &\leq -\delta\|\nabla w_2^-(t)\|_2^2
      + \frac{\alpha\|\nabla\varphi\|_\infty}{\varepsilon^{1-\alpha}}\|\nabla w_2^-(t)\|_2\|w_2^-(t)\|_2
      + \|(w_1(t)+\varepsilon)^\alpha\Delta\varphi\|_2\|w_2^-(t)\|_2\\
      &\leq -\delta\|\nabla w_2^-(t)\|_2^2
      + \frac{\alpha\|\nabla\varphi\|_\infty}{\varepsilon^{1-\alpha}}\|\nabla w_2^-(t)\|_2\|w_2^-(t)\|_2
      + C(\|w_1(t)\|_2^\alpha + 1)\|w_2^-(t)\|_2.
   \end{align*}
   Moreover by $\|w_1(t)\|_2 \leq M_0$ and Young's inequality, we have
   \begin{align*}
      \left(\frac{\alpha\|\nabla\varphi\|_\infty}{\varepsilon^{1-\alpha}}\|\nabla w_2^-(t)\|_2 + C(M_0^\alpha + 1)\right)\|w_2^-(t)\|_2
      &\leq C^\prime(\|\nabla w_2^-(t)\|_2 + 1)\|w_2^-(t)\|_2\\
      &\leq \frac{\theta}{2}(\|\nabla w_2^-(t)\|_2 + 1)^2 
      + \frac{(C^\prime)^2}{2\theta}\|w_2^-(t)\|_2^2\\
      &\leq \theta(\|\nabla w_2^-(t)\|_2^2 + 1) 
      + \frac{(C^\prime)^2}{2\theta}\|w_2^-(t)\|_2^2,
   \end{align*}
   where 
   $$C^\prime \coloneq 
   \max\left\{\frac{\alpha\|\nabla\varphi\|_\infty}{\varepsilon^{1-\alpha}},C(M_0^\alpha + 1)\right\}$$
   and $\theta >0$ satisfies 
   $$\theta < \min_{t\in[a,T_0]}\frac{\delta\|\nabla w_2^-(t)\|_2^2}{\|\nabla w_2^-(t)\|_2^2 + 1}.$$
   Note that if $\|\nabla w_2^-(t)\|_2 = 0$ 
   then $w_2^-(t) = 0$ a.e. in $\Omega$ obviously.
   Hence we have
   $$\frac{1}{2}\frac{d}{dt}\|w_2^-(t)\|_2^2 
   \leq \frac{(C^\prime)^2}{2\theta}\|w_2^-(t)\|_2^2.$$
   By Gronwall's lemma and $w_0 \geq 0$ a.e. in $\Omega\times[0,T_0]$, it follows
   $$\|w_2^-(t)\|_2 = 0\quad \mathrm{for}\ t \in [a,T_0],$$
   then
   $$w_2^-(t) = 0\quad \mathrm{a.e.\ in}\ \Omega,\ \mathrm{for}\ t \in [a,T_0].$$
   Since $a>0$ is arbitary, 
   we see $w_2 \geq 0$ a.e. in $\Omega\times(0,T_0]$, 
   that is, $w_2 \geq 0$ a.e. in $\Omega\times[0,T_0]$.

   {\bf Step 3: Convergences and properties.}

   Repeating the above process, 
   we can construct a sequence $\{w_k\}$ such that
   $$w_k \in Y\cap C((0,T_0];W^{2,p+1-\alpha}(\Omega))\cap C^1((0,T_0];L^{p+1-\alpha}(\Omega)),$$
   $$w_k \geq 0\ \mathrm{a.e.\ in}\ \Omega\times[0,T_0].$$
   We will show that $\{w_k\}_{k\geq2}$ is a Cauchy sequence in $C([0,T_0];W^{1,p+1-\alpha}(\Omega))$ 
   if $T_0$ is small enough.
   First, for $t \in [0,T_0]$ by $L^p$-$L^q$ estimates and the mean value theorem,
   we have
   \begin{align*}
      &\|w_{k+1}(t)-w_k(t)\|_{p+1-\alpha}\\
      &\leq \int_0^t \left\|e^{\delta(t-s)\Delta}
      \left\{\alpha\left(\frac{\nabla w_{k+1}(s)}{(w_k(s)+\varepsilon)^{1-\alpha}} - \frac{\nabla w_k(s)}{(w_{k-1}(s)+\varepsilon)^{1-\alpha}}\right)\cdot\nabla\varphi\right\}\right\|_{p+1-\alpha}\, ds\\
      &\quad + \int_0^t \left\|e^{\delta(t-s)\Delta}\left\{((w_k(s)+\varepsilon)^\alpha - (w_{k-1}(s)+\varepsilon)^\alpha)\Delta\varphi\right\}\right\|_{p+1-\alpha}\, ds\\
      &\leq \int_0^t C^\prime\alpha\|\nabla\varphi\|_\infty\left\|\frac{\nabla w_{k+1}(s)}{(w_k(s)+\varepsilon)^{1-\alpha}} - \frac{\nabla w_k(s)}{(w_{k-1}(s)+\varepsilon)^{1-\alpha}}\right\|_{p+1-\alpha}\, ds\\
      &\quad + \int_0^t C^\prime\|\Delta\varphi\|_\infty\|(w_k(s)+\varepsilon)^\alpha - (w_{k-1}(s)+\varepsilon)^\alpha\|_{p+1-\alpha}\, ds\\
      &\leq \int_0^t C_4\left(\|\nabla w_{k+1}(s) - \nabla w_k(s)\|_{p+1-\alpha} + \|w_k(s) - w_{k-1}(s)\|_{p+1-\alpha}\right)\, ds\\
      &\leq C_4T_0\left(\sup_{0\leq s\leq T_0}\|w_{k+1}(s) - w_k(s)\|_{W^{1,p+1-\alpha}} + \sup_{0\leq s\leq T_0}\|w_k(s) - w_{k-1}(s)\|_{W^{1,p+1-\alpha}}\right),
   \end{align*}
   where $C_4 = C_4(\alpha,\varepsilon,\varphi,\Omega)$ is a constant.
   Similarly, it follows
   \begin{align*}
      &\|\nabla w_{k+1}(t) - \nabla w_k(t)\|_{p+1-\alpha}\\
      &\leq \int_0^t \left\|\nabla e^{\delta(t-s)\Delta}\left\{\alpha\left(\frac{\nabla w_{k+1}(s)}{(w_k(s)+\varepsilon)^{1-\alpha}} - \frac{\nabla w_k(s)}{(w_{k-1}(s)+\varepsilon)^{1-\alpha}}\right)\nabla\varphi\right\}\right\|_{p+1-\alpha}\, ds\\
      &\quad + \int_0^t \left\|\nabla e^{\delta(t-s)\Delta}\left\{((w_k(s)+\varepsilon)^\alpha - (w_{k-1}(s)+\varepsilon)^\alpha)\Delta\varphi\right\}\right\|_{p+1-\alpha}\, ds\\
      &\leq C_5T_0^{\frac{1}{2}}\left(\sup_{0\leq s\leq T_0}\|w_{k+1}(s) - w_k(s)\|_{W^{1,p+1-\alpha}} + \sup_{0\leq s\leq T_0}\|w_k(s) - w_{k-1}(s)\|_{W^{1,p+1-\alpha}}\right),
   \end{align*}
   where $C_5 = C_5(\alpha,\varepsilon,\varphi,\Omega)$ is a constant.
   Hence, choosing $T_0$ such that
   $$T_0 < \min\left\{\frac{1}{8C_4}, \frac{1}{64C_5^2}\right\},$$
   we have
   \begin{equation*}
      \sup_{0\leq t\leq T_0}\|w_{k+1}(t) - w_k(t)\|_{W^{1,p+1-\alpha}} 
      \leq \frac{1}{2}\sup_{0\leq t\leq T_0}\|w_k(t) - w_{k-1}(t)\|_{W^{1,p+1-\alpha}}.
   \end{equation*}
   Thus, for $m, n \in \mathbb{N}$ with $n>m$, it follows
   \begin{align*}
      \sup_{0\leq t\leq T_0}\|w_n(t) - w_m(t)\|_{W^{1,p+1-\alpha}} 
      &\leq \sum_{k=m}^{n-1} \sup_{0\leq t\leq T_0}\|w_{k+1}(t) - w_k(t)\|_{W^{1,p+1-\alpha}}\\
      &\leq \sum_{k=m}^{n-1} \left(\frac{1}{2}\right)^k \left(\sup_{0\leq t\leq T_0}\|w_1(t)\|_{W^{1,p+1-\alpha}}+\|w_0\|_{W^{1,p+1-\alpha}}\right)\\
      &\leq \left(\frac{1}{2}\right)^{m-1} 5M_0\ \rightarrow 0\quad \mathrm{as}\ n,m \rightarrow \infty.
   \end{align*}
   Since $(C([0,T_0];W^{1,p+1-\alpha}(\Omega)),\sup_{0\leq t\leq T_0}\|\cdot\|_{W^{1,p+1-\alpha}})$ is complete, 
   there exists a function $w \in C([0,T_0];W^{1,p+1-\alpha}(\Omega))$ 
   such that $w_k \rightarrow w$ in $C([0,T_0];W^{1,p+1-\alpha}(\Omega))$.
   In addition, we see $w \geq 0$ a.e. in $\Omega\times[0,T_0]$.
   Indeed, for all $\psi \in C_c^\infty(\Omega\times[0,T_0])$ with $\psi\geq0$, 
   we have
   \begin{align*}
      &\left|\int_0^{T_0}\int_{\Omega} w_k(x,t)\psi(x,t)\, dx\, dt - \int_0^{T_0}\int_{\Omega} w(x,t)\psi(x,t)\, dx\, dt\right|\\
      &\leq \int_0^{T_0} \|w_k(t) - w(t)\|_{p+1-\alpha}\|\psi(t)\|_{\frac{p+1-\alpha}{p-\alpha}}\, dt\\
      &\leq \int_0^{T_0} \|\psi(t)\|_{\frac{p+1-\alpha}{p-\alpha}}\, dt\sup_{0\leq t\leq T_0}\|w_k(t) - w(t)\|_{p+1-\alpha}
      \to 0\quad \mathrm{as}\ k\to\infty.
   \end{align*}
   Hence we obtain 
   \begin{align*}
      &0\leq \int_0^{T_0}\int_{\Omega} w_k(x,t)\psi(x,t)\, dx\, dt \to 
      \int_0^{T_0}\int_{\Omega} w(x,t)\psi(x,t)\, dx\, dt\quad \mathrm{as}\ k\to\infty\\
      &\hspace{7cm} \forall \psi \in C_c^\infty(\Rd\times[0,T_0])\ \mathrm{with}\ \psi \geq0,
   \end{align*}
   then
   $$w(x,t) \geq 0\ \mathrm{a.e.\ in}\ \Omega\times[0,T_0].$$

   Next we will prove that $w$ satisfies
   \begin{align*}
      w(t) 
      &= e^{\delta t\Delta}w_0 
      + \int_0^t e^{\delta(t-s)\Delta}\left(\frac{\alpha\nabla w(s)}{(w(s)+\varepsilon)^{1-\alpha}}\nabla\varphi 
      + (w(s)+\varepsilon)^\alpha\Delta\varphi\right)\, ds\\
      &= e^{\delta t\Delta}w_0 
      + \int_0^t e^{\delta(t-s)\Delta}\{\nabla\cdot((w(s)+\varepsilon)^\alpha\nabla\varphi)\}\, ds.
   \end{align*}
   By setting
   $$W(t) 
   \coloneq e^{\delta t\Delta}w_0 
   + \int_0^t e^{\delta(t-s)\Delta}\left(\frac{\alpha\nabla w(s)}{(w(s)+\varepsilon)^{1-\alpha}}\nabla\varphi 
   + (w(s)+\varepsilon)^\alpha\Delta\varphi\right)\, ds
   \quad t \in [0,T_0],$$
   it follows from $L^p$-$L^q$ estimates and the mean value theorem that
   \begin{align*}
      &\|W(t) - w_k(t)\|_{W^{1,p+1-\alpha}}\\
      &\leq \int_0^t (C^\prime + \tilde{C}(t-s)^{-\frac{1}{2}})
      \left\|\frac{\nabla w(s)}{(w(s)+\varepsilon)^{1-\alpha}} - \frac{\nabla w_k(s)}{(w_{k-1}(s)+\varepsilon)^{1-\alpha}}\right\|_{p+1-\alpha}\, ds\\
      &\quad + \int_0^t (C^\prime + \tilde{C}(t-s)^{-\frac{1}{2}})\|w(s) - w_{k-1}(s)\|_{p+1-\alpha}\, ds\\
      &\leq C_6(T_0 + T_0^{\frac{1}{2}})
      \left(\sup_{0\leq s\leq T_0}\|\nabla w(s) - \nabla w_k(s)\|_{p+1-\alpha} + \sup_{0\leq s\leq T_0}\|w(s) - w_{k-1}(s)\|_{p+1-\alpha}\right),
   \end{align*}
   where $C_6 = C_6(\alpha,\varepsilon,\varphi,\Omega)$ is a constant.
   Hence we have $w_k \rightarrow W$ in $C([0,T_0];W^{1,p+1-\alpha}(\Omega))$.
   Due to the uniquness of limit, it follows $w = W$ in $C([0,T_0];W^{1,p+1-\alpha}(\Omega))$.
   Moreover, we see $w \in C([0,T_0];L^1\cap L^2(\Omega))$.
   Indeed, for $t,s \in [0,T_0]$, 
   by H\"older's inequality, we have
   \begin{align*}
      \|w(t) - w(s)\|_1
      \leq |\Omega|^{\frac{p+1-\alpha}{p-\alpha}}\|w(t) - w(s)\|_{p+1-\alpha}.
   \end{align*}
   On the other hand, by the Sobolev embedding theorem, it follows
   \begin{align*}
      \|w(t) - w(s)\|_2
      \leq C_s\|w(t) - w(s)\|_{W^{1,p+1-\alpha}},
   \end{align*}
   where $C_s$ is a constant.
   Since $w \in C([0,T_0];W^{1,p+1-\alpha}(\Omega))$, 
   we obtain $w \in C([0,T_0];L^1\cap L^2(\Omega))$.
   Adapting the same argument for $w_2$ in Step 2, 
   we have $w \in C((0,T_0];W^{2,p+1-\alpha}(\Omega))\cap C^1((0,T_0];L^{p+1-\alpha}(\Omega))$ and $w$ satisfies
   \begin{align}\label{pev}
      \begin{cases}
         \dt w = \delta\Delta w + \nabla\cdot((w+\varepsilon)^\alpha\nabla\varphi) &\mathrm{a.e.\ in}\ \Omega\times(0,T_0),\\
         (\nabla w\cdot\boldsymbol{n})|_{\partial\Omega} = 0 &\mathrm{for}\ t \in (0,T_0],\\
         w(0) = w_0 &\mathrm{in}\ L^1\cap L^2\cap W^{1,p+1-\alpha}(\Omega).
      \end{cases}
   \end{align}
   Now, we will show $\|w(t)\|_1 = \|w_0\|$ for $t \in [0,T_0]$.
   Let $t \in [0,T_0]$. 
   Integrating the first equation of \eqref{pev} in $\Omega$,
   we have
   \begin{align*}
      \frac{d}{dt}\int_{\Omega} w(t)\, dx 
      &= \delta\int_{\Omega} \Delta w(t)\, dx 
      + \int_{\Omega} \nabla\cdot((w(t)+\varepsilon)^\alpha\nabla\varphi)\, dx.
   \end{align*}
   Since $\nabla w\cdot\boldsymbol{n} = 0$ and $\nabla\varphi\cdot\boldsymbol{n} = 0$ on $\partial\Omega\times(0,T_0]$,
   we infer from integration by parts that
   \begin{align*}
      \frac{d}{dt}\int_{\Omega} w(t)\, dx 
      &= \delta\int_{\partial\Omega} \nabla w(t)\cdot\boldsymbol{n}\, dS 
      + \int_{\Omega} (w(t)+\varepsilon)^\alpha\nabla\varphi\cdot\boldsymbol{n}\, dS
      =0.
   \end{align*}
   By integrating over $[0,t]$, it follows
   \begin{align*}
      \int_{\Omega} w(t)\, dx = \int_{\Omega} w(0)\, dx = \int_{\Omega} w_0\, dx.
   \end{align*}

   {\bf Step 4: Uniqueness.}

   Let $y_1, y_2 \in C([0,T_0];W^{1,p+1-\alpha}(\Omega))\cap C((0,T_0];W^{2,p+1-\alpha}(\Omega))\cap C^1((0,T_0];L^{p+1-\alpha}(\Omega))$ be solutions to \eqref{pev}.
   Then by \cite[Proposition 4.1.2]{Lun}, they are mild solutions to \eqref{pev}:
   \begin{align*}
      &y_1(t) 
      = e^{\delta t\Delta}w_0 
      + \int_0^t e^{\delta(t-s)\Delta}\{\nabla\cdot((y_1(s) + \varepsilon)^\alpha\nabla\varphi)\}\, ds,\\
      &y_2(t)
      = e^{\delta t\Delta}w_0 
      + \int_0^t e^{\delta(t-s)\Delta}\{\nabla\cdot((y_2(s) + \varepsilon)^\alpha\nabla\varphi)\}\, ds.
   \end{align*}
   Define
   \begin{equation*}
      t_0 \coloneq \max\{t \in [0,T_0]; y_1(s) = y_2(s)\ \mathrm{for}\ 0\leq s \leq t\},
   \end{equation*}
   and set $y_0 \coloneq y_1(t_0) = y_2(t_0)$.
   If $t_0 < T_0$, the problem
   \begin{equation}
      \dt w(t) = \delta\Delta w(t) + \nabla\cdot((w(t)+\varepsilon)^\alpha\nabla\varphi),\ t>t_0,\
      w(t_0) = y_0,
   \end{equation}
   has a unique mild solution in a set 
   \begin{equation*}
      Y' = \left\{y \in C([t_0,t_0 + a];W^{1,p+1-\alpha}(\Omega));
      \sup_{t_0\leq t\leq t_0+a}\|y(t)\|_{W^{1,p+1-\alpha}} \leq R\right\},
   \end{equation*}
   provided $R$ is large enough and $a$ is small enough.
   Since $y_1$ and $y_2$ are bounded with value in $W^{1,p+1-\alpha}(\Omega)$, 
   there exists $R$ such that $\|y_i(t)\|_{W^{1,p+1-\alpha}} \leq R$ for $t_0\leq t \leq T_0, i=1,2$.
   Thus it follows $y_1 = y_2$ in $Y^\prime$.
   On the other hand, $y_1$ and $y_2$ are two different solutions of \eqref{pev} in $[t_0,t_0+a]$, 
   for every $a \in (0,T_0-t_0]$.
   This is a contradiction.
   Hence $t_0 = T_0$, and the solution of \eqref{pev} is unique in 
   $C([0,T_0];W^{1,p+1-\alpha}(\Omega))\cap C((0,T_0];W^{2,p+1-\alpha}(\Omega))\cap C^1((0,T_0];L^{p+1-\alpha}(\Omega))$.
   The proof is completed.
\end{proof}

\subsection*{Acknowledgements}
\quad The author is greatly indebted to Professor Kentaro Fujie 
for his helpful comments during the preparation of the paper.



\end{document}